\documentclass[bibliography=totocnumbered,11pt]{article}
\usepackage{bbm}
\usepackage{amssymb,amsthm,amsmath}   
\usepackage[nottoc,numbib]{tocbibind}
\usepackage[font=small,labelfont=bf,justification=centerlast]{caption}
\usepackage{hyperref}
\usepackage{varioref}
\usepackage[utf8]{inputenc}
\usepackage[T1]{fontenc}
\usepackage{color}
\usepackage{graphics}
\usepackage[dvips]{epsfig}
\usepackage{psfrag}\graphicspath{ {./illustrations/} }
\usepackage{float}
\usepackage{amsthm}
\usepackage{pinlabel}
\usepackage{textcomp}
\usepackage{faktor}
\usepackage{array}
\usepackage{subcaption}
\usepackage{tikz-cd}
\usepackage{enumerate}
\usepackage[⟨options⟩]{overpic}

\usepackage[all]{xy}

\theoremstyle{definition}
\newtheorem{definition}{Definition}[section]
\newtheorem{example}{Example}[section]
\newtheorem{remark}{Remark}[section]

\theoremstyle{plain}
\newtheorem{theorem}{Theorem}[section]
\newtheorem*{theorem*}{Theorem}

\newtheorem{cor}{Corollary}[section]
\newtheorem{prop}{Proposition}[section]
\newtheorem{lemma}{Lemma}[section]

\AtEndDocument{\bigskip{\footnotesize%
  \textsc{Salammbo Connolly, Universit\'e Paris-Saclay, CNRS, Laboratoire de mathématiques d'Orsay, 91405, Orsay, France.} \par
  \textit{E-mail address}: \texttt{salammbo.connolly@universite-paris-saclay.fr} \par
}}

\title{Continuation maps for the Morse fundamental group}
\author{Salammbo Connolly}

\begin{document}
\maketitle

\begin{abstract}
We study properties of the continuation map for the Morse fundamental group $\pi_1^\text{Morse}(f,\ast)$ associated to a Morse-Smale pair $(f,g)$ on a manifold $M$. We get a morphism between $\pi_1^\text{Morse}(f_0,\ast_0)$ and $\pi_1^\text{Morse}(f_1,\ast_1)$  and show that it is functorial. We also define the morphism in the case of Morse data over different manifolds, thanks to the use of grafted trajectories. Finally, given an interpolation function on $M\times\mathbb{R}$ between two Morse functions (used for example to define the continuation map), we study the Morse fundamental group associated to that function and show that it is isomorphic to a relative fundamental group on $M\times\mathbb{R}$.    
\end{abstract}

\section{Introduction}
Morse theory is a powerful tool which allows us to extract topological invariants of smooth manifolds from the dynamical properties of gradient fields of certain functions on that manifold. In particular, one may recover the singular homology of the manifold thanks to a set of Morse data $(f,g)$, where $f$ is a Morse function over a manifold $M$, and $g$ is a metric on $M$. This inspired Floer homology, one of the most powerful tools of symplectic topology, which allows one to build an invariant from a set of Floer data $(M,\omega,H,J)$, where $(M,\omega)$ is a symplectic manifold, $H$ is a Hamiltonian function on $M$, and $J$ is an almost complex structure which is compatible with $\omega$. 

Another invariant of $M$ one can obtain from Morse theory is the fundamental group: its generators are built out of unstable manifolds of index 1 critical points of $f$, and relations are obtained by looking at the boundary of unstable manifolds of index 2 critical points. One can also see this construction in terms of moduli spaces of the form 
\[\mathcal{M}(x,M)=\{\gamma:\mathbb{R}_-\rightarrow M, \gamma'(t)=-\nabla f(\gamma(t)), \lim_{t\rightarrow-\infty} \gamma(t)=x, \gamma(0)\in M\},\]
thanks to the diffeomorphism 
\begin{equation}\label{iso:w_m} {\rm ev}: \mathcal{M}(x,M)\rightarrow W^u(x),\\ 
\gamma \mapsto \gamma(0),\end{equation} which to a trajectory starting in $x$ associates a point of the unstable manifold of $x$.
This point of view was adopted by Barraud in \cite{B} to build a Floer fundamental group constructed thanks to moduli spaces of solutions to the (potentially cut-off) Floer equation, which is isomorphic to $\pi_1(M)$. Barraud and Bertuol also built a construction of the fundamental group of $M$ using stable Morse theory \cite{BB}. 

One of the key elements to show that Morse or Floer homology are invariants of the (symplectic) manifold is the \textit{continuation map}, which was developed by Andreas Floer. This allows us to build a chain map from the Morse (or Floer) complex associated to one set of Morse data (or set of Floer data) to another. The idea is to take a smooth homotopy from one pair $(f,g)$, or $(H,J)$, to another, in such a way that there are Morse or Floer trajectories that go from generators of the first complex to generators of the second. This construction has been useful beyond Floer's work on proving invariance, for example in the construction of flow categories \cite{AB} or Kuranishi charts \cite{FOOO}. However, in the case of fundamental groups, it has not been done (although, in \cite{BB}, a different method is used, using hybrid trajectories). While continuation maps are not necessary to prove invariance of the Morse fundamental group (this is already a consequence of the isomorphism with the topological fundamental group), it is useful to get proofs of invariance that come purely from Morse or Floer theory, for example if one wants to get more applications, and especially if one wants to develop this tool in other settings. Furthermore, the continuation map being an interesting structure in and of itself, studying it in a less homological manner gives us additional insight.

The goal of this paper is therefore to study the properties of continuation maps for the Morse fundamental group. 

\subsection{The Morse fundamental group}

Let $f:M\rightarrow\mathbb{R}$ be a Morse function on a closed manifold $M$, and $g$ a metric on $M$. We assume that the pair satisfy the Smale condition, that is that for any pair $(a,b)$ of critical points of $f$, the unstable manifold of $a$ intersects the stable manifold of $b$ transversally. This is a generic condition. 

We will now recall how to retrieve the fundamental group of $M$ from the Morse data. We will use some of the approach and terminology given by Barraud in \cite{B}. The main objects in the Morse construction of the fundamental group are unstable manifolds of index 1 and 2 critical points. Recall that unstable manifolds of index $n$ critical points are disks of dimension $n$. Since we would like to use compact objects, rather than considering the compactification in $M$ of the unstable manifolds, we will consider Latour cells, which were introduced by Latour in 1994 \cite{Lat} (see also \cite{Qin},\cite{BDHO} for good expositions), which use the trajectory point of view provided by the diffeomorphism \eqref{iso:w_m} and includes the data provided by the moduli spaces 
\[\mathcal{M}(x,y)=\faktor{\{\gamma:\mathbb{R}\rightarrow M, \gamma'=-\nabla f(\gamma), \lim_{t\rightarrow-\infty} \gamma(t)=x, \lim_{t\rightarrow+\infty} \gamma(t)=y\}}{\mathbb{R}},\] and their compactification by moduli spaces of broken trajectories:
$$\overline{\mathcal{M}}(x,y)=\bigcup_{x>p_1>…>p_i>y}\mathcal{M}(x,p_1)\times\mathcal{M}(p_1,p_2)\times\dots\mathcal{M}(p_i,y).$$
\begin{definition}[Latour] Let $x\in\text{Crit}(f)$. Then its unstable Latour cell is defined as
$$\overline{\mathcal{M}}(x,M):=\mathcal{M}(x,M)\cup\bigcup_{y\in\text{Crit}(f)}\overline{\mathcal{M}}(x,y)\times \mathcal{M}(y,M).$$
\end{definition} We may consider the natural topology on these cells which makes them compact. 

The evaluation map \eqref{iso:w_m} extends to the boundary of Latour cells in the following way:
\begin{equation}\label{iso:w_u_compact} {\rm ev}:\overline{\mathcal{M}}(x,y)\times \mathcal{M}(y,M)\rightarrow W^u(y),\\ 
([\gamma_0],[\gamma_1],…,[\gamma_i],\gamma) \mapsto \gamma(0),\end{equation} 
This map on Latour cells is in general not a diffeomorphism, but Latour cells still carry nice properties that make them natural objects to consider in our setting. The following theorem was proven by Latour \cite{Lat}, as well as Qin \cite{Qin}. 
\begin{theorem}[Latour, Qin] Let $x\in{\rm Crit}_k(f)$. Then its unstable Latour cell is a compact manifold of dimension $k$ with boundary and corners, of which the interior is diffeomorphic to $W^u(x)$. Moreover, it is homeomorphic to a closed disk $D^k$.
\end{theorem}

Latour cells for index 1 critical points therefore have the topology of a closed interval. We give ourselves the following definition:

\begin{definition} A \textit{Morse step} (or \textit{$f$-step} when specifying the Morse function is necessary) is the data of the unstable Latour cell of an index 1 critical point of $f$, with a choice of orientation.\end{definition}

Essentially, Morse steps are the paths with which we will build the loops that generate the fundamental group. 

If a step corresponds to the Latour cell of a point $x\in{\rm Crit}_1(f)$, we say that it is a step \textit{through} $x$. Since they are oriented intervals, we can define a start and an end, which are boundary elements of the form $([\gamma],\gamma_y)$, where $\gamma$ is a trajectory from $x$ to an index 0 critical point $y$, and $\gamma_y$ is the constant trajectory at $y$. 
We say that a step {\em begins} at an index 0 critical point $y_0$ if the beginning of the interval is of the form $([\gamma_0],\gamma_{y_0})$, and {\em ends} at an index 0 critical point $y_1$ if the end is of the form $([\gamma_1],\gamma_{y_1})$. Two steps are {\em consecutive} if the first ends where the second one begins. 

\begin{definition} A \textit{Morse loop} (or \textit{$f$-loop} when specifying the Morse function is necessary) is a word $\sigma_1\sigma_2…\sigma_k$, where $\sigma_i$ and $\sigma_{i+1}$ are consecutive Morse steps, and such that the start of $\sigma_1$ is the end of $\sigma_k$. \end{definition}

Since we want to define a fundamental group, we will need a base point. Generically, one may choose a point $\ast$ such that it does not belong to the stable manifold of any critical point of index greater or equal to 1. There therefore exists a unique gradient trajectory $\gamma_\ast$ such that $\gamma_\ast(0)=\ast$, and the limit at $+\infty$ of $\gamma_\ast$ is a a minimum of $f$. We call $\gamma_\ast^{-1}$ the trajectory which goes in the opposite direction, ie $\gamma_\ast^{-1}(t)=\gamma_\ast(-t)$. 

\begin{definition} A based Morse loop in $\ast$ is a word of the form $\gamma_\ast\sigma_1…\sigma_k\gamma_\ast^{-1}$, where $\sigma_1…\sigma_k$ is a Morse loop such that $\sigma_1$ begins, and $\sigma_k$ ends, in the constant trajectory at the minimum reached by $\gamma_\ast$. 
\end{definition}

We call $\tilde{\mathcal{L}}(f,\ast)$ the set of Morse loops based at $\ast$. We may equip it with an equivalence relation:
$$\gamma_\ast\sigma_1…\sigma_{i-1}\sigma_i\sigma_{i}^{-1}\sigma_{i+1}…\sigma_k\gamma_\ast^{-1}\sim\gamma_\ast\sigma_1…\sigma_{i-1}\sigma_{i+1}…\sigma_k\gamma_\ast^{-1},$$ where $\sigma_i^{-1}$ is the same unstable manifold as $\sigma_i$ with the opposite orientation. We then pose $\mathcal{L}(f,\ast)=\tilde{\mathcal{L}}(f,\ast)/\!\!\sim$. This set is naturally equipped with a group structure where the operation is $$(\gamma_\ast\sigma_1…\sigma_i\gamma_\ast^{-1})\cdot(\gamma_\ast\sigma'_1…\sigma'_j\gamma_\ast^{-1})=\gamma_\ast\sigma_1…\sigma_i\sigma_1'…\sigma'_j\gamma_\ast^{-1}.$$

One may then consider the evaluation map \eqref{iso:w_m} on Morse loops: this gives a map from a step through $x$ to the unstable manifold of $x$, which forms a topological path in $M$. One may naturally concatenate the evaluations of two consecutive steps in $M$ as paths in $M$, and the evaluation of a Morse loop is therefore a topological loop in $M$. One may therefore consider an alternative but equivalent definition of $\mathcal{L}^{\rm Morse}(f,\ast)$ as the group generated by concatenations of unstable manifolds of index 1 critical points of $f$, conjugated by the path through the base point. 

For relations, we will look at Latour cells of index 2 critical points.
This is a topological 2-disk, and its boundary is formed of moduli spaces of the form $$\mathcal{M}(z,p)\times \overline{\mathcal{M}}(p,M),$$ where $|p|<2$. Since the boundary of the Latour cell is a circle, one may order these boundary components along this circle. Since for any index 1 critical point $x$, $\overline{\mathcal{M}}(x,M)$ has two boundary components which correspond to breaks in index 0 critical points, this order therefore alternates between breaks in index 0 and index 1 components. One can then consider the ordered periodic sequence of components of the form $$\mathcal{M}(z,x)\times \overline{\mathcal{M}}(x,M), |x|=1$$ along the boundary of $\overline{\mathcal{M}}(z,M)$. Each such component corresponds to a Morse step, and two components which appear consecutively in this sequence correspond to consecutive Morse steps. The whole sequence therefore corresponds to a Morse loop $\ell$ which is uniquely determined by $z$ up to shift in the letters of the word, and we write $\ell = \partial W^u(z)$. Equivalently, the evaluation of the boundary of $\overline{\mathcal{M}}(z,M)$ gives us a periodic sequence of concatenated unstable manifolds of index 1 critical points, and so we may say that $\ell=\partial W^u(z)$ if, and only if, ${\rm ev}(\ell)={\rm ev}(\partial \overline{\mathcal{M}}(z,M))$. 

We then call $\tilde{\mathcal{R}}(f,\ast)$ the normal subgroup of $\tilde{\mathcal{L}}(f,\ast)$ generated by loops of the form $$\gamma_\ast\sigma_1…\sigma_k\ell\sigma_k^{-1}…\sigma_1^{-1}\gamma_\ast^{-1},$$ where $\ell = \partial W^u(z)$ for some $z\in\text{Crit}_2(f)$, and set $\mathcal{R}(f,\ast)=\tilde{\mathcal{R}}(f,\ast)/\hspace{-4pt}\sim$. We then set $$\pi_1^{\rm Morse}(f,\ast):=\mathcal{L}(f,\ast)/\mathcal{R}(f,\ast).$$
\begin{remark}Another equivalent definition of relations may be given by considering orbits of ``crocodile walk'', introduced in \cite{B}. While we do not explicitly use crocodile walks in this text, many of our discussions may be translated into that language, which furthermore becomes necessary when one steps out of the Morse setting and wishes to study moduli spaces with more complicated topologies than Latour cells. We therefore warmly direct any interested readers in the direction of this article. \end{remark}

\begin{theorem*}[folklore] Let $(f,g)$ be a Morse-Smale pair on a manifold $M$, and $\ast$ be a generic point on $M$. Then $$\pi_1^{\rm Morse}(f,\ast)\cong\pi_1(M,\ast).$$\end{theorem*} 
For a proof of this statement, one can go to for example the PhD thesis of Florian Bertuol \cite{Be}. One can see the statement as a consequence of the fact that the fundamental group of a CW-complex can be computed by considering its 2-skeleton. A more explicit proof comes from making topological loops transverse to stable manifolds and then pushing them down by the gradient flow to unstable manifolds of index 1 and 0 critical points (one can then do the same for disks which represent homotopies between loops by floating them down to index 2 critical point).      

In order to better understand what the relations look like, and because it will be useful, we will define the following notion. For $z_1, z_2\in\text{Crit}_2(f)$ such that there exists a critical point $y$ in the intersection of their unstable manifolds, we say that the boundary parts $\overline{\mathcal{M}}(z_1,p)\times\mathcal{M}(p,M)$ and $\overline{\mathcal{M}}(z_2,p)\times\mathcal{M}(p,M)$ are \textit{compatible}. We then define the \textit{stitching} of their Latour cells as
$$\overline{\mathcal{M}}(z_1,M)\#\overline{\mathcal{M}}(z_2,M):= \overline{\mathcal{M}}(z_1,M)\cup\overline{\mathcal{M}}(z_2,M)/\sim,$$ where $(\gamma_1,\gamma)\sim(\gamma_2,\gamma)$ for any broken trajectories $\gamma_1\in\overline{\mathcal{M}}(z_2,p)$ and $\gamma_2\in\overline{\mathcal{M}}(z_1,p)$, and $\gamma\in\mathcal{M}(p,M)$. If the stitching of multiple index 2 Latour cells forms a disk, we say that it is a \textit{relation patch}. This is a natural notion, as the evaluation map naturally descends to relation patches. Its boundary is then a loop, of which the evaluation is a Morse loop. In this case, we say that a Morse loop \textit{bounds} a relation patch. The following lemma will come in handy:
\begin{lemma}\label{lemma:trivial_disks} Let $\ell$ be a Morse loop which bounds a relation patch. Then if $p$ is a sequence formed of the concatenation of $\gamma_\ast$ and a consecutive sequence of Morse steps, the Morse loop $p\ell p^{-1}$ is trivial in $\pi_1^{\rm Morse}(f,\ast)$.
\end{lemma}

\begin{proof} We will perform the proof inductively on the quantity $n$ of Latour cells which are stitched together. If $n=1$, the statement is trivial by definition of the relations in $\pi_1^{\rm Morse}(f,\ast)$. 

Suppose the statement is true for any relation patch formed of $n$ Latour cells. Consider now a relation patch formed of $n+1$ Latour cells, and call $\ell_{n+1}$ the loop that bounds it. Since each Latour cells is a disk, we can always remove one such that the rest is still a disk. Moreover, we can assume that the point at which $\ell_{n+1}$ is concatenated with $p$ does not lie in the boundary of that cell, which means we can remove that cell without any consideration of the base point. We call the loop that bounds this new collection of cells $\ell_n$, and write it as
$$\ell_n=\sigma_1…\sigma_k.$$ By our hypothesis, $p\ell_np^{-1}\in \mathcal{R}(f,\ast)$. 

Since we chose to remove a cell such that the patch stays a disk, the sequence of compatible boundary parts along which we stitched it to the rest must be connected, and in particular its image by the evaluation map forms a sequence of consecutive Morse steps. Then this sequence of consecutive steps must appear in $\ell_n$, as the cell that we removed is the obstruction to these steps appearing on the outside of the collection. We call $\sigma_i…\sigma_j$ this sequence. We may write the Morse loop which corresponds to the boundary of the removed cell as $$\sigma_i…\sigma_j\tilde{\sigma}_{p}…\tilde{\sigma}_{q},$$ since the boundary is connected and the sequence $\sigma_i…\sigma_j$ corresponds to a connected part of that boundary. Then $\tilde{\sigma}_{p}…\tilde{\sigma}_{q}$ must appear as a consecutive sequence of Morse steps in $\ell_{n+1}$, and this gives us
\begin{align*}p\ell_{n+1}p^{-1}&=p\sigma_1…\sigma_{i-1}\tilde{\sigma}^{-1}_{q}…\tilde{\sigma}^{-1}_{p}\sigma_{j+1}…\sigma_kp^{-1},\\
&\sim p\sigma_1…\sigma_{i-1}\tilde{\sigma}^{-1}_{q}…\tilde{\sigma}^{-1}_{p}\sigma_j^{-1}…\sigma_i^{-1}\sigma_i…\sigma_j\sigma_{j+1}…\sigma_kp^{-1},\\
p\ell_{n+1}p^{-1}&\sim p\sigma_1…\sigma_{i-1}\sigma_i…\sigma_j\sigma_{j+1}…\sigma_kp^{-1}=p\ell_np^{-1}\in\mathcal{R}(f,\ast),
\end{align*} because $\tilde{\sigma}^{-1}_{q}…\tilde{\sigma}^{-1}_{p}\sigma_j^{-1}…\sigma_i^{-1}$ is trivial as it corresponds to the boundary of a two dimensional Latour cell (in the opposite orientation as given before). This concludes our proof.
\end{proof}

\subsection{The main results} In this work, we use the continuation map set up to build a map $\tilde\phi_F$ from Morse steps associated to a Morse-Smale pair $(f_0,g_0)$ to another Morse-Smale pair $(f_1,g_1)$. More precisely, we build a function $F$ on $M\times\mathbb{R}$ such that 
$$\text{Crit}_k(F)=\text{Crit}_{k-1}(f_0)\times\{0\}\cup\text{Crit}_k(f_1)\times\{1\}.$$ We also consider a metric $G$ on $M\times\mathbb{R}$ such that the pair $(F,G)$ is Morse-Smale and such that the restriction of $(F,G)$ to $s=0$ (resp. $s=1$) matches $(f_0,g_0)$ (resp. $(f_1,g_1))$, and suppose that the $F$ is strictly decreasing between $s=0$ and $s=1$ in the $s$-direction. We call such a pair an \textit{interpolation pair}. Up to a small perturbation of the metric, for any minimum $\ast_0$ of $f_0$, there is a unique trajectory which goes from $(\ast_0,0)$ in the positive direction, which ends in a point in $\text{Crit}_1(f_1)\times\{1\}$, which we call $(\phi_F(\ast_0),1)$. We then use the Latour cells of points in $\text{Crit}_1(f_0)\times\{0\}$ to build a map $\tilde\phi_F$ and show that it induces a morphism between the fundamental groups:
\begin{theorem}\label{thm:interpolation_intro}
Given two Morse-Smale pairs $(f_0,g_0)$ and $(f_1,g_1)$ on a manifold $M$, a minimum $\ast_0$ of $f_0$, and an interpolation pair $(F,G)$ between them, then $\tilde\phi_F$ induces a morphism \[\phi_F: \pi_1^{\text{Morse}}(f_0,\ast_0)  \rightarrow \pi_1^{\text{Morse}}(f_1,\phi_F(\ast_0)).\] \end{theorem}

\begin{remark} As in this theorem, for the rest of the paper, we will choose the base point to be a minimum of $f_0$. This allows us to naturally have a base point for the the loops in the image of this morphism which depends only on the interpolation pair and the first base point. If one wants a more general choice of base point, one needs to choose base points $\ast_0$ and $\ast_1$ that are generic points on $M$, a minimum $y$ of $f_0$, a path from $\ast_0$ to $y$, and a path from $\ast_1$ to $\phi_F(y)$. \end{remark}
 
We then show that this morphism is functorial in the following sense:
\begin{theorem}
Let $(f_0,g_0), (f_1,g_1), (f_2,g_2)$ be three Morse-Smale pairs over $M$, and take interpolation pairs $(F_{ij},G_{ij})$ between $(f_i,g_i)$ and $(f_j,g_j)$, which induce morphisms $\phi_{ij}:\pi_1^{\rm Morse}(f_i,\ast_i)\rightarrow\pi_1^{\rm Morse}(f_j,\phi_{ij}(\ast_i))$, for $i<j$, between their Morse fundamental groups. Then there exists an isomorphism $\psi$ such that the following diagram is commutative: 

\[\begin{tikzcd}[column sep=4em, row sep=3em]
\pi_1^\text{Morse}(f_0,\ast_0) \arrow[r, "\phi_{02}"] \arrow[d, "\phi_{01}"]
& \pi_1^\text{Morse}(f_2,\phi_{02}(\ast_0)) \arrow[d, "\psi"', "\simeq"] \\
\pi_1^\text{Morse}(f_1,\phi_{01}(\ast_0))  \arrow[r, "\phi_{12}"]
& \pi_1^\text{Morse}(f_2,\phi_{12}(\phi_{01}(\ast_0))).
\end{tikzcd}\]
\end{theorem}

This theorem readily allows us to prove that the Morse fundamental group only depends on the topology of the underlying manifold $M$, and not on the other choices:%discuss with F+A

\begin{cor} Let $(f_0,g_1)$ and $(f_1,g_1)$ be two Morse-Smale pairs on a manifold $M$, and $(F,G)$ an interpolation pair between them, which induces a morphism $\phi_{F}: \pi_1^\text{Morse}(f_0,\ast_0)\rightarrow \pi_1^\text{Morse}(f_1,\phi_{01}(\ast_0))$, for $\ast_0\in{\rm Crit}_0(f_0)$. Then $\phi_F$ is an isomorphism.
\end{cor} 

Furthermore, using grafted trajectories as in \cite{C} (see also \cite{KM}, Section I.2.8.), one can build a morphism between fundamental groups of two different manifolds $X$ and $Y$. More precisely, if we take a Morse-Smale pair $(f_0,g_0)$ on $X$ and an interpolation pair $(f_0,g_0)$ from $(f_0,g_0)$ to itself, and do the same for a Morse-Smale pair $(f_1,g_1)$ on Y with an interpolation pair $(f_1,g_1)$, then using a smooth map $H:X\rightarrow Y$, one may "graft" anti-gradient trajectories of $f_1$ to those of $f_0$, which allows us to build a similar morphism between Morse fundamental groups. If furthermore, $H$ satisfies the generic condition that pairs $(x,y)\in\text{Crit}(f_0)\times\text{Crit}(f_1)$ don't intersect its graph, we call $(f_0,f_1,H)$ a set of grafted interpolation data. 
\begin{theorem}
Given a Morse-Smale pair $(f_0,g_0)$ on a manifold $X$ and a Morse-Smale pair $(f_1,g_1)$ on a manifold $Y$, a set of grafted interpolation data $(f_0,f_1, H)$, and suitably generic metrics on $X$ and $Y$, then for any minimum $\ast_0$ of $f_0$, there exists a well defined morphism $$\phi^{\rm gr}_H: \pi_1^{\text{Morse}}(f_0,\ast_0)\rightarrow \pi_1^{\text{Morse}}(f_1,\phi_H^{\rm gr}(\ast_0)).$$ 
\end{theorem}
We then prove that this morphism does not depend on the homotopy class of $H$:
\begin{theorem} 
Let $H_0$ and $H_1$ be two homotopic differentiable maps between manifolds $M_0$ and $M_1$. Let $$\phi_{H_0}^{\rm gr}:\pi_1^{\text{Morse}}(f_0,\ast_0)\rightarrow\pi^{\text{Morse}}_1(f_1,\phi_H^{\rm gr}(\ast_0))$$ and $$\phi_{H_1}^{\rm gr}:\pi_1^{\text{Morse}}(f_0,\ast_0)\rightarrow\pi_1^{\text{Morse}}(f_1,\phi_H^{\rm gr}(\ast_0))$$ be the associated grafted continuation maps between their Morse fundamental groups. Then there exists an isomorphism $\psi$ such that the following diagram is commutative: 
\[\begin{tikzcd}[column sep=5em, row sep=3em]
\pi^{\text{Morse}}_1(f_0,\ast_0) \arrow[r,"\phi_{H_0}^{\rm gr}"] \arrow[rd, "\phi_{H_1}^{\rm gr}"']
& \pi^{\text{Morse}}_1(f_1,\phi_{H_1}^{\rm gr}(\ast_0)) \arrow[d, "\psi"', "\simeq"] \\
&  \pi^{\text{Morse}}_1(f_1,\phi_{H_2}^{\rm gr}(\ast_0)).
\end{tikzcd}\]
\end{theorem}
It is also functorial in the same sense as before:
\begin{theorem}\label{thm:diag}
Let $(f_0,g_0), (f_1,g_1)$, and $(f_2,g_2)$ be three Morse-Smale pairs over manifolds $M_0$, $M_1$ and $M_2$. Let $(f_0,f_1, H_{01})$ (resp. $(f_1,f_2, H_{12})$, $(f_{0},f_2, H_{02})$, where $H_{02}=H_{12}\circ H_{01}$) be a set of grafted interpolation data between $f_0$ and $f_1$ (resp. between $f_1$ and $f_2$, between $f_0$ and $f_2$), and $\ast_0\in{\rm Crit}_0(f_0)$, which induce morphisms $\phi^{\rm gr}_{12}$, $\phi^{\rm gr}_{23}$ and $\phi^{\rm gr}_{13}$ between the corresponding Morse fundamental groups, for suitably generic metrics. Then there exists an isomorphism $\psi$ such that the following diagram is commutative: 

\[\begin{tikzcd}[column sep=4em, row sep=3em]
\pi_1^\text{Morse}(f_0,\ast_0) \arrow[r, "\phi^{\rm gr}_{02}"] \arrow[d, "\phi^{\rm gr}_{01}"]
& \pi_1^\text{Morse}(f_2,\phi_{H_{02}}^{\rm gr}(\ast_0)) \arrow[d, "\psi"', "\simeq"] \\
\pi_1^\text{Morse}(f_1,\phi_{H_{01}}^{\rm gr}(\ast_0))  \arrow[r, "\phi^{\rm gr}_{12}"]
& \pi_1^\text{Morse}(f_2,\phi_{H_{12}}^{\rm gr}(\phi_{H_{01}}^{\rm gr}(\ast_0))).
\end{tikzcd}\]
\end{theorem}

We deduce the following consequence:
\begin{cor} Let $(f_0,g_1)$ be a Morse-Smale pair on a manifold $M_0$ and $(f_1,g_1)$ be Morse-Smale pair on a manifold $M_1$, and suppose there exists a diffeomorphism $H:M_0\rightarrow M_1$. Then for a suitably generic metric, the induced morphism $$\phi^{\rm gr}_{H}: \pi_1^\text{Morse}(f_0,\ast_0)\rightarrow \pi_1^\text{Morse}(f_1,\phi_{H}^{\rm gr}(\ast_0)),$$ for $\ast_0\in{\rm Crit}_0(f_0)$, is an isomorphism.
\end{cor} 

\begin{remark}These results then prove that the Morse fundamental group can be seen as a functor from the category of based smooth manifolds to the category of groups, where the morphisms are given differentiable maps which send the first base point onto the second. In order to define the functor, one then simply needs to choose Morse functions which have minimums at those base points, which is always possible. \end{remark}

Finally, we develop a fundamental group-like invariant of a type of Morse-Smale pair $(F,g)$ on $M\times\mathbb{R}$ which we call \textit{interpolation-type} pairs. These Morse-Smale pairs generalize the interpolation pairs in the following ways: they are Morse-Smale pairs on $M\times\mathbb{R}$ which when restricted to certain slices $M\times\{p\}$, match Morse-Smale pairs on $M$, and such that the critical points of the function belong only to these slices, with alternating index shifts. We also allow the function to have different behaviors at infinity. Since these functions are defined on non-compact manifolds, our construction is one of a \textit{relative Morse fundamental group} on $M\times\mathbb{R}$. In particular, we show that it is isomorphic not to the regular fundamental group of $M$, but to a relative fundamental group:
\begin{theorem}Let $(F,G)$ be an interpolation-type pair on $M\times\mathbb{R}$, where $M$ is a closed, smooth manifold. Then for a suitable choice of base point $\ast$, there exists a constant $c\in\mathbb{R}$ and a point $\ast'\in M\times\mathbb{R}$ which depends on $\ast$, such that, if $\pi_1^\text{Morse}(F,\ast)$ is the Morse fundamental group of $F$ with base point $\ast$, and $\pi_1(\mathbb{R}\times M, F^{-1}(]-\infty,c])\cup\{\ast'\},\ast')$ is the fundamental group of $\mathbb{R}\times M$ relatively to $F^{-1}(]-\infty,c])\cup\{\ast\}$ and with base point $\ast'$, then \[\pi_1^\text{Morse}(F,\ast)\cong \pi_1(\mathbb{R}\times M, F^{-1}(]-\infty,c])\cup\{\ast'\},\ast').\]   \end{theorem}

We may then use this invariant to deduce properties about the continuation maps between fundamental groups themselves. In particular, by looking at the Morse fundamental group associated to the standard interpolation-pair $(F,G)$ used in Theorem \ref{thm:interpolation_intro}, one can use this theorem to get an alternative proof to the surjectivity of the morphism $\phi_F$. 

\subsection{Plan of the paper} In Section \ref{section:continuation_map}, we define the continuation map for the Morse fundamental group, first in the standard case in subsection \ref{subsection:usual_continuation_map}, then in the grafted case in subsection \ref{subsection:grafted_cont_map}. We then prove that these morphisms are functorial in Section \ref{section:functoriality}, again first in the standard case in subsection \ref{subsection:functoriality}, then in the grafted case in subsection \ref{subsection:grafted_functoriality}. Finally, we define the Morse fundamental group of interpolation-type functions on $M\times\mathbb{R}$ in Section \ref{section:relative}. 
\subsection{Acknowledgements} The author would like to thank her PhD supervisors Frédéric Bourgeois and Agnès Gadbled for their constant support and enthusiasm, and for carefully following the advancement of this project. She is also grateful to Jean-François Barraud for stimulating discussions about his work on the Morse, Floer and stable Morse fundamental groups, for his general encouragement, and for sending an early draft of his paper with Bertuol. She would also like to thank Julio Sampietro-Christ for discussions on grafted trajectories and for pointing out some references. 

During this project, the author benefited from funding and events organised thanks to the ANR project CoSy (ANR-21-CE40-0002), in particular the ANR meeting and the Spring school on Floer homotopy theory, as well as an invitation by Jean Gutt to the Institut de Mathématiques de Toulouse.

Finally, she is grateful to the anonymous referee for many insightful comments and suggestions, and for the care they gave when refereeing this paper. 
 
\section{A continuation map for the Morse fundamental group} \label{section:continuation_map}
\subsection{Two pairs of Morse data on a manifold $M$}\label{subsection:usual_continuation_map}
Let $M$ be a closed smooth manifold, and $(f_0, g_0)$ and $(f_1, g_1)$ be two Morse-Smale pairs on $M$.

Consider the function $h:\mathbb{R}\rightarrow\mathbb{R}$ defined by $h(s)= s^3-\frac{3}{2}s^2$, which has a maximum in $0$ and a minimum in $1$. Also consider a function $\tilde{F}:M\times\mathbb{R}\rightarrow\mathbb{R}$ which is equal to $f_0$ at $s\leq\varepsilon$ and $f_1$ at $s\geq1-\varepsilon$, for some small $\varepsilon>0$.
We set $F=\tilde{F}+Ch$, for some $C>0$ big enough so that $$\frac{\partial \tilde{F}}{\partial s}(x,s)+Ch'(s)<0, \forall (x,s)\in M\times]0,1[.$$ We perturb the metric slightly to a metric $G$ on $M\times\mathbb{R}$ which is still equal to $g_0\oplus g_\mathbb{R}$ for $s\leq\varepsilon$ and equal to $g_1\oplus g_\mathbb{R}$ for $s\geq 1-\varepsilon$, where $g_\mathbb{R}$ is the standard metric on $\mathbb{R}$, and such that the pair $(F,G)$ is Morse-Smale. 

For any $k$, the index $k$ critical points of $F$ are
$$\text{Crit}_k(F)= \text{Crit}_{k-1}(f_0)\times\{0\}\cup\text{Crit}_{k}(f_1)\times\{1\}.$$ Furthermore, the only anti-gradient trajectories of $F$ that do not go towards infinity and do not correspond to anti-gradient trajectories of $f_0$ or $f_1$ are trajectories from critical points of $f_0$ to critical points of $f_1$ (since $F$ decreases along flow lines).

\begin{figure}
\labellist 
\small\hair 2pt
\pinlabel {$f_0$} [bl] at 185 150
\pinlabel {$f_1$} [bl] at 280 65
\pinlabel {$\mathbb{R}$} [bl] at 400 50
\endlabellist
  \centerline{\includegraphics[width=10cm, height=4cm]{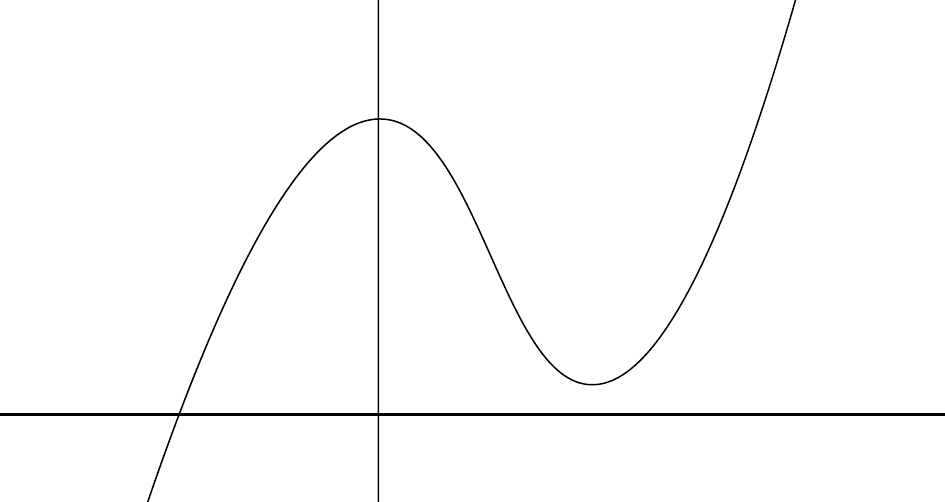}}
  \caption{$F: M\times\mathbb{R}\rightarrow\mathbb{R}$ which has a maximum at $s=0$ in $f_0$ and minimum at $s=1$ in $f_1$.}
  \label{fig:function_f1_f2}
\end{figure} 

For any critical point $y$ of $f_1$, anti-gradient trajectories of $F$ starting in $(y,1)$ stay in $M\times\{1\}$ and correspond exactly to anti-gradient trajectories of $f_1$. We therefore get 
\begin{equation}\label{diffeo_f_1}W^u((y,1))\simeq\mathcal{M}((y,1),M)\simeq\mathcal{M}(y,M)\simeq W^u(y),\end{equation} where the manifolds on the left are unstable manifolds of critical points $F$, and the ones on the right critical points of $f_1$. 

Furthermore, if $x$ is a critical point of $f_0$, we may look at the moduli space $\mathcal{M}((x,0),M\times[0,1])$ of trajectories which start in $(x,0)$ and go towards $M\times\mathbb{R}_+$:
\[
\mathcal{M}((x,0),M\times[0,1]) = \left\{\begin{array}{l}
  \gamma:\mathbb{R}_-\rightarrow M\times\mathbb{R}, \gamma'(t)=-\nabla F(\gamma(t)), \\[4pt] \lim_{t\rightarrow-\infty} \gamma(t)=(x,0), \\[4pt]
  \gamma(0)\in M\times\{u\}, 0\leq u<1
  \end{array}\right\}.
\]
As with ordinary Latour cells, it may be compactified by adding broken trajectories, which either break in points of $M\times\{0\}$ or in $M\times\{1\}$. We call such a compactification the {\em positive Latour cell} of $(x,0)$, and denote it $\overline{\mathcal{M}}((x,0),M\times[0,1])$. This has the topology of a half-disk of which the diameter is $\overline{\mathcal{M}}((x,0),M\times\{1\})$.

We then define the following map on portions of these cells by defining on each moduli space of broken trajectories:
\[ {\rm ev}_{f_1}: \overline{\mathcal{M}}((x,0),(y,1))\times\mathcal{M}((y,1),M\times\{1\})\rightarrow W^u(y),\] which is a composition of the map {\rm ev} and the diffeomorphism \eqref{diffeo_f_1}.

Let us now focus our attention on Morse loops for $f_0$ based at a point $\ast_0$ which is a local minimum of $f_0$, that is Morse loops in the set $\mathcal{L}(f_0, \ast_0)$. Each step is formed of the Latour cell of a point in ${\rm Crit}_1(f_0)$, and two are consecutive if the Latour cells contain broken trajectories which break in the same point of ${\rm Crit}_0(f_0)$.

Let $y$ be a point in ${\rm Crit}_0(f_0)$. Its Latour cell contains only the constant trajectory at itself. However, if we look at the point $(y,0)\in{\rm Crit}(F)$, it is an index 1 critical point. Its Latour cell is therefore a line segment, which corresponds to the parameterizations of two trajectories: one goes towards $-\infty$, and the other goes to a minimum of $F$ in $M\times\{1\}$, which corresponds to an index 0 critical point of $f_1$ which we call $\phi_F(y)$. The positive Latour cell therefore corresponds to the moduli space of parameterizations of the trajectory which goes from $(y,0)$ to $(\phi_F(y),1)$. 

Let $x$ be a point in ${\rm Crit}_1(f_0)$. Its Latour cell is homeomorphic to a closed interval: there are therefore two trajectories $\tilde{\beta}_1$ and $\tilde\beta_2$ from $x_1$ to two points in ${\rm Crit}_0(f_0)$, $y_1$ and $y_2$. The positive Latour cell $\overline{\mathcal{M}}((x,0),M\times[0,1])$ can be seen as a two-dimensional half-disk of which one boundary part corresponds to the Morse step through $x$. The rest of the boundary is formed of the two families of broken interrupted trajectories which break at $(y_1,0)$ and $(y_2,0)$, as well as a a connected union of moduli spaces of the form $$\overline{\mathcal{M}}((x,0),(p,1))\times\overline{\mathcal{M}}((p,1), M\times\{1\}),$$ where the index of $p$ alternates between 0 and 1. One may then consider the ordered sequence of connected components of these moduli spaces where the index of $p$ is 1. This sequence then corresponds to a sequence of consecutive Morse steps of $f_1$. If the step through $x$ begins at $y_1$, we may consider an orientation of these steps such that the sequence begins at $\phi_F(y_1)$ and ends at $\phi_F(y_2)$. See Figure \ref{fig:morphisme_phi_morse} for an illustration. We then define a map $\tilde{\phi}_F$ which to the step through $x$ assigns this sequence of consecutive steps. 

Equivalently, the map $\tilde{\phi}_F$ can be defined as the image of the part of the boundary made trajectories which break at $\{s=1\}$ by the evaluation map ${\rm ev}_{f_1}$. 

\begin{figure}
\labellist
\small\hair 2pt
\pinlabel {$(x,0)$} [bl] at 15 150
\pinlabel {$(y_1,0)$} [bl] at 50 290
\pinlabel {$(y_2,0)$} [bl] at 50 30
\pinlabel {$\beta_1$} [bl] at 30 220
\pinlabel {$\beta_2$} [bl] at 30 100
\pinlabel {$\alpha_1$} [bl] at 200 300
\pinlabel {$\alpha_2$} [bl] at 200 25
\pinlabel {$(\phi_F(y_1),1)$} [bl] at 310 305
\pinlabel {$(\phi_F(y_2),1)$} [bl] at 310 20
\endlabellist
  \centerline{\includegraphics[width=10cm]{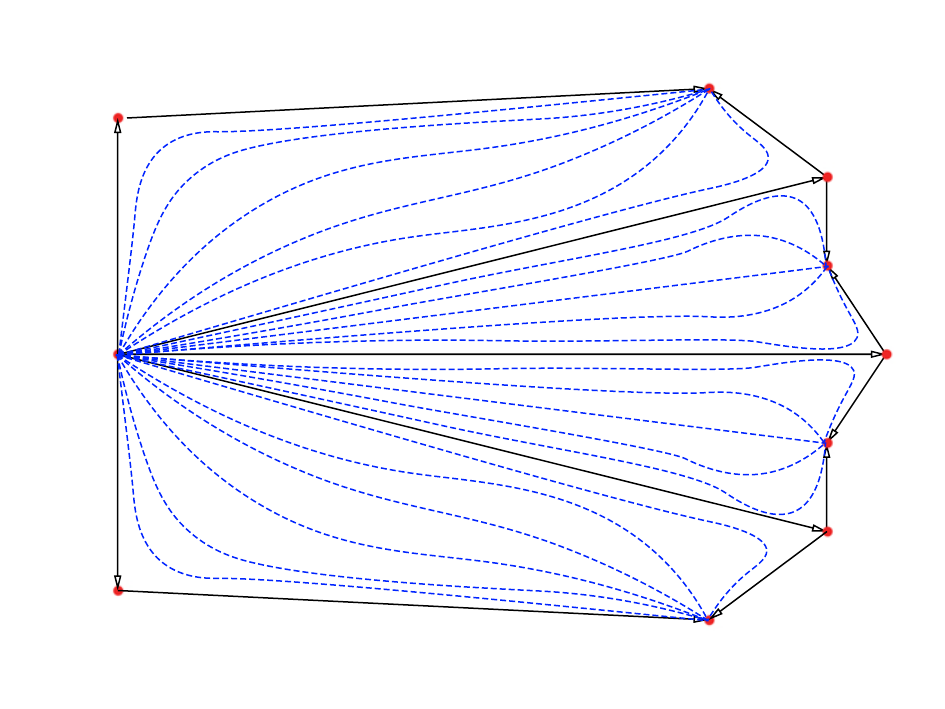}}
  \caption{Illustration of the the positive Latour cell of $(x,0)$.}
  \label{fig:morphisme_phi_morse}
\end{figure} 

Let us now see what happens if we concatenate the Morse step $\sigma$ through a point $x_1\in{\rm Crit}_1(f_0)$ with another Morse step $\sigma'$ through a point $x_2\in{\rm Crit}_1(f_0)$. Suppose $\sigma$ ends in $y_2$, and $\sigma'$ begins in $y_2$. There is a unique trajectory from $(y_2,0)$ to an index 0 critical point $(\phi_F(y_2),1)$. By construction, the series of steps in the image of $\sigma$ by $\widetilde{\phi_F}$ ends at $\phi_F(y_2)$, whereas the series of steps in the image of $\sigma'$ begins at $\phi_F(y_2)$. This means that $\tilde\phi_F$ is compatible with concatenation. If $\ell=\sigma_1…\sigma_k$ is a Morse loop for $f_0$ based in $\ast$, we may then set $$\phi_F(\ell):=\tilde{\phi}_F(\sigma_1)…\tilde{\phi}_F(\sigma_k)\in\mathcal{L}(f_1,\phi_F(\ast)).$$ 

This map is clearly compatible with concatenation of Morse loops and sends loops based in $\ast\in\text{Crit}_0(f_0)$ to loops based in $\phi_F(\ast)\in\text{Crit}_0(f_1)$. It also sends the constant loop at $\ast$ to the constant loop at $\phi_F(\ast)$ (where the constant loop is given by the empty word). Therefore, $\phi_F$ induces a well defined group morphism \[\phi_F: \mathcal{L}(f_0, \ast_0) \rightarrow \mathcal{L}(f_1,\phi_F(\ast_0)) \rightarrow \pi_1^{\text{Morse}}(f_1,\phi_F(\ast_0)).\] 

We will prove the following statement:
\begin{theorem}\label{thm:quotient}
Given two Morse-Smale pairs $(f_0,g_0)$, $(f_1,g_1)$ on a manifold $M$, a minimum $\ast_0$ of $f_0$, and an interpolation pair $(F,G)$ between them, $\phi_F$ descends to the quotient and induces a morphism \[\phi_F: \pi_1^{\text{Morse}}(f_0,\ast_0)  \rightarrow \pi_1^{\text{Morse}}(f_1,\phi_F(\ast_0)).\]
\end{theorem}
Since relations for $\pi_1^{\rm Morse}(f_0,\ast)$ come from Latour cells of index two critical points of $f_0$, in order to prove this theorem, one must study the positive Latour cell $\overline{\mathcal{M}}((z,0),M\times[0,1])$ associated to index 3 critical points $(z,0)$ of $F$, where $z\in\text{Crit}_2(f_0)$. This is a 3-dimensional closed disk. Its boundary can be decomposed in three parts: 
\begin{align*} \partial \mathcal{M}_{+}((z,0),M) = &\ \overline{\mathcal{M}}((z,0),M\times\{0\})\cup \\ &\partial\mathcal{M}((z,0),M\times]0,1[)\cup\mathcal{B}((z,0),M\times\{1\}),\end{align*}
where $\overline{\mathcal{M}}((z,0),M\times\{0\})$ is the part of the boundary composed of (broken) trajectories which end in $M\times\{0\}$, that is the compactification of
\[
\mathcal{M}((z,0),M\times\{0\}) = \left\{\begin{array}{l}
  \gamma:\mathbb{R}_-\rightarrow M\times\mathbb{R}, \gamma'(t)=-\nabla F(\gamma(t)), \\[4pt]
  \lim_{t\rightarrow-\infty} \gamma(t)=x, \gamma(0)\in M\times\{0\}
  \end{array}\right\},
\]
$\partial\mathcal{M}((z,0),M\times]0,1[)$ is the part of the boundary given by broken trajectories which end in $M\times\{u\}$, for $0<u<1$, and $\mathcal{B}((z,0),M\times\{1\})$ is the part of the boundary composed of broken trajectories which break in $M\times\{1\}$ (and therefore end in $M\times\{1\}$). 

We will start by proving the following lemma:
\begin{lemma}\label{lem:quotient} $\mathcal{B}((z,0),M\times\{1\})$ is topologically a compact 2-disk. Moreover, if $\ell=\partial W^u(z)$, then the evaluation the boundary of $\mathcal{B}((z,0),M\times\{1\})$ by ${\rm ev}_{f_1}$, with a choice of orientation, corresponds $\widetilde{\phi_F}(\ell)$. 
\end{lemma}
\begin{proof}Since $\overline{\mathcal{M}}((z,0),M\times[0,1])$ is topologically a 3-disk, its boundary is a 2-sphere. In order to show that $\mathcal{B}((z,0),M\times\{1\})$ is a 2-disk, we will therefore show that its complement in $\partial\overline{\mathcal{M}}((z,0),M\times[0,1])$ is itself a 2-disk. This complement is the union of $\overline{\mathcal{M}}((z,0),M\times\{0\})\cup\partial\mathcal{M}((z,0),M\times]0,1[)$ along their shared boundary $\partial\overline{\mathcal{M}}((z,0),M\times\{0\})$.

By construction, $\overline{\mathcal{M}}((z,0),M\times\{0\})$ is diffeomorphic to the Latour cell of $z$ as a critical point of $f_0$. Let us show that $\partial\mathcal{M}((z,0),M\times]0,1[)$ is a cylinder.
Indeed, it is formed of trajectories which are broken at least once at $u=0$. It is therefore the gluing of moduli spaces of the form
$$\overline{\mathcal{M}}((z,0),(p,0))\times\overline{\mathcal{M}}((p,0),M\times[0,1]),$$ where $p$ is an index 0 or 1 critical point of $f_0$, along shared boundary components. We start by supposing that there is a unique critical point $p$ in the unstable manifold of $z$. Then $p$ must be of index 0, and the moduli space $\overline{\mathcal{M}}((z,0),(p,0))\simeq\overline{\mathcal{M}}(z,p)\simeq\mathcal{S}^1$. On the other hand, $\overline{\mathcal{M}}((p,0),M\times[0,1])\simeq[0,1]$, since it is the set of parameterizations of the trajectory from $(p,0)$ to a unique minimum at $s=1$. Then this gives us $\mathcal{S}^1\times[0,1]$, which gives us a cylinder.

On the other hand, if there are several critical points in the unstable manifold of $z$, then for any such $p$ of index 0, any connected component of $\overline{\mathcal{M}}((z,0),(p,0))$ must be homeomorphic to a closed line segment. Indeed, it is homeomorphic to $\overline{\mathcal{M}}(z,p)$, which is a one dimensional connected manifold with boundary. This is because $\overline{\mathcal{M}}(z,M)$ is a closed disk, and so its boundary is a circle formed of the gluing of moduli spaces of the form $\mathcal{M}(z,p)\times\mathcal{M}(p,M)$ along compatible boundary elements. If $\overline{\mathcal{M}}(z,p)$ did not have boundary, this union would therefore be disconnected, which is impossible. Therefore, $\overline{\mathcal{M}}((z,0),(p,0))\times\overline{\mathcal{M}}((p,0),M\times[0,1])$ is the product of two line segments, and therefore a 2-disk. On the other hand, if $|p|=1$, then $\mathcal{M}((z,0),(p,0))\times\overline{\mathcal{M}}((p,0),M\times[0,1])$ is the product of a point and the positive Latour cell of an index 1 critical point, and so is also a disk. This is therefore a patch of 2-disks, each corresponding to a connected component of the disk $\mathcal{M}((z,0),M\times\{0\})$. One may order these components along the boundary of this disk. This therefore gives us a cyclical gluing of 2-disks which forms a cylinder.

The complement of $\mathcal{B}((z,0),M\times\{1\})$ in $\partial\overline{\mathcal{M}}((z,0),M\times[0,1])$ is therefore a disk with a cylinder attached to its boundary, which also forms a disk. This proves that $\mathcal{B}((z,0),M\times\{1\})$ is itself a disk. Its boundary matches the boundary component of the cylinder at $u=1$. The image by ${\rm ev}_{f_1}$ of the boundary at $\{s=1\}$ of the cylinder is the concatenation of the images of the boundary at $\{s=1\}$ of the positive Latour cell of each index 1 critical point which appears in $\overline{\mathcal{M}}((z,0),(x,0))$, in the order in which they appear along the boundary. By construction, if $\ell=\partial W^u(z)$, then the evaluation of the boundary at $u=1$ of the cylinder corresponds to $\tilde{\phi}_F(\ell)$. By the identification of the boundaries, this means that the image of the boundary of $\mathcal{B}((z,0),M\times\{1\})$ by the evaluation map also corresponds to $\tilde{\phi}_F(\ell)$.
\end{proof}

\begin{proof}[Proof of Theorem \ref{thm:quotient}] Recall that $ \pi_1^{\text{Morse}}(f_i,\ast_i) = \faktor{\mathcal{L}(f_i,\ast_i)}{\mathcal{R}(f_i,\ast_i)}.$ We need to show that $\mathcal{R}(f_1,\ast_1)\subset \ker(\phi_F)$. 
Let $\ell$ be a loop in $\mathcal{L}(f_1,\ast_1)$ that is the evaluation of the boundary of the Latour cell of some critical point $z \in {\rm Crit}_2(f_1)$. We will show that its image $\phi_F(\ell)$ is trivial in $\pi^{\text{Morse}}_1(f_2,\ast_2)$. 

As shown in the previous two lemmas, $\phi_F(\ell)$ is the image of the boundary of the disk $\mathcal{M}((z,0),M\times\{1\})$ by ${\rm ev}_{f_2}$. It bounds a disk formed of moduli spaces of trajectories in $M\times\mathbb{R}$, but this does not directly translate to moduli spaces of trajectories of $f_1$. To show that $\phi_F(\ell)$ is trivial, we must study the overall structure of $\mathcal{M}((z,0),M\times\{1\})$.
 
This moduli space is a disk formed of the union of moduli spaces of broken trajectories where the first break is in a critical point of the form $(q,1)$, for $q\in\text{Crit}(f_2)$ (we also include the boundaries of these moduli spaces, which may give us breaks in $M\times\{0\}$, but we will simply see these types of broken trajectories as appearing in the boundaries of the other type of moduli space, which are of dimension 2).  
Depending on the index of this critical point $q$, each moduli space has a different contribution to  $\overline{\mathcal{M}}((z,0),M\times\{1\})$ and to the evaluation of its boundary.

If $q$ is of index 2, then the image of $\mathcal{M}((z,0),(q,1))\times\overline{\mathcal{M}}((q,1),M\times\{1\})$ is diffeomorphic to the Latour cell of $q$. If $(p,1)$ is an index 1 or index 0 critical point in the boundary of the unstable manifold of $(q,1)$, then the moduli space $\overline{\mathcal{M}}((z,0),(p,1))\times\overline{\mathcal{M}}((p,1),M\times\{1\})$ joins $\mathcal{M}((z,0),(q,1))\times\overline{\mathcal{M}}((q,1),M\times\{1\})$ along their common boundary part $$\mathcal{M}((z,0),(q,1))\times\mathcal{M}((q,1),(p,1))\times\mathcal{M}((p,1),M\times\{1\}).$$ Therefore, if two index 2 critical points $(q,1)$ and $(q',1)$ share a compatible boundary part, then the associated moduli disks in $\mathcal{M}((z,0),M\times\{1\})$ which are separated by $\mathcal{M}((z,0),(p,1))\times\mathcal{M}((p,1),M\times\{1\})$. See Figure \ref{fig:index_2_moduli}. Latour cells of index 2 critical points therefore are disks which are separated or have ``holes'' which correspond to other types of moduli spaces of trajectories which break in critical points of smaller indexes. We must therefore understand what happens when $|q|<2$.

\begin{figure}
\begin{subfigure}{0.6\textwidth}
\labellist
\small\hair 2pt
\pinlabel {$\bullet$} [bl] at 265 265
\pinlabel {$\bullet$} [bl] at 580 265
\pinlabel {$z_1$} [bl] at 230 290
\pinlabel {$z_2$} [bl] at 545 290
\endlabellist
  \centerline{\includegraphics[width=6cm]{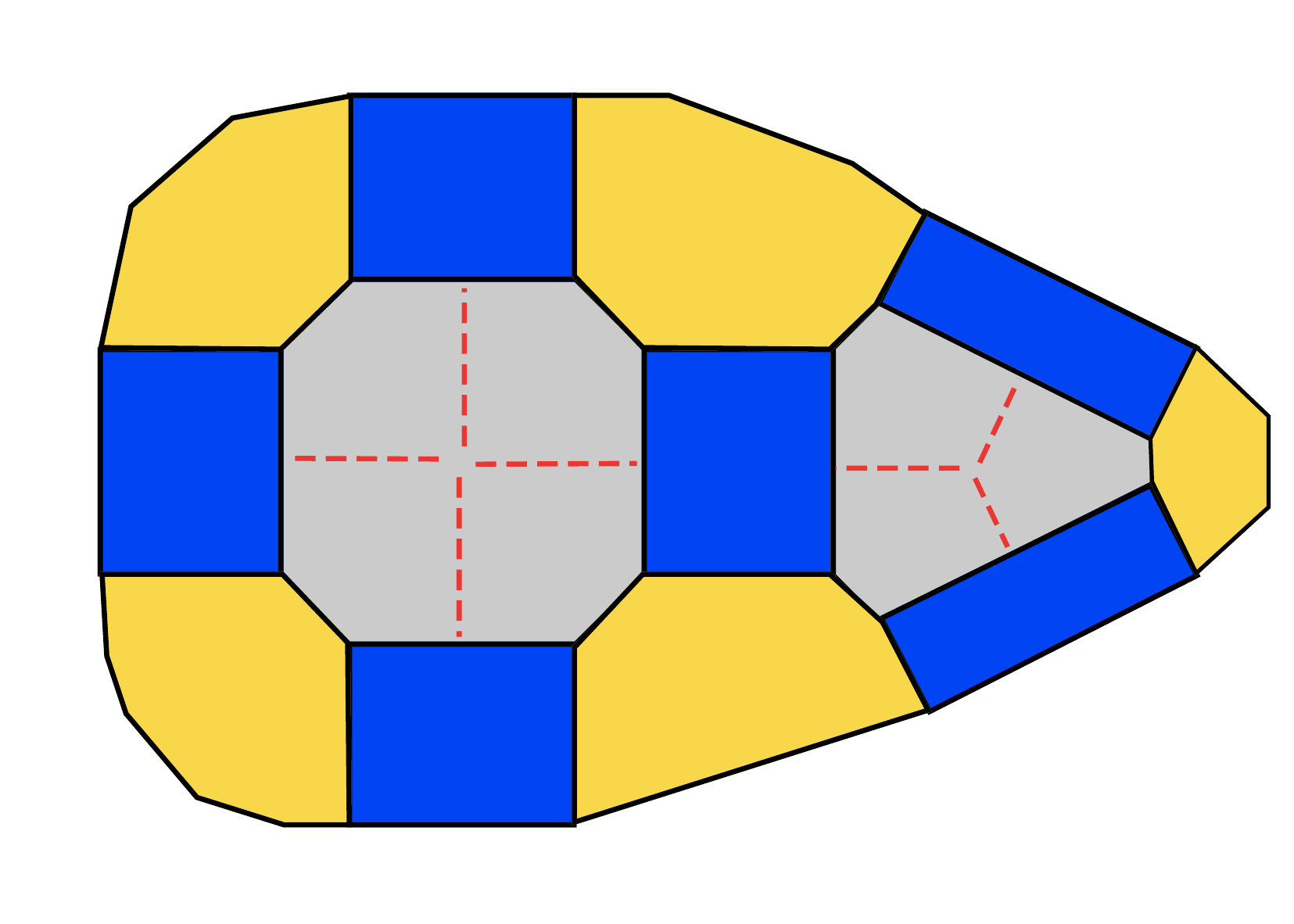}}
  \caption{$\overline{\mathcal{M}}((z,0),M\times\{1\})$.}
  \label{fig:index_2_moduli}
\end{subfigure}
\begin{subfigure}{0.4\textwidth}
\labellist
\small\hair 2pt
\pinlabel {$\bullet$} [bl] at 155 155
\pinlabel {$\bullet$} [bl] at 360 155
\pinlabel {$z_1$} [bl] at 120 175
\pinlabel {$z_2$} [bl] at 325 170
\endlabellist
  \centerline{\includegraphics[width=5cm]{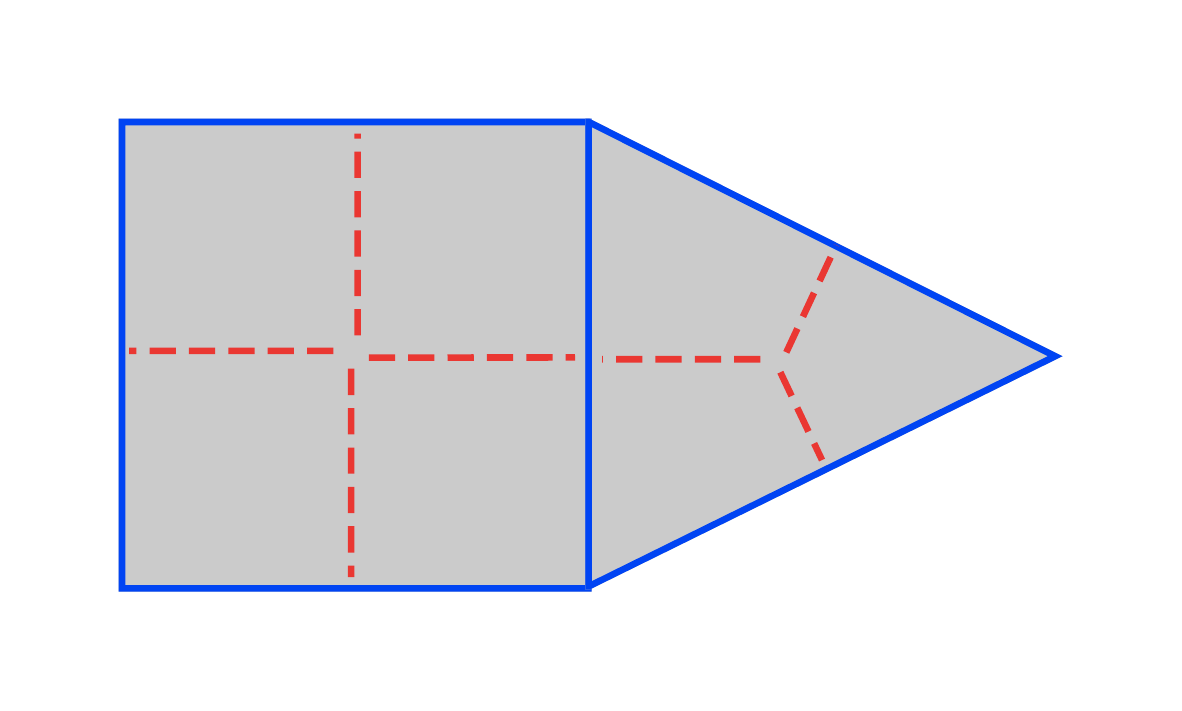}}
  \caption{The associated relation patch.}
  \label{fig:index_2_image}
\end{subfigure}
\caption{On the left, a schematic illustration of $\overline{\mathcal{M}}((z,0),M\times\{1\})$ formed of moduli spaces which correspond to breaks in index 2 critical points $z_1$ and $z_2$ (in grey, with the rigid trajectories from $z_1$ and $z_2$ depicted by dotted red lines), separated by moduli spaces which correspond to breaks in index 1 (in blue) or 0 (in yellow) critical point. On the right, the corresponding relation patch, where the blue lines correspond to the projection of the blue squares onto the second coordinate.}
\label{fig:index_2}
\end{figure}

If $q$ is of index 1, then we have several possibilities, depending on the boundary of $\overline{\mathcal{M}}((z,0),(q,1))$. If a component of $\overline{\mathcal{M}}((z,0),(q,1))$ does have boundary, then its trajectories must break in index 2 critical points of $F$, and the product with $\overline{\mathcal{M}}((q,1),M\times\{1\})$ forms a square, of which the image by ${\rm ev}_{f_2}$ is a line which corresponds to the unstable manifold of $q$. We first consider the case where both breaks are in points of the form $(z',1)$ and $(z'',1)$, for $z',z''\in\text{Crit}_2(f_2)$. In this case, the moduli space $\overline{\mathcal{M}}((z,0),(q,1))\times\overline{\mathcal{M}}((q,1),M\times\{1\})$ appears in the interior of the disk, between the moduli spaces $\mathcal{M}((z,0),(z',1))\times\overline{\mathcal{M}}((z',1),M\times\{1\})$ and $\mathcal{M}((z,0),(z'',1))\times\overline{\mathcal{M}}((z'',1),M\times\{1\})$. The moduli space $\overline{\mathcal{M}}((q,1),M\times\{1\})$ then corresponds to a compatible boundary part between the Latour cells of $z'$ and $z''$ (see the blue line which separates the grey disks in \ref{fig:index_2_image}). Alternatively, suppose that the moduli space $\overline{\mathcal{M}}((z,0),(q,1))$ breaks on one end at a point of the form $(z',1)$, for $z'\in\text{Crit}_2(f_2)$, and on the other end at a point of the form $(x_1,0)$, for $x_1\in\text{Crit}_1(f_1)$. Then the boundary at the break at $(x_1,0)$ appears on the outer boundary of $\overline{\mathcal{M}}((z,0),(q,1))$, and the given step through $q$, which is in the image of the step through $x_1$ by $\tilde{\phi}_F$, appears in the evaluation of the boundary of the Latour cell of $z'$. This corresponds to the blue lines which bound the disk in Figure \ref{fig:index_2_image}. In either of these cases, this sort of moduli space does does not prevent the Latour cells of index 2 critical points from forming a relation patch.

On the other hand, if both breaks are in points of the form $(x_1,0)$ and $(x_2,0)$, where $x_1,x_2\in\text{Crit}_1(f_0)$, then this is not the case. Call $\sigma_1$ and $\sigma_2$ the steps through $x_1$ and $x_2$, respectively. Then if $\sigma_q^{\pm1}$ are the steps through $q$, then $\sigma_q$ or $\sigma_q^{-1}$ are in the image of $\sigma_1$ and $\sigma_2$ by $\tilde{\phi}_F$. Suppose $\sigma_q$ is in the image of $\sigma_1$ by $\tilde{\phi}_F$. Then it corresponds to a choice of orientation for the moduli space $$\mathcal{M}((z,0),(x_1,0))\times\mathcal{M}((x_1,0),(q,1))\times\overline{\mathcal{M}}((q,1),M\times\{1\})$$ which is one boundary part of the square $$\overline{\mathcal{M}}((z,0),(q,1))\times\overline{\mathcal{M}}((q,1),M\times\{1\}).$$ Reading the boundary of this square with the compatible orientation, one can see that the opposite boundary part $$\mathcal{M}((z,0),(x_2,0))\times\mathcal{M}((x_2,0),(q,1))\times\overline{\mathcal{M}}((q,1),M\times\{1\})$$ has the opposite orientation. One can therefore deduce that $\sigma_q^{-1}$ appears in the image of $\sigma_2$. 

In this case, the first coordinate of the moduli space $\overline{\mathcal{M}}((z,0),(q,1))\times\overline{\mathcal{M}}((q,1),M\times\{1\})$ therefore forms a one-dimensional family which goes from one part of the boundary of the disk $\overline{\mathcal{M}}((z,0),M\times\{1\})$ to another. Hence, the product forms a strip in  $\overline{\mathcal{M}}((z,0),M\times\{1\})$ which cuts the disk in two. Furthermore, consider another such strip $\overline{\mathcal{M}}((z,0),(q',1))\times\overline{\mathcal{M}}((q',1),M\times\{1\})$, for $q'\in\text{Crit}_1(f_1)$, which also forms a one-dimensional family from one part of the boundary of $\overline{\mathcal{M}}((z,0),M\times\{1\})$ to another. Then these families cannot intersect, as that would mean there exists a trajectory which lies both in $\mathcal{M}((z,0),(q,1))$ and in $\mathcal{M}((z,0),(q',1))$. In particular, both boundary parts that each family connects must lie on the same side of the disk as one another. Then, depending on where the base point lies, we get $\tilde{\phi}_F(\ell)$ either of the form \begin{equation}\label{eq:ind_1_normal_1}\tilde{\phi}_F(\ell)=…\sigma_q…\sigma_q^{-1}…\sigma_{q'}…\sigma_{q'}^{-1}…\end{equation} or of the form \begin{equation}\label{eq:ind_1_normal_2}\tilde{\phi}_F(\ell)=…\sigma_q…\sigma_{q'}…\sigma_{q'}^{-1}…\sigma_q^{-1}…\end{equation}
If the parts represented by the dots in these equations bound relation patches, then this forms a trivial Morse loop. See Figure \ref{fig:strips}.
\begin{figure}
\begin{subfigure}{0.5\textwidth}
\labellist
\small\hair 2pt
\pinlabel {$\ell$} [bl] at 30 230
\pinlabel {$\sigma_1$} [bl] at 75 560
\pinlabel {$\sigma_2$} [bl] at 180 640
\pinlabel {$\ell'$} [bl] at 750 250
\pinlabel {$\sigma_1^{-1}$} [bl] at 150 18
\pinlabel {$\sigma_2^{-1}$} [bl] at 265 -20
\endlabellist
  \centerline{\includegraphics[width=5cm]{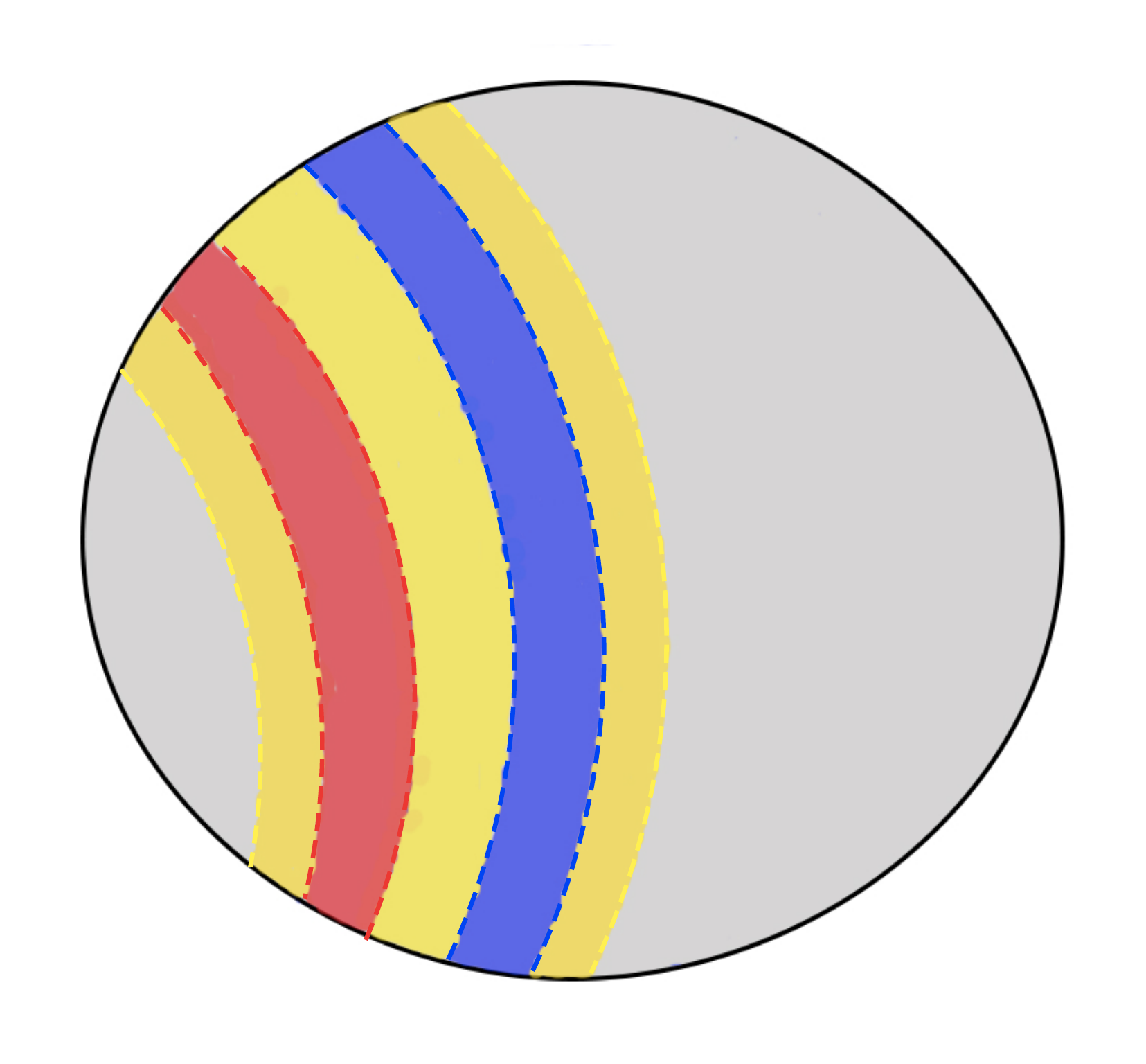}}
  \caption{Strips in $\overline{\mathcal{M}}((z,0),M\times\{1\})$.}
  \label{fig:disk_strips}
\end{subfigure}
\begin{subfigure}{0.5\textwidth}
\labellist
\small\hair 2pt
\pinlabel {$\ell$} [bl] at 10 210
\pinlabel {$\sigma_1$} [bl] at 310 225
\pinlabel {$\sigma_2$} [bl] at 440 260
\pinlabel {$\ell'$} [bl] at 900 250
\pinlabel {$\sigma_1^{-1}$} [bl] at 310 120
\pinlabel {$\sigma_2^{-1}$} [bl] at 440 165
\endlabellist
  \centerline{\includegraphics[width=5cm]{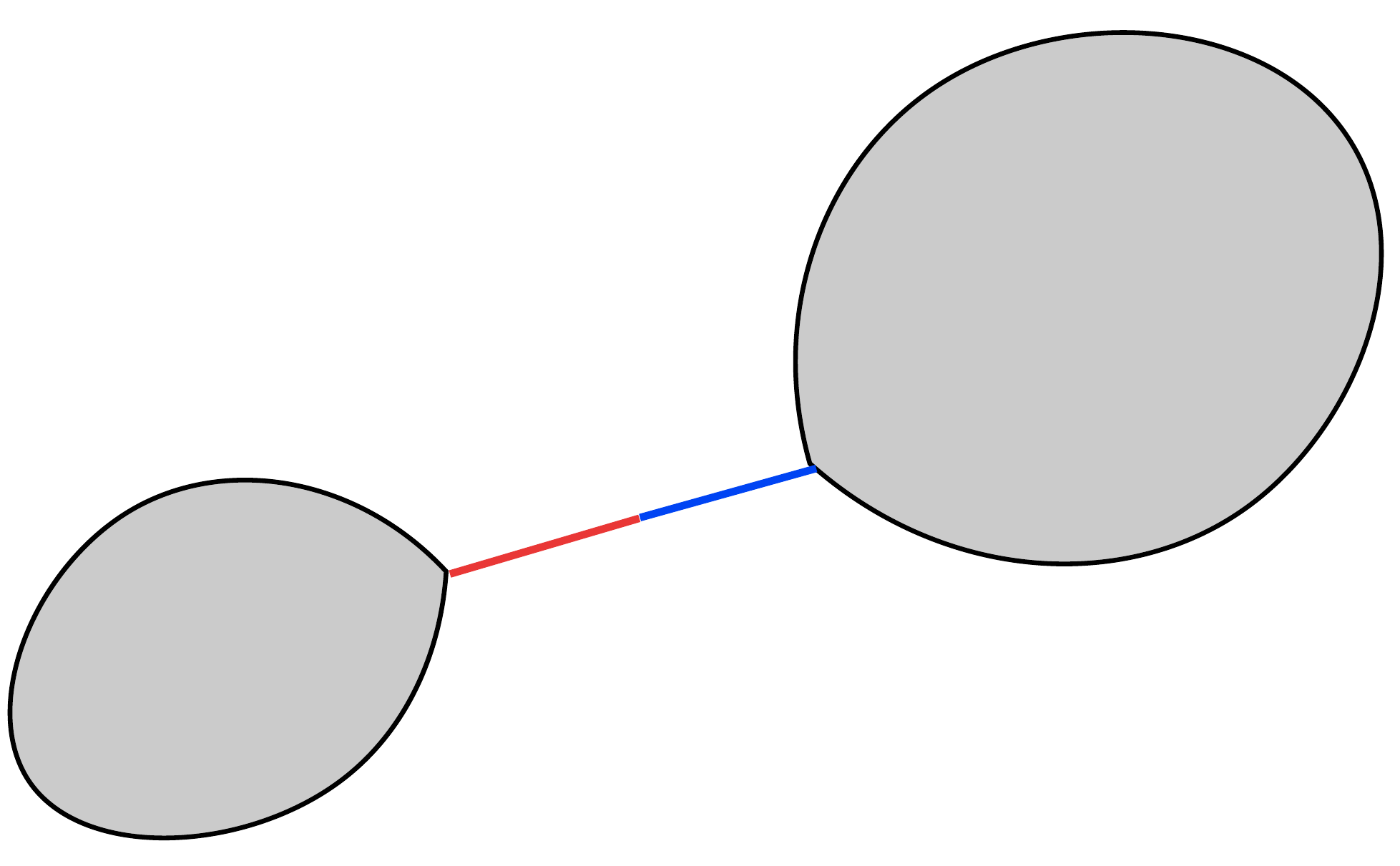}}
  \caption{Image of the strips by ${\rm ev}_{f_2}$.}
  \label{fig:image_strips}
\end{subfigure}
\caption{On the left, a schematic illustration of two strips (in red and blue) in $\overline{\mathcal{M}}((z,0),M\times\{1\})$ which correspond to moduli spaces of broken trajectories at index 1 critical points which go from one part of the boundary to another (surrounded by yellow strips which correspond to the moduli spaces of the form $\overline{\mathcal{M}}((z,1),(y,0))\times\overline{\mathcal{M}}((y,0),M\times\{1\})$, where each $y$ is an index 0 critical point in the unstable manifolds of the index 1 critical points). On the right, the image of the strips by ${\rm ev}_{f_2}$.}
\label{fig:strips}
\end{figure}

The last possibility when $q$ is of index 1 is for one of the components of $\overline{\mathcal{M}}((z,0),(q,1))$ to have no boundary. Then it is diffeomorphic to $\mathcal{S}^1$, and the product with $\overline{\mathcal{M}}((q,1),M\times\{1\})$ gives us an annulus in $\overline{\mathcal{M}}((z,0),M\times\{1\})$, which joins the rest of the disk along the two circles which form its boundary. The annulus therefore separates $\overline{\mathcal{M}}((z,0),M\times\{1\})$ into an outer annulus and an inner disk. Furthermore, the boundary of the annulus corresponds to a break at an index 0 critical point. All the moduli spaces which lie around on the outside of the annulus and share a boundary component with it therefore correspond to Latour cells of $f_1$ which share a compatible boundary part which corresponds to breaks at that critical point. We may therefore contract the inner disk and the annuli onto a point, and stitch together all the other moduli spaces along that boundary part. Since we are only interested in the loop that forms the boundary of $\overline{\mathcal{M}}((z,0),M\times\{1\})$, then we may assume there does not exist such annuli. For a helpful illustration, see Figure \ref{fig:annulus}.

Note that this same argument may be used when $q$ is of index 0. Indeed, the moduli space $\overline{\mathcal{M}}((z,0),(q,1))\times\mathcal{M}((q,1),M\times\{1\})$ glues together with the rest of the moduli spaces to form a disk. It it either contractible, or separates the moduli space into an outer annulus and several inner disks. In the same way as before, we may still contract it to a point and stitch together the rest of the moduli spaces which lie around its outer boundary along that point. 

\begin{figure}
\begin{subfigure}{0.5\textwidth}
  \centerline{\includegraphics[width=5cm]{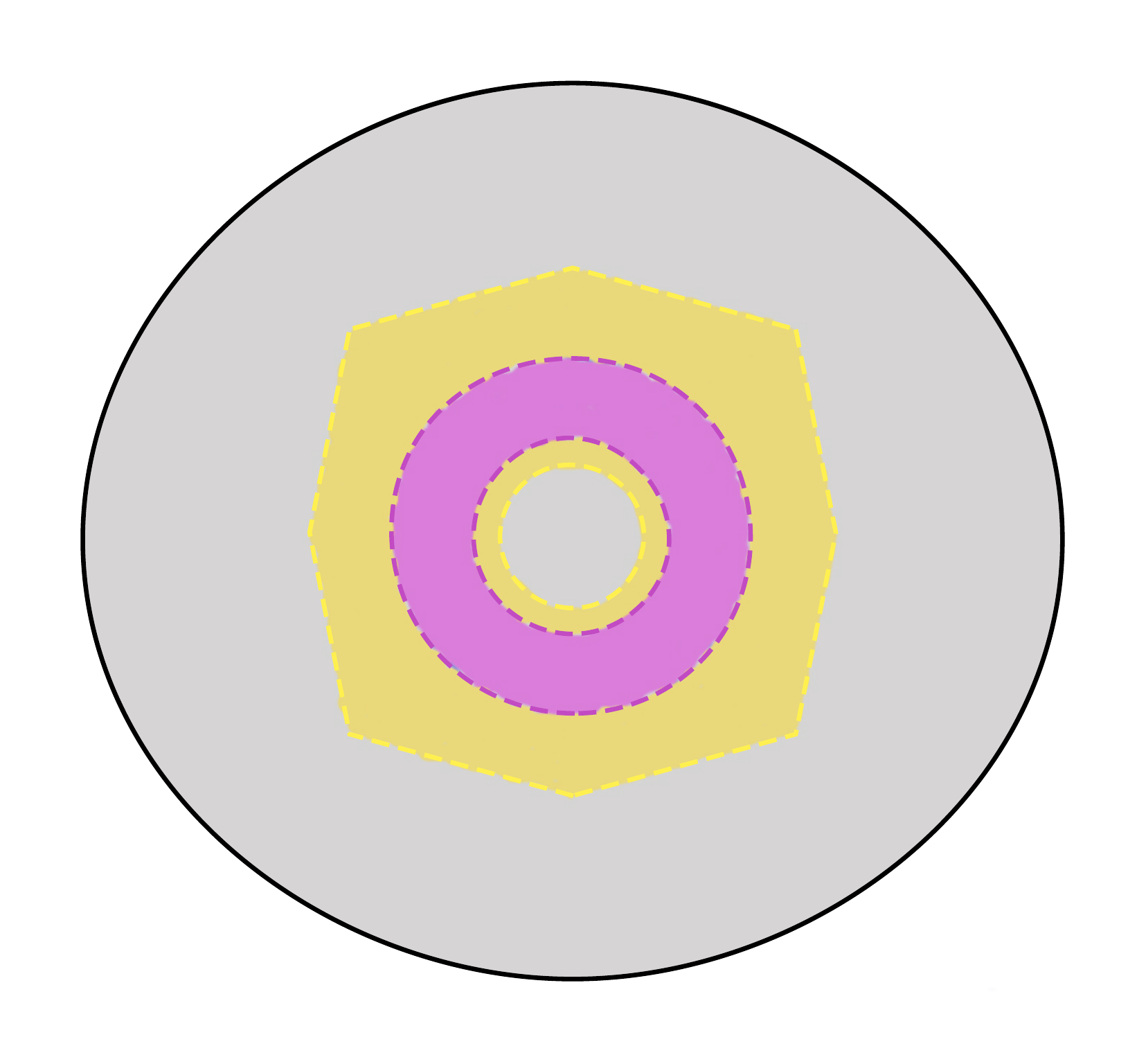}}
  \caption{An annulus in $\overline{\mathcal{M}}((z,0),M\times\{1\})$.}
  \label{fig:disk_annulus}
\end{subfigure}
\begin{subfigure}{0.5\textwidth}
  \centerline{\includegraphics[width=5cm]{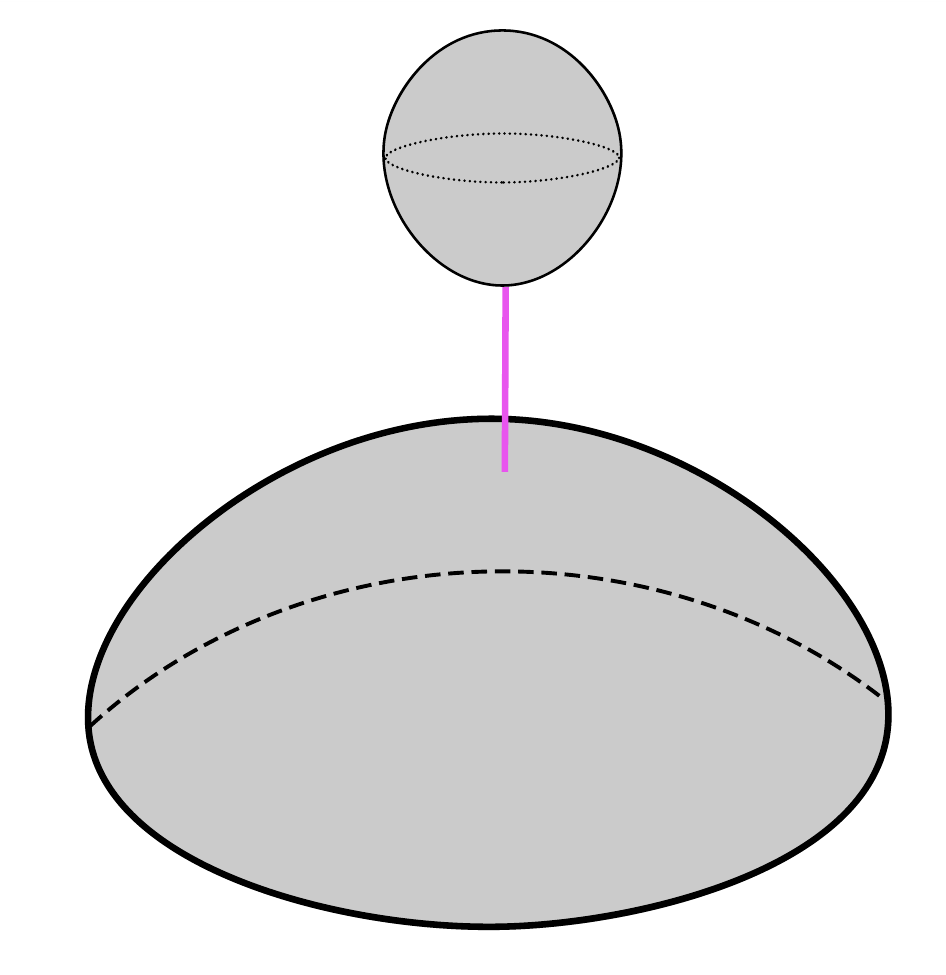}}
  \caption{Its image by ${\rm ev}_{f_1}$.}
  \label{fig:image_annulus}
\end{subfigure}
\caption{Schematic illustration of an annulus (in pink) in $\overline{\mathcal{M}}((z,0),M\times\{1\})$, as well as its image by ${\rm ev}_{f_1}$: the image of the annulus is the pink line, and the image of the part of the disk bounded by the annulus is a sphere. The yellow zones depict the moduli spaces of the form $\overline{\mathcal{M}}((z,1),(y,0))\times\mathcal{M}((y,0),M\times\{1\})$, where each $y$ is an index 0 critical point, and are therefore sent by ${\rm ev}_{f_2}$ onto points. Since we are interested in the image of boundary of the whole disk, one can ignore the pink line and the sphere attached.}
\label{fig:annulus}
\end{figure}

The image of the boundary of $\overline{\mathcal{M}}((z,0),M\times\{1\})$ by ${\rm ev}_{f_1}$ is therefore either formed of parts of the boundary of relation patches, or by steps which appear as in equations \eqref{eq:ind_1_normal_1} and \eqref{eq:ind_1_normal_2}. It is therefore in the normal subgroup of $\mathcal{L}(f_1,\phi_F(\ast_0))$ generated by relation patches. Then the fact that the Morse loop forming its boundary is trivial is a consequence of Lemma \ref{lemma:trivial_disks}. Therefore, $\phi_F(\ell)=0$. \end{proof}

\begin{remark} One may want to see what happens to such a map if we choose instead base points which are not minimums of our Morse functions, but rather generic points on $M$. For $\pi_1^{\rm Morse}(f_0,\ast_0)$, if $\ast_0$ is in the stable manifold of a minimum $y_0$ of $f_0$, this amounts to conjugating $y_0$-based loops with the half-trajectory which passes through $\ast_0$ and goes to $y_0$. We may then apply the morphism $\phi$, which gives us loops based in $\phi_F(y_0)$. However, there is no reason for $\ast_1$ to be in the stable manifold of $\phi_F(y_0)$. One must then choose a path from $\ast_1$ to $\phi_F(y_0)$, and push it down by the gradient flow of $f_1$. The result will be the concatenation of a half-trajectory from $\ast_1$ to a certain minimum of $f_1$, and then a series of steps which goes from that minimum to $\phi_F(y_0)$. We may then conjugate by this to get loops based in $\ast_1$. Unfortunately, there is no canonical way to make the choice of path from $\ast_1$ to $\phi_F(y_0)$ which depends purely on the Morse data.  
\end{remark}

\subsection{The grafted continuation map}\label{subsection:grafted_cont_map}
This construction can be extended to the case where $f_0$ and $f_1$ are two Morse functions over different manifolds by considering grafted flow lines (see \cite{C},\cite{KM}). More precisely, consider $\tilde{H}: X \rightarrow Y$ a differentiable map between two smooth, closed manifolds, $(f_0,g_0)$ a Morse-Smale pair over $X$, $(f_1,g_1)$ a Morse-Smale pair over $Y$. We may assume, up to a small perturbation of $H$, that no pairs of critical points of $f_0$ and $f_1$ lie in the graph of $H$. We call such a triple $(f_0,f_1,H)$ a set of grafted interpolation data between $f_0$ and $f_1$. 

Since no pairs of critical points lie on the graph $\Gamma(H)$ of $H$, we may perturb the metrics $g_0$ and $g_1$ slightly so that for any pair of critical points $x\in\text{Crit}(f_0), y\in\text{Crit}(f_1)$, $(W^u(x)\times W^s(y))\pitchfork \Gamma(H)$. This is a generic condition that is independent of the Morse-Smale condition on the pairs $(f_0,g_0)$ and $(f_1,g_1)$, and so we may assume that both conditions hold. 

$\tilde{H}$ naturally extends to a differentiable map $H:X\times\mathbb{R}\rightarrow Y\times\mathbb{R},\ (x, s) \mapsto (H(x), s)$.   
We set $F_0 = f_0+h$, $F_1=f_1+h$, where as in the non-grafted case,  $h(s)= s^3-\frac{3}{2}s^2$ (we can see each $F_i$ as a constant interpolation from $f_i$ to itself). We also extend the metrics to $M\times\mathbb{R}$ by taking a split metric with the usual metric on $\mathbb{R}$.

We are interested in the moduli spaces of grafted flow lines between critical points of $f_0$ (which correspond to critical points of $f_0$, with a Morse index shift of one) and critical points of $f_1$ (which correspond to critical points of $f_1$, with the same Morse index), that is the moduli space of pairs of paths $(\gamma_{-}, \gamma_{+})$ that satisfy 
\begin{align*} \gamma_{-}: \mathbb{R}_{-} &\rightarrow X\times\mathbb{R},\ \gamma_{-}'(t) = -\nabla f_0(\gamma_{-}(t)), \\
\gamma_{+}: \mathbb{R}_{+} &\rightarrow Y\times\mathbb{R},\ \gamma_{+}'(t) = -\nabla f_1(\gamma_{+}(t)),\end{align*}
and such that 
\begin{gather*} \lim_{t\rightarrow-\infty} \gamma_{-}(t) \in {\rm Crit}(f_0), \gamma_-(0)\in X \times\{1/2\},\\
\lim_{t\rightarrow+\infty} \gamma_{+}(t) \in {\rm Crit}(f_1), \gamma_+(0)\in Y \times\{1/2\},\\
H(\gamma_{-}(0))=\gamma_{+}(0).\end{gather*} See Figure \ref{fig:grafted_traj}.

\begin{figure}
  \centerline{\includegraphics[width=10cm]{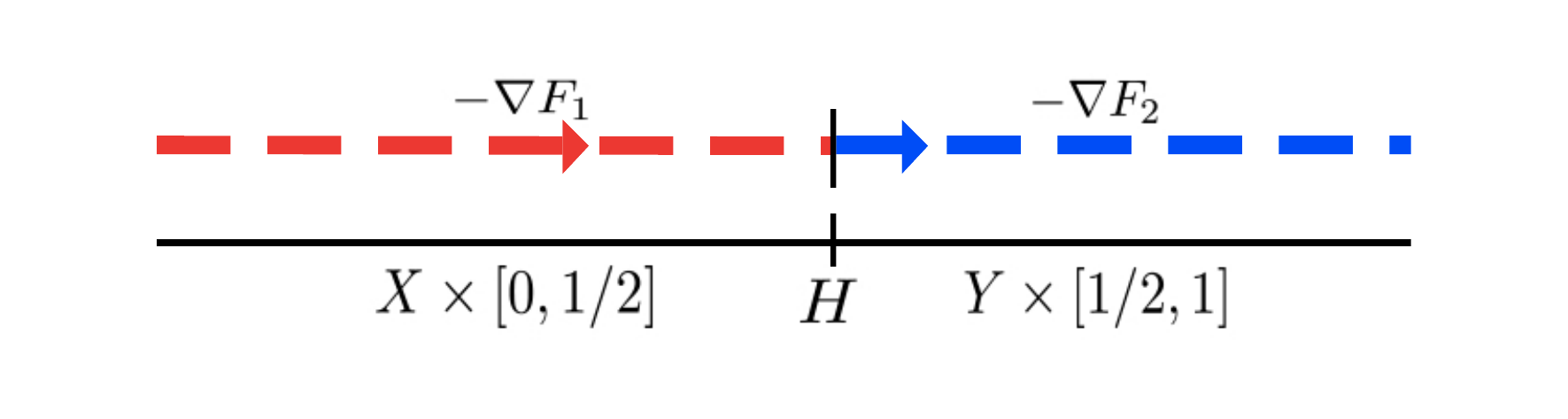}}
  \caption{A grafted trajectory.}
  \label{fig:grafted_traj}
\end{figure}

For $(x,0)\in{\rm Crit}(f_0)$, $(y,1)\in{\rm Crit}(f_1)$, if $\mathcal{M}_{\rm gr}((x,0),(y,1))$ is the moduli space of grafted trajectories between the two, we get a diffeomorphism: 
\begin{align*}\mathcal{M}_{\rm gr}((x,0),(y,1))&\rightarrow [W^{u}((x,0))\times W^{s}((y,1))] \cap \Gamma(H_{|\{s=1/2\}}),\\
(\gamma_-,\gamma_+)&\mapsto(\gamma_-(0),\gamma_+(0)),\end{align*}
where $\Gamma(H)$ is the graph of $H$. We have chosen metrics so that $\Gamma(H_{|\{s=1/2\}})$ intersects $W^u((x,0))\times W^s((y,1))$ transversally for any $x\in{\rm Crit}(f_0)$ and $y\in{\rm Crit}(f_1)$: since the dynamics are split and the gradient flow of $h$ intersects $\{s=1/2\}$ transversally, this is equivalent to the condition $(W^u(x)\times W^s(y))\pitchfork \Gamma(H)$. It follows that
\begin{align*}
\dim\mathcal{M}_{\rm gr}(x,y)=&\ {\rm ind}_{f_0}((x,0))+\dim (Y\times\mathbb{R}_+)-{\rm ind}_{f_1}((y,1))\\&+\dim \Gamma(H_{|\{s=1/2\}}) -\dim ((X\times\mathbb{R}_-)\times(Y\times\mathbb{R}_+)),\\
=&\ {\rm ind}_{f_0}((x,0)) + \dim Y + 1 - {\rm ind}_{f_1}((y,1)) \\&+\dim X - \dim X - 1 -\dim Y -1,\\
\dim\mathcal{M}_{\rm gr}(x,y)=&\ {\rm ind}_{f_0}((x,0))-{\rm ind}_{f_1}((y,1))-1.
\end{align*}This implies that for there to be a finite amount of grafted trajectories between two points, their Morse indices must be one apart. 

More generally, one can look at the moduli space $$\mathcal{M}_{{\rm gr}}((x,0),X\times[0,1/2[\cup Y\times[1/2,1[)$$ of grafted trajectories from $(x,0)\in\text{Crit}(f_0)$ which move in the positive direction, without specifying the end. This is the analog to $\mathcal{M}((x,0),M\times[0,1])$ in the non grafted case. To be more precise, this is the moduli space of pairs of trajectories $(\gamma_-,\gamma_+)$, such that there exists $u\in[0,1[$ which satisfies:
\begin{gather*}\text{ if $u<1/2$, } \gamma_{-}: \mathbb{R}_{-} \rightarrow X\times\mathbb{R},\ \gamma_{-}'(t) = -\nabla f_0(\gamma_{-}(t)),
\\ \lim_{t\rightarrow-\infty} \gamma_{-}(t) =(x,0), \gamma_-(0)\in X\times\{u\}, \gamma_+=\varnothing,\\
\text{ and if $u\geq 1/2$},  \gamma_{-}: \mathbb{R}_{-} \rightarrow X\times\mathbb{R},\ \gamma_{-}'(t) = -\nabla f_0(\gamma_{-}(t)),
\\ \lim_{t\rightarrow-\infty} \gamma_{-}(t) =(x,0), \gamma_-(0)\in X\times\{1/2\},\\
\gamma_{+}: [0,v] \rightarrow Y\times\mathbb{R},\gamma_{+}'(t) = -\nabla f_1(\gamma_{+}(t)),\\ H(\gamma_-(0))=\gamma_+(0), \gamma_+(v)\in Y\times\{u\}.\end{gather*}
\begin{remark} Notice that $v$ is entirely determined by $u$: since the dynamics are split on $Y\times\mathbb{R}$, one can compute that the real coordinate of $\gamma_+$ is given by $\frac{1}{1+e^{-3t}}$. One can then see that for $\gamma_+(v)$ to be in $M\times\{u\}$, then $v=\frac{1}{3}\log(\frac{u}{1-u})$.
\end{remark}
This moduli space can be seen as the gluing of 
\[
\mathcal{M}((x,0),X\times[0,1/2]) = \left\{\begin{array}{l}
    (\gamma_-,M), \gamma_-:\mathbb{R}_-\rightarrow  X\times\mathbb{R}, \\[3pt] \gamma_{-}'(t) = -\nabla f_0(\gamma_{-}(t)), \\[3pt]
    \lim_{t\rightarrow-\infty} \gamma_{-}(t)=(x,0), \\[3pt] \gamma_{-}(0)\in X\times[0,1/2]
  \end{array}\right\},
\]
and 
\[
 \mathcal{M}_{\rm gr}((x,0),Y\times[1/2,1[) = \left\{\begin{array}{l}
    (\gamma_-, \gamma_+), \gamma_-:\mathbb{R}_{-} \rightarrow X\times\mathbb{R},\\[3pt]
    \gamma_{-}'(t) = -\nabla f_0(\gamma_{-}(t)),\\[3pt]
    \lim_{t\rightarrow-\infty} \gamma_{-}(t) =(x,0),\\[3pt]
    \gamma_-(0)\in X\times\{1/2\}, \gamma_{+}: [0,v] \rightarrow Y\times\mathbb{R},\\[3pt]
    \gamma_{+}'(t) = -\nabla f_1(\gamma_{+}(t)),H(\gamma_-(0))=\gamma_+(0),\\[3pt]
    \gamma_+(v)\in Y\times\{u\}, u\in[1/2,1[
  \end{array}\right\},
\]
along $H_{|\{s=1/2\}}$ (more specifically, we quotient the union of these two sets by the equivalence relation which identifies a trajectory which ends in $X\times\{1/2\}$ with the corresponding grafted trajectory which ends in $Y\times\{1/2\}$). One can see that $\mathcal{M}((x,0),X\times[0,1/2[)$ is diffeomorphic to $W^u((x,0))\cap (X\times[0,1/2[)$, which is diffeomorphic to $W^u(x)\times[0,1/2[$ due to the split nature of the flow, whereas, thanks to the lack of critical points in $Y\times[1/2,1[$, $\mathcal{M}_{\rm gr}((x,0),Y\times[1/2,1[)$ is diffeomorphic to $(W^u((x,0))\cap (X\times\{1/2\}))\times\mathbb{R}_+$, which is diffeomorphic to $W^u(x)\times\{1/2\}\times\mathbb{R}_+$. In particular, $$\dim\mathcal{M}_{{\rm gr}}((x,0),X\times[0,1/2[\cup Y\times[1/2,1[)={\rm ind}_{F_0}((x,0))={\rm ind}_{f_0}(x)+1.$$

We may compactify this moduli space by adding broken grafted trajectories. The boundary is therefore composed of moduli spaces of the form
\begin{gather*}\mathcal{M}((x,0),(x_1,0))\times…\times\mathcal{M}((x_{n-1},0), (x_n,0))\\ \times\mathcal{M}_{\rm gr}((x_n,0),X\times[0,1/2[\cup Y\times[1/2,1[),\end{gather*} where $\text{ind}_{f_0}(x)>…>\text{ind}_{f_0}(x_n)$, or of the form
\begin{gather*}\mathcal{M}((x,0),(x_1,0))\times…\times\mathcal{M}((x_{n-1},0), (x_n,0))\times\mathcal{M}_{\rm gr}((x_n,0),(y_1,1))\\ \times\mathcal{M}((y_1,1),(y_2,1))\times…\times\mathcal{M}((y_{m-1},1),(y_m,1))\times\mathcal{M}((y_m,1),Y\times\{1\}),\end{gather*} where $\text{ind}_{f_0}(x)>…>\text{ind}_{f_0}(x_n)>\text{ind}_{f_1}(y_1)>…>\text{ind}_{f_1}(y_m)$. As in the non-grafted case, these compactified moduli spaces have the topology of closed disks of dimension $\text{ind}_{f_0}(x)+1$.
If the end of the broken grafted trajectory is in $Y\times\{1\}$, we set
\[{\rm ev_{gr}}(([\gamma_{-,1}],…,\gamma_{-,n},\gamma_{+,1},[\gamma_{+,2}],…,[\gamma_{+,m-1}],\gamma_{+,m}))=\gamma_{+,m}(0),\] In this case, since the diffeomorphism \eqref{diffeo_f_1} still applies in the non grafted case, we may compose this ${\rm ev_{gr}}$ with that diffeomorphism to get a map
\[{\rm ev}_{{\rm gr},f_1}:\overline{\mathcal{M}}_{\rm gr}((x,0),(y,1))\times\mathcal{M}((y,1), Y\times\{1\}) \rightarrow W^u(y).\]

If $(x,0)$ is of index 2 (corresponding to the index 1 critical point $x$ of $f_0$), then the Morse step through $x$ appears in the boundary of $\overline{\mathcal{M}}_{{\rm gr}}((x,0),X\times[0,1/2[\cup Y\times[1/2,1[)$ as the boundary part $\overline{\mathcal{M}}((x,0),M\times\{0\})$. The moduli space formed of breaks after the graft give us a concatenation of $f_1$-steps in $Y\times\{1\}$. We may then define a map which to the Morse step through $x$ gives us the concatenation of $f_1$-steps that appear in the boundary part $\mathcal{M}_{{\rm gr}}((x,0),Y\times\{1\})$. 
Like in the non-grafted case, for any minimum of $f_0$, there is a unique grafted trajectory from that minimum to a minimum of $f_1$. We may then do the same construction as before, as in particular the same concatenation rules apply as before, and we get a group morphism:
\[\tilde\phi^{\rm gr}_H: \mathcal{L}(f_0, \ast_0) \rightarrow \mathcal{L}(f_1,\ast_1) \rightarrow \pi_1^{\text{Morse}}(f_1,\phi_H^{\rm gr}(\ast_0)),\]
where $(\phi_H^{\rm gr}(\ast_0),1)\in Y\times\mathbb{R}$ is the minimum reached by the grafted trajectory from $(\ast_0,0)\in X\times\mathbb{R}$. 

\begin{theorem}\label{thm:quotient_grafted}
Let $(f_0,g_0)$ be a Morse-Smale pair on a manifold $X$ and $(f_1,g_1)$ be a Morse-Smale pair on a manifold $Y$, and $(f_0,f_1, H)$ a set of grafted interpolation data between $f_0$ and $f_1$, and metrics on $X$ and $Y$ such that the graph of $H$ intersects $W^u(x)\times W^s(y)$ transversally for any $x\in {\rm Crit}(f_0),y\in{\rm Crit}(f_1)$. Then $\tilde\phi^{\rm gr}_H$ descends to the quotient and induces a morphism $$\phi^{\rm gr}_H: \pi_1^{\text{Morse}}(f_0,\ast_0)\rightarrow \pi_1^{\text{Morse}}(f_1,\phi_H^{\rm gr}(\ast_0)).$$ 
\end{theorem}
\begin{proof}
We must show that $\mathcal{R}(f_0,\ast_0)\subset\ker(\tilde\phi^{\rm gr}_H)$. Take a loop $\ell$ which is the boundary of the unstable manifold of an index 2 critical point $z$ of $f_0$. This point corresponds to an index 3 critical point $(z,0)$ of $f_0$. The proof follows a similar structure to that of the proof of Theorem \ref{thm:quotient}: we will show that the boundary part $\mathcal{B}_{\rm gr}((z,0),Y\times\{1\})$ of $\mathcal{M}_{\rm gr}((z,0),X\times[0,1/2[\cup Y\times[1/2,1[)$, formed of broken trajectories which end (and therefore break) in $Y\times\{1\}$, is a disk, and identify the boundary of said disk with $\phi^{\rm gr}_H(\ell)$ (the grafted analog of Lemma \ref{lem:quotient}). We will then use this identification to prove that $\phi^{\rm gr}_H(\ell)$ must be trivial.

Let us look at $\overline{\mathcal{M}}_{\rm gr}((z,0),X\times[0,1/2[\cup Y\times[1/2,1[)$. It is a topological closed disk of dimension 3. Its boundary is therefore a sphere. The boundary part at $u=1$ is the moduli space $\overline{\mathcal{M}}((z,0),X\times\{0\})$ and is therefore a disk. Furthermore, each "slice" $\mathcal{M}_{\rm gr}((z,0),\{u\})$ is of the form $\overline{\mathcal{M}}((z,0),X\times\{0\})\times\{u\}$ if $0<u<1/2$, and 
$\overline{\mathcal{M}}_{\rm gr}((z,0),\{0\}))\times\{u\}\times\{v\}$ when $1/2\leq u<1$ (where $v$ depends on $u$) and are therefore disks. One can then deduce that $\overline{\mathcal{M}}_{\rm gr}((z,0),X\times]0,1/2[\cup Y\times[1/2,1[)$ is a 3-dimensional filled cylinder which goes from $u=0$ to $u=1$. The boundary at $u=1$ must therefore be a disk, as its complement in the sphere is a disk, and of which its own boundary can be identified with $\tilde{\phi}_H^{\rm gr}(\ell)$ by construction. The grafted analog of Lemma \ref{lem:quotient} therefore holds.

Since the moduli spaces involved have the same topology and properties as in the non-grafted case, the arguments from the proof of Theorem \ref{thm:quotient} hold, with the only change being that we replace moduli spaces of the form $\overline{\mathcal{M}}((p,0),(p,1))$ and $\overline{\mathcal{M}}((p,0),M\times[0,1])$ with respectively $\overline{\mathcal{M}}_{\rm gr}((p,0),(p,1))$ and $\overline{\mathcal{M}}_{\rm gr}((p,0),X\times[0,1/2[\cup Y\times[1/2,1[)$, and {\rm ev} by ${\rm ev_{gr}}$. We also use ${\rm ev}_{{\rm gr},f_1}$ instead of ${\rm ev}_{f_1}$, and for trajectories that end in $X\times\{0\}$, we use the diffeomorphism $\overline{\mathcal{M}}((p,0),X\times\{0\})\simeq\overline{\mathcal{M}}(p,X)$ (the moduli space on the right hand side referring to anti-gradient trajectories of $f_0$) and the map ${\rm ev}_{f_0}$ in the same way as in the non-grafted case, which allows us to identify $\ell$ in the boundary of $\overline{\mathcal{M}}_{\rm gr}((z,0),\{0\})$ and $\tilde{\phi}_H^{\rm gr}(\ell)$ in the boundary of $\mathcal{B}_{\rm gr}((z,0),Y\times\{1\})$. These arguments allow us to conclude that $\phi_H^{\rm gr}(\ell)$ is trivial, which concludes our proof.
\end{proof}

\section{Functoriality of the continuation map}\label{section:functoriality}
\subsection{Three pairs of Morse data on a same manifold}\label{subsection:functoriality}
We want to show the functoriality of $\phi$. For this, let's take a third Morse-Smale pair  $(f_2,g_2)$ on $M$. We fix three interpolation pairs $(F_{01}, G_{01})$, $(F_{12}, G_{12})$ and $(F_{02}, G_{02})$ on $M\times\mathbb{R}$, which allow us to define $\phi_{01}: \pi_1^\text{Morse}(f_0) \rightarrow \pi_1^\text{Morse}(f_1)$, $\phi_{12}: \pi_1^\text{Morse}(f_1) \rightarrow \pi_1^\text{Morse}(f_2)$, and $\phi_{02}: \pi_1^\text{Morse}(f_0) \rightarrow \pi_1^\text{Morse}(f_2)$. We will show that $\phi_{12}\circ\phi_{01} = \phi_{02}\circ id = \phi_{02}$. 

Let us now look at a Morse function $\tilde F:M\times [0,1]\times[0,1] \rightarrow \mathbb{R}$, such that there exists $\varepsilon>0$ such that:
\begin{itemize}
\item $\tilde{F}|_{M\times[0,1]\times[0,\varepsilon]}\equiv f_0$,
\item $\tilde{F}|_{M\times[0,\varepsilon]\times[0,1]}\equiv \tilde{F}_{01}$,
\item $\tilde{F}|_{M\times[1-\varepsilon,1]\times[0,1]}\equiv \tilde{F}_{02}$,
\item $\tilde{F}|_{M\times[0,1]\times[1-\varepsilon,1]}\equiv \tilde{F}_{12}$,
\item $F(\cdot, s, s') = \tilde{F}+C(h(s)+h(s'))$, for constant $C>0$, such that
$$\frac{\partial \tilde{F}}{\partial s}(x,s,s')+Ch'(s)<0, \forall (x,s,s')\in M\times]0,1[\times[0,1],$$ and
$$\frac{\partial \tilde{F}}{\partial s'}(x,s,s')+Ch'(s')<0, \forall (x,s,s')\in M\times[0,1]\times]0,1[.$$
\end{itemize}
We also pick a metric $G$ such that the pair $(F,G)$ is Morse-Smale, and such that it matches $g_0$ on $M\times[0,1]\times[0,\varepsilon]$, matches $G_{01}$ on $[0,\varepsilon]\times[0,1]$, matches $G_{02}$ on $[1-\varepsilon,1]\times[0,1]$, and matches $G_{12}$ on $[0,1]\times[1-\varepsilon,1]$. We call such an $(F,G)$ an interpolation square pair. This type of set up can also be found in \cite{CR, AD}. See Figure \ref{fig:F_square} for a schematic illustration of this setting. 

The only critical points of $F$ are critical points of $f_0$ over $M\times(\{(0,0)\}\cup\{(1,0)\}$, $f_1$ over $M\times\{(0,1)\}$, and $f_2$ over $M\times\{(1,1)\}$. Moreover, we get:
\begin{align*}{\rm Crit}_{\ast}(F) =&\ {\rm Crit}_{\ast-2}(f_2)\times\{1\}\times\{1\} \cup {\rm Crit}_{\ast-1}(f_1)\times\{1\}\times\{0\} \\&\ \cup {\rm Crit}_{\ast-1}(f_0)\times\{0\}\times\{1\} \cup{\rm Crit}_{\ast}(f_0)\times\{0\}\times\{0\}.\end{align*}
In particular, $F$'s only minimums correspond to minimums of $f_2$ (and $F$'s only maximums correspond to maximums of $f_0$).

\begin{figure}
\labellist
\small\hair 2pt
\pinlabel {$F_{01}$} [bl] at 75 130
\pinlabel {$F_{02}$} [bl] at 347 130
\pinlabel {$F_{12}$} [bl] at 215 260
\pinlabel {$f_0$} [bl] at 220 00
\pinlabel {$f_1$} [bl] at 100 260
\pinlabel {$f_2$} [bl] at 330 260
\endlabellist
  \centerline{\includegraphics[width=9cm]{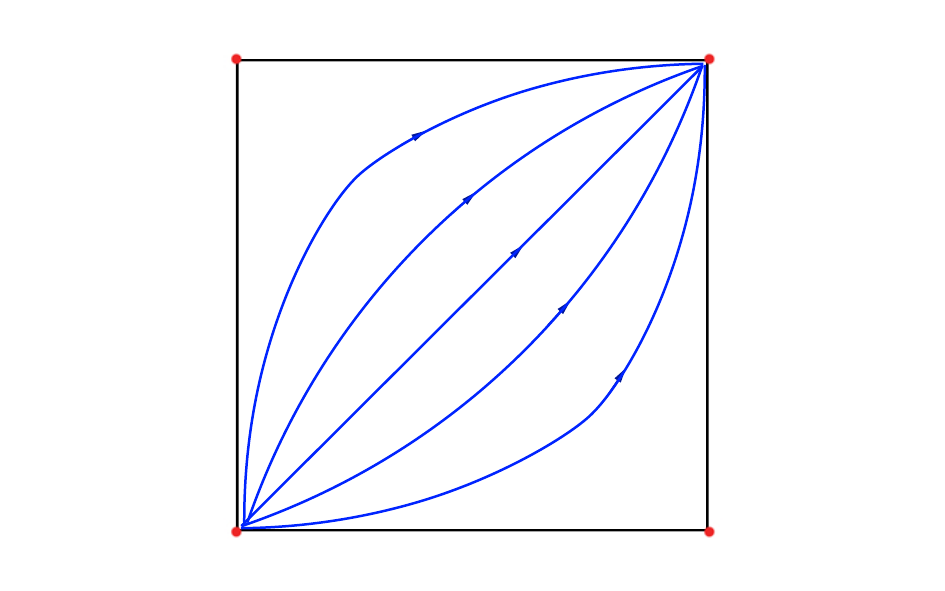}}
  \caption{Some of the anti-gradient lines of $h(s)+h(s')$ on $[0,1]\times[0,1]$.}
  \label{fig:F_square}
\end{figure} 

Looking at trajectories from critical points in $M\times\{0\}\times\{0\}$ to critical points in $M\times\{0\}\times\{1\}$ gives us $\phi_{01}$, whereas looking at trajectories from critical points in $M\times\{0\}\times\{1\}$ to $M\times\{1\}\times\{1\}$ gives us $\phi_{12}$, and looking at trajectories from critical points in $M\times\{1\}\times\{0\}$ to $M\times\{1\}\times\{1\}$ gives us $\phi_{02}$. Also, looking at trajectories in $M\times[0,1]\times\{0\}$ gives us the identity. 

We will be interested in trajectories from points of the form $(x,0,0)$ to points of the form $(x',1,1)$, where $x\in{\rm Crit}_1(f_0)$ and $x'\in{\rm Crit}_1(f_2)$. 
Call $\mathcal{M}(x)$ the space of such trajectories: \[\mathcal{M}(x) = \bigcup_{x'\in{\rm Crit}_1(f_2)} \mathcal{M}((x,0,0),(x',1,1)).\] It is a one dimensional moduli space. 
Its compactification is obtained by adding broken trajectories. 

$\mathcal{M}(x)$ appears in the boundary of a bigger moduli space $$\widetilde{\mathcal{M}}(x)=\bigcup_{y\in{\rm Crit}_0(f_2)} \mathcal{M}((x,0,0),(y,1,1)),$$ which is of dimension 2. We may parametrize this moduli space by two coordinates, $\alpha$ and $\lambda$: first, $\lambda\in]0, \pi/2[$, the angle in $[0,1]\times[0,1]$ of the projection of the tangent vector to the trajectory at $-\infty$, where $\lambda=\varepsilon$, with $\varepsilon$ small enough, corresponds to when the beginning of the trajectory is close to $M\times[0,1]\times\{0\}$ and $\lambda=\pi/2-\varepsilon$ corresponds to when the beginning of the trajectory is close to $M\times\{0\}\times[0,1]$. We call $\widetilde{\mathcal{M}}_\lambda(x)$ the moduli space of trajectories in $\widetilde{\mathcal{M}}(x)$ which begin in an angle $\lambda$. Second, $\alpha\in]0,\pi[$ is the angle in the restriction of the unstable manifold of $(x,0,0)$ to points reached by trajectories in $\widetilde{\mathcal{M}}_\lambda(x)$ of the tangent vector to the trajectory at $-\infty$. More specifically, we look at the moduli space $\widetilde{\mathcal{M}}_\lambda(x,M)$ of trajectories $\gamma$ which begin in $x$ at an angle $\lambda$ (in the same sense as before) and such that $\gamma(0)\in M\times[0,1]^2$. This moduli space is a two dimensional half disk, and every trajectory in $\widetilde{\mathcal{M}}_\lambda(x)$ is represented by a one-dimensional family of trajectories in $\widetilde{\mathcal{M}}_\lambda(x,M)$. For every trajectory in this family, the tangent vector to the trajectory at $-\infty$ is the same (up to multiplication). We then define the angle $\alpha\in]0,\pi[$ of a trajectory in $\widetilde{\mathcal{M}}_\lambda(x)$ as the angle of this vector in $\widetilde{\mathcal{M}}_\lambda(x,M)$. We make the following orientation choice: for every $x\in\text{Crit}_1(f_0)$, if we are considering a step through $x$ which begins at a broken trajectory $(\alpha_1,\beta_1)$ and ends at a broken trajectory $(\alpha_2,\beta_2)$, then $\alpha_1$ and $\alpha_2$ still exist as trajectories from $(x,0,0)$ which stay in $M\times\{0\}\times\{0\}$, and we say that for every $\lambda\in]0, \pi/2[$, trajectories in $\widetilde{\mathcal{M}}_\lambda(x)$ have an angle $\alpha$ which is close to 0 if they are near $\alpha_1$, and close to $\pi$ if they are near $\alpha_2$. Every trajectory in $\widetilde{\mathcal{M}}(x)$ is uniquely determined by $\lambda$ and $\alpha$. See Figure \ref{fig:angles}. 

\begin{figure}
\begin{subfigure}{0.5\textwidth}
\labellist
\small\hair 2pt
\pinlabel {$f_0$} [bl] at 220 30
\pinlabel {$f_1$} [bl] at 100 290
\pinlabel {$f_2$} [bl] at 330 290
\endlabellist
  \centerline{\includegraphics[width=8cm]{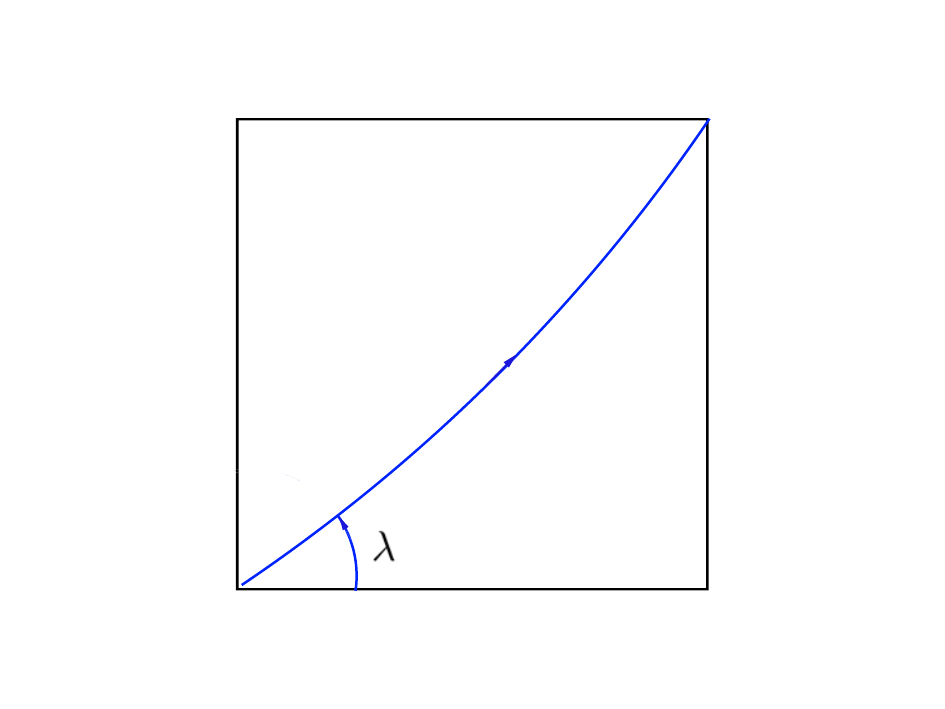}}
  \caption{The angle $\lambda$ of a trajectory.}
  \label{fig:lambda}
\end{subfigure}
\begin{subfigure}{0.5\textwidth}
\labellist
\small\hair 2pt
\pinlabel {$(x,0,0)$} [bl] at 60 185
\pinlabel {$(y_1,0,0)$} [bl] at 50 290
\pinlabel {$(y_2,0,0)$} [bl] at 50 30
\pinlabel {$(m_1,1,1)$} [bl] at 310 305
\pinlabel {$(m_2,1,1)$} [bl] at 310 20
\endlabellist
  \centerline{\includegraphics[width=8cm]{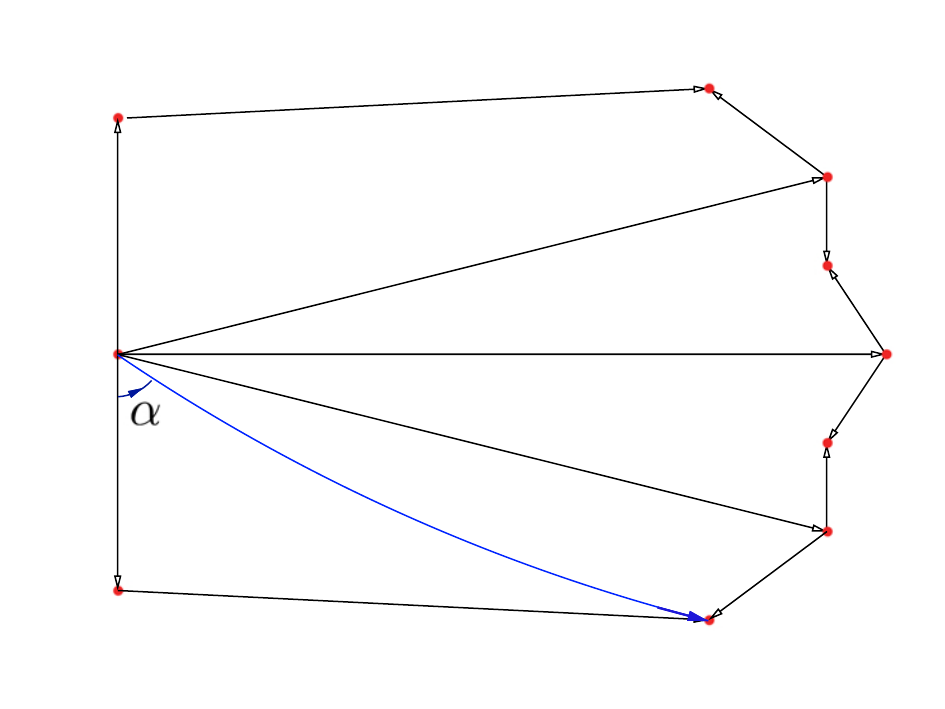}}
  \caption{The angle $\alpha$ of a trajectory at angle $\lambda$.}
  \label{fig:alpha}
\end{subfigure}
\begin{subfigure}{\textwidth}
\labellist 
\small\hair 2pt
\pinlabel {$\lambda$} [bl] at 400 60
\pinlabel {$\alpha$} [bl] at 60 300
\pinlabel {$0$} [bl] at 35 35
\pinlabel {$\pi/2$} [bl] at 375 25
\pinlabel {$\pi$} [bl] at 30 275
\endlabellist
  \centerline{\includegraphics[width=8cm]{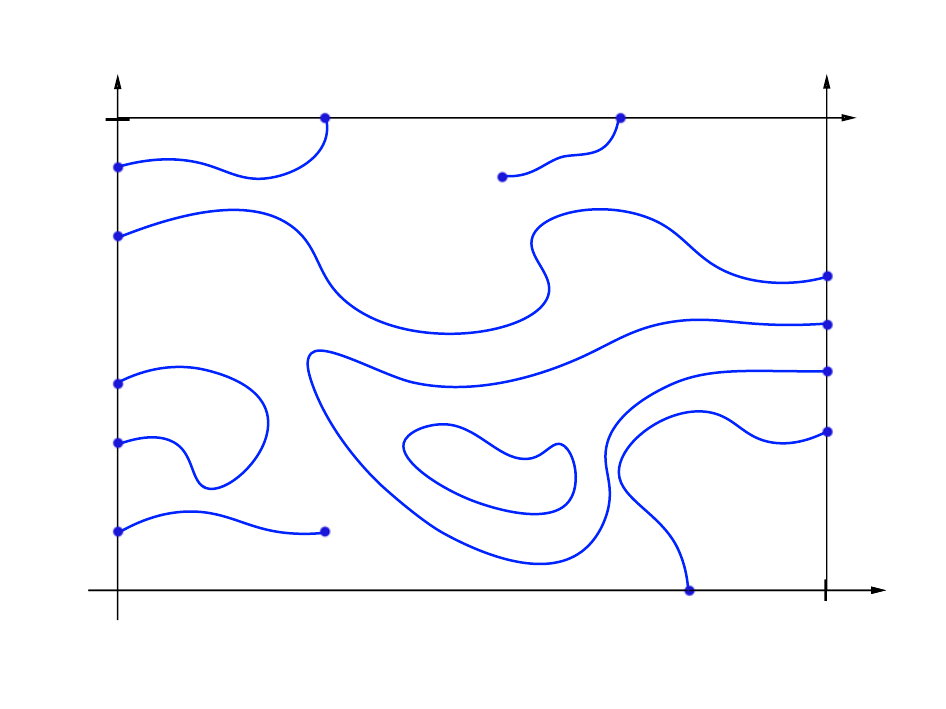}}
  \caption{$\mathcal{M}(x)$ in the boundary of $\widetilde{\mathcal{M}}(x)$, parametrized by $\lambda$ and $\alpha$.}
  \label{fig:lambda_alpha}
\end{subfigure}
\caption{The angles $\lambda$ and $\alpha$ defining elements of $\mathcal{M}(x)$.}
\label{fig:angles}
\end{figure} 

If we allow $\lambda=0,\pi/2$, and $\alpha=0,\pi$, we may retrieve all broken trajectories in the compactification of $\widetilde{\mathcal{M}}(x)$. These broken trajectories are however not uniquely determined by $\lambda$ and $\alpha$, but $\lambda$ and $\alpha$ do determine the behavior of the trajectory up to the first break. This implies that trajectories in $\mathcal{M}(x)$ are also uniquely determined by $\lambda$ and $\alpha$. We call $\mathcal{M}_\lambda(x)$ the moduli space of trajectories in $\mathcal{M}(x)$ which start at an angle $\lambda$. For every $\lambda$, associating to the step through $x$ in $M$, the concatenation of $f_2$-steps given by index 1 critical points of $f_2$ reached by trajectories in $\mathcal{M}_\lambda(x)$, gives us a map $\tilde{\phi}_\lambda$ from the space of $f_0$-steps to the space of $f_2$-paths.

\begin{prop}\label{prop:generic_lambda}
For generic values of $\lambda\in]0,\pi/2[$, $\tilde{\phi}_\lambda$ induces a well defined morphism $$\phi_\lambda: \pi_1^\text{Morse}(f_0,\ast_0) \rightarrow \pi_1^\text{Morse}(f_2,\phi_\lambda(\ast_0)),$$ where $\ast_0\in{\rm Crit}_0(f_0)$, and $(\phi_\lambda(\ast_0),1,1)$ is the other end of the unique trajectory from $(\ast_0,0,0)$ of angle $\lambda$. 
\end{prop}

\begin{proof}
Let us start by remarking that for $x\in{\rm Crit}_1(f_0)$, for generic choices of $\lambda$, the space of trajectories (modulo reparametrisation) from $(x,0,0)$ of angle $\lambda$ is a one dimensional moduli space, and if $y\in{\rm Crit}_0(f_0)$, the space of trajectories (modulo reparametrisation) from $(y,0,0)$ of angle $\lambda$ is a zero dimensional moduli space. Moreover, there is a unique trajectory from $(y,0,0)$ of angle $\lambda$, and a finite amount of trajectories from points of the form $(x,0,0)\in{\rm Crit}_3(F)$ to points of the form $(x',1,1)\in{\rm Crit}_1(F)$ of angle $\lambda$. This tells us that the image of the Morse step through $x$ by $\phi_\lambda$ is a $f_2$-step, and that this map is compatible with concatenation. We therefore get a morphism $\phi_\lambda: \mathcal{L}(f_0,\ast) \rightarrow \pi_1^\text{Morse}(f_2,\phi_\lambda(\ast)).$

Since for generic values of $\lambda$, the restriction of $(F,G)$ to the part of $M\times[0,1]\times[0,1]$ where the angle is $\lambda$ is Morse-Smale, and the proof that this morphism descends to the quotient is analogous to that of Theorem \ref{thm:quotient}, with the only change being that we replace every moduli space involved in the proof with one of trajectories starting at angle $\lambda$. \end{proof}

The case where trajectories break in points $(x'', 1, 0)$, where $x''\in{\rm Crit}_1(f_0)$, corresponds to the case where $\lambda = \pi/2$, and so $\phi_{\pi/2}=\phi_{12}\circ\phi_{01}$. The case where trajectories break in $(x'', 0, 1)$, where $x''\in{\rm Crit}_1(f_1)$, corresponds to $\lambda=0$, and so $\phi_{0}=\phi_{02}\circ id=\phi_{02}$. We want to understand how $\phi_\lambda$ changes as $\lambda$ goes from $0$ to $\pi/2$. We will prove the following statement: 

\begin{prop}\label{prop:diag_lambda}
For any non-generic value $\lambda\in]0,\pi/2[$, and $\ast_0\in\text{Crit}_0(f_0)$, there exists a small enough $\varepsilon$ and an isomorphism $\psi_{\lambda}$ such that the following diagram is commutative:
\[\begin{tikzcd}[column sep=5em, row sep=3em]
\pi^{\text{Morse}}_1(f_0,\ast_0) \arrow[r,"\phi_{\lambda-\varepsilon}"] \arrow[rd, "\phi_{\lambda+\varepsilon}"']
& \pi^{\text{Morse}}_1(f_2,\phi_{\lambda-\varepsilon}(\ast_0)) \arrow[d,"\psi_\lambda"] \\
&  \pi^{\text{Morse}}_1(f_2,\phi_{\lambda+\varepsilon}(\ast_0)).
\end{tikzcd}\]
More precisely, $\psi_{\lambda}$ is the conjugation by the $f_2$-Morse path $\sigma_\lambda$ from $\phi_-(\ast_0)$ to $\phi_{+}(\ast_0)$ which appears in the boundary of the unstable manifold of $(\ast_0,0,0)$ between $\lambda-\varepsilon$ and $\lambda+\varepsilon$, ie $$\psi_\lambda(\ell_-)=\sigma_\lambda \ell_- \sigma_\lambda^{-1},$$ where $\ell_-$ is a $\phi_-(\ast_0)$ based Morse loop.
\end{prop}
\begin{proof}There are two types of non generic points, which we can see in Figure \ref{fig:lambda_alpha}: vertical tangencies, which correspond to birth/deaths of pairs of trajectories, and breaking points, which correspond to when the trajectory at an angle goes to an index 2 critical point of $F$, either in $M\times\{0\}\times\{0\}$ or in $M\times\{1\}\times\{1\}$. 

First, let's look at what happens at vertical tangencies. Suppose that we have a one parameter family of trajectories in $\mathcal{M}(x)$ which admits a vertical tangency at $\lambda$. At $\lambda$, there is either a death or a birth of a trajectory that then splits into two. Suppose we are in the case of a death. Then at $\lambda-\varepsilon$, there exist two trajectories from $(x,0,0)$ to $(x',1,1)$. These trajectories then merge at $\lambda$, and cease to exist at $\lambda+\varepsilon$. In particular, the fact that they merge tells us that they each give Morse steps which are part of the same connected component of $\mathcal{M}((x,0,0),(x',1,1))$. So these steps must be successive Morse steps with opposite orientations: one contributes the step $\sigma_{x'}$ to the image of $\sigma_x$, whereas the other contributes $\sigma_{x'}^{-1}$. We therefore get $$\tilde{\phi}_{\lambda-\varepsilon}(\sigma_x)=\sigma_{x_1}…\sigma_{x_i}\sigma_{x'}\sigma_{x'}^{-1}\sigma_{x_{i+1}}…\sigma_{x_n},$$ then $$\tilde{\phi}_{\lambda+\varepsilon}(\sigma_x)=\sigma_{x_1}…\sigma_{x_i}\sigma_{x_{i+1}}…\sigma_{x_n}.$$ With the word reduction rule of $\mathcal{L}(f_2,\ast)$, these therefore define the same Morse loops. By symmetry, the case of a birth takes us from $$\tilde{\phi}_{\lambda-\varepsilon}(\sigma_x)=\sigma_{x_1}…\sigma_{x_i}\sigma_{x_{i+1}}…\sigma_{x_n}$$ to $$\tilde{\phi}_{\lambda+\varepsilon}(\sigma_x)=\sigma_{x_1}…\sigma_{x_i}\sigma_{x'}\sigma_{x'}^{-1}\sigma_{x_{i+1}}…\sigma_{x_n}.$$ 

Now let's look at breaking points, which, other than the ones previously mentioned, come in two types: a trajectory in $\mathcal{M}_\lambda(x)$ can break in $(y,0,0)$, with $y\in {\rm Crit}_0(f_0)$, or can break in a point $(z,1,1)$, with $z\in{\rm Crit}_2(f_2)$. 

The case where the breaking happens in $(z,1,1)$, with $z\in{\rm Crit}_2(f_2)$ corresponds to the $\lambda\in]0,\pi/2[, \alpha\in]0,\pi[$ case. It is depicted in Figure \ref{fig:break_z}.  Let us look at a family of generic trajectories which goes from $(x,0,0)$ to $(x',1,1)$ before $\lambda$ and breaks in $(z,1,1)$ in $\lambda$. Let us compare what happens at $\lambda-\varepsilon$ and $\lambda+\varepsilon$ for some $\varepsilon>0$. Let $\sigma_x$ be the Morse-step formed by the Latour cell of $x\in M$. Suppose $$\tilde{\phi}_{\lambda-\varepsilon}(\sigma_x)=\sigma_{x_1}…\sigma_{x_n},$$ for $x_1, …, x_n \in {\rm Crit}_1(f_2)$. Suppose $x_i,…, x_{i+j}$ are index 1 critical points of $f_1$ reached by trajectories starting at $z$  (see Figure \ref{fig:break_z_1}). Then at $\lambda$, the trajectories to $x_i, …, x_{i+j}$ break in $z$. We get a unique trajectory from $(x,0,0)$ to $(z,1,1)$, then $j+1$ rigid trajectories from $(z,1,1)$. In fact, every trajectory from $z$ gets glued to the trajectory from $(x,0,0)$ to $(z,1,1)$, and so at $\lambda$ we get broken trajectories from $(x,0,0)$ to each index 1 critical point which appears in the boundary of the unstable manifold of $z$ (see Figure \ref{fig:break_z_2}). These broken trajectories are therefore each either the end of a family of trajectories from $(x,0,0)$ to these points which exists before $\lambda$ and ceases to exist after $\lambda$, in which case the steps through these points are in the image of $\sigma_x$ by $\phi_{\lambda-\varepsilon}$, or the beginning of a family of trajectories from $(x,0,0)$ to these points which does not exist before $\lambda$ and starts to exist after $\lambda$, in which case they are in the image of $\sigma_x$ by $\phi_{\lambda+\varepsilon}$. In particular, $\sigma_{x_i}…\sigma_{x_{i+j}}$ are no longer in the image of $\sigma_x$ by $\phi_{\lambda+\varepsilon}$. Let us call $m_1$ (resp. $m_k$) the minimum of $f_1$ which is the beginning of the Morse path $\sigma_{x_i}…\sigma_{x_{i+j}}$. Then the moduli spaces $$\mathcal{M}((x,0,0),(z,1,1))\times\overline{\mathcal{M}}((z,1,1),(m_1,1,1))$$ and $$\mathcal{M}((x,0,0),(z,1,1))\times\overline{\mathcal{M}}((z,1,1),(m_k,1,1))$$ are two compact one-dimensional manifolds with boundary which together separate the boundary of the moduli space $$\mathcal{M}((x,0,0),(z,1,1))\times\overline{\mathcal{M}}((z,0,0),M\times\{1\}\times\{1\})$$ into two connected components (since this moduli space is a disk and therefore has for boundary $\mathbb{S}^1$). Then since at $\lambda-\varepsilon$ and $\lambda+\varepsilon$, the boundary part $\overline{\mathcal{M}}_\lambda((x,0,0),M\times\{1\}\times\{1\})$ forms a Morse path, at either angle, one either goes around one side of the Latour cell of $(z,0,0)$ or the other. Therefore, $\sigma_{x_i},…, \sigma_{x_{i+j}}$ are all the steps along one of these sides of the Latour cell of $(z,0,0)$. Let $\tilde{x}_1, …, \tilde{x}_m$ be the other index 1 points on the other side of the boundary of the unstable manifold of $z$ (in order of appearance of the associated steps in the boundary of the Latour cell of $z$). We get $$\tilde{\phi}_{\lambda+\varepsilon}(\sigma_x)= \sigma_{x_1}…\sigma_{x_{i-1}}\sigma_{\tilde{x}_1}…\sigma_{\tilde{x}_m}\sigma_{x_{i+j+1}}…\sigma_{x_n}.$$ See Figure \ref{fig:break_z_3}. If $r$ is the concatenation of steps which defines the relation given by $z$, then $r=(\sigma_{x_i}…\sigma_{x_{i+j}})^{-1}\sigma_{\tilde{x}_1}…\sigma_{\tilde{x}_m}$, and therefore, $$\tilde{\phi}_{\lambda+\varepsilon}(\sigma_x)= \sigma_{x_1}…\sigma_{x_{i-1}}\sigma_{x_i}…\sigma_{x_{i+j}}r\sigma_{x_{i+j+1}}…\sigma_{x_n},$$ and so once we pass to the quotient, this step creates the same loops as the step $\phi_{\lambda-\varepsilon}(\sigma_x)$.

\begin{figure}
\begin{subfigure}{\textwidth}
\labellist 
\small\hair 2pt
\pinlabel {$x$} [bl] at 65 110
\pinlabel {$y_1$} [bl] at 65 35
\pinlabel {$y_2$} [bl] at 65 180
\pinlabel {$x_1$} [bl] at 335 125
\pinlabel {$x_2$} [bl] at 350 225
\pinlabel {$x_3$} [bl] at 415 190
\pinlabel {$x_4$} [bl] at 425 145
\pinlabel {$x_5$} [bl] at 415 100
\pinlabel {$x_6$} [bl] at 365 75
\pinlabel {$z$} [bl] at 383 177
\endlabellist
  \centerline{\includegraphics[width=9cm]{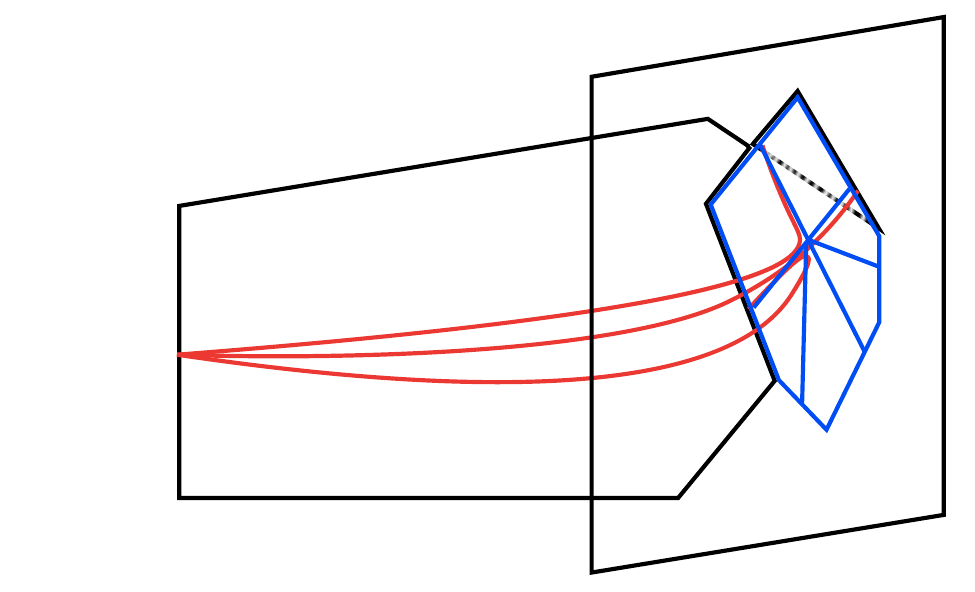}}
  \caption{Trajectories from $(x,0,0)$ at $\lambda-\varepsilon$.}
  \label{fig:break_z_1}
\end{subfigure}
\begin{subfigure}{\textwidth}
\labellist 
\small\hair 2pt
\pinlabel {$x$} [bl] at 65 110
\pinlabel {$y_1$} [bl] at 65 35
\pinlabel {$y_2$} [bl] at 65 180
\pinlabel {$x_1$} [bl] at 335 125
\pinlabel {$x_2$} [bl] at 350 225
\pinlabel {$x_3$} [bl] at 415 190
\pinlabel {$x_4$} [bl] at 425 145
\pinlabel {$x_5$} [bl] at 415 100
\pinlabel {$x_6$} [bl] at 365 75
\pinlabel {$z$} [bl] at 383 177
\endlabellist
  \centerline{\includegraphics[width=9cm]{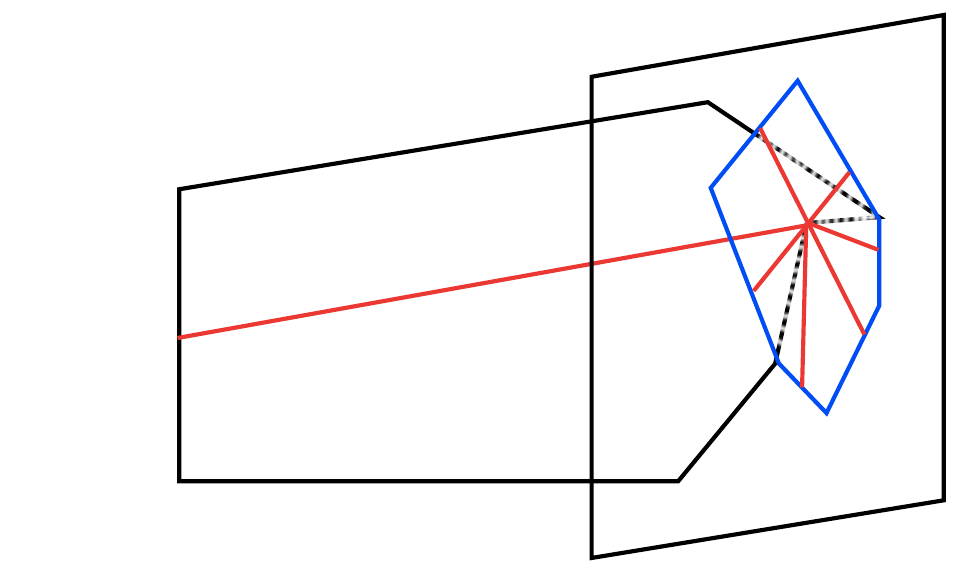}}
  \caption{Broken trajectories from $(x,0,0)$ at $\lambda$.}
  \label{fig:break_z_2}
\end{subfigure}
\labellist 
\small\hair 2pt
\pinlabel {$x$} [bl] at 65 110
\pinlabel {$y_1$} [bl] at 65 35
\pinlabel {$y_2$} [bl] at 65 180
\pinlabel {$x_1$} [bl] at 335 125
\pinlabel {$x_2$} [bl] at 350 225
\pinlabel {$x_3$} [bl] at 415 190
\pinlabel {$x_4$} [bl] at 425 145
\pinlabel {$x_5$} [bl] at 415 100
\pinlabel {$x_6$} [bl] at 365 75
\pinlabel {$z$} [bl] at 383 177
\endlabellist
\begin{subfigure}{\textwidth}
  \centerline{\includegraphics[width=9cm]{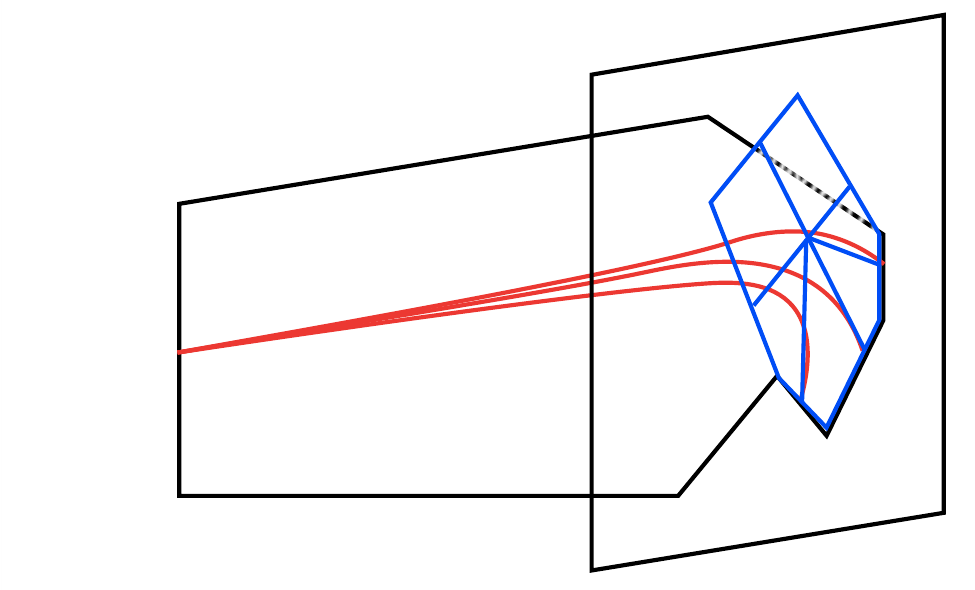}}
  \caption{Trajectories from $(x,0,0)$ at $\lambda+\varepsilon$.}
  \label{fig:break_z_3}
\end{subfigure}
\caption{A family of trajectories which breaks in $\lambda$, with $\lambda\in]0,\pi/2[,\alpha\in]0,\pi[$.}
\label{fig:break_z}
\end{figure} 

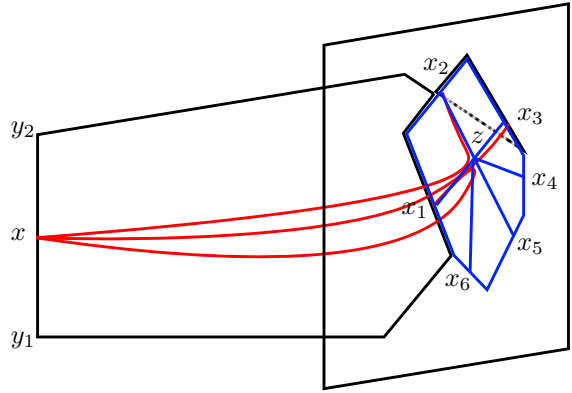
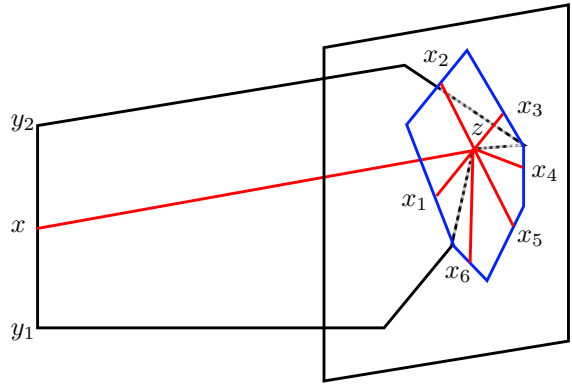
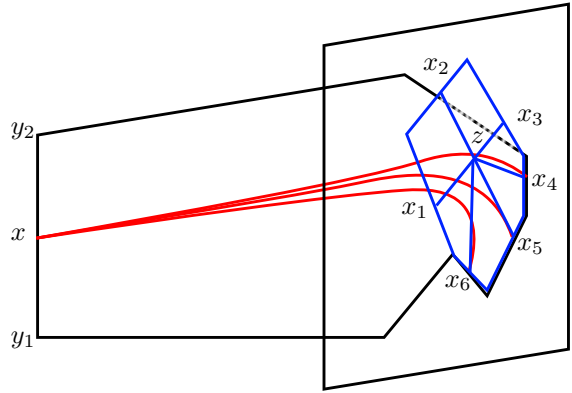
\begin{figure}
\begin{subfigure}{\textwidth}
\labellist 
\small\hair 2pt
\pinlabel {$x$} [bl] at 35 160
\pinlabel {$y_1$} [bl] at 35 35
\pinlabel {$y_2$} [bl] at 30 230
\pinlabel {$x'$} [bl] at 365 125
\pinlabel {$x_2$} [bl] at 350 195
\pinlabel {$y'$} [bl] at 385 65
\pinlabel {$y''$} [bl] at 398 160
\endlabellist
  \centerline{\includegraphics[width=7cm]{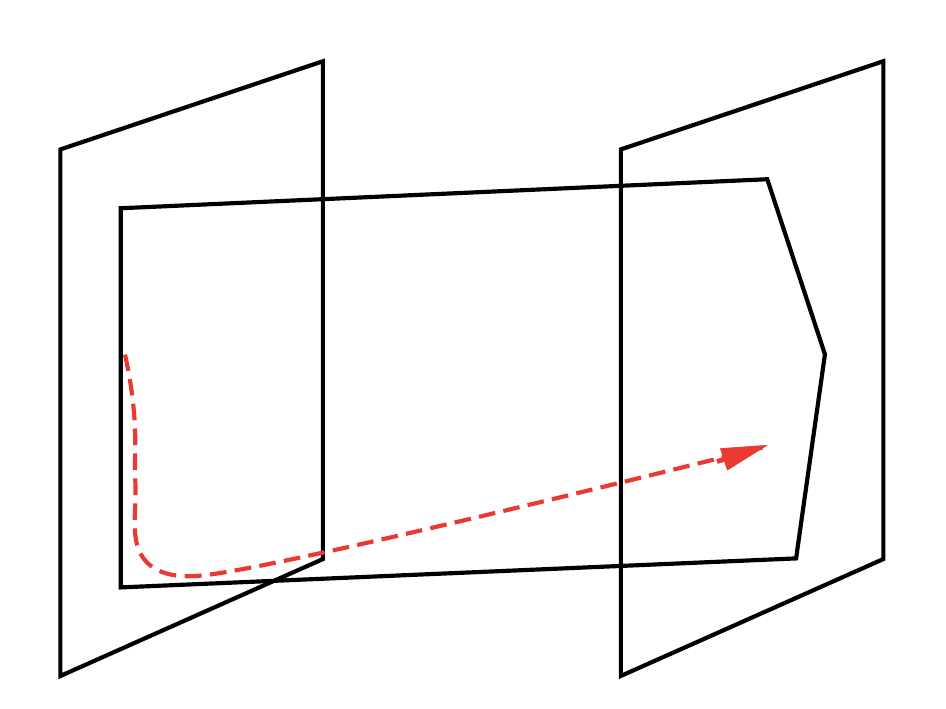}}
  \caption{A trajectory from $(x,0,0)$ at $\lambda-\varepsilon$.}
  \label{fig:break_y_1}
\end{subfigure}
\begin{subfigure}{\textwidth}
\labellist 
\small\hair 2pt
\pinlabel {$x$} [bl] at 35 160
\pinlabel {$y_1$} [bl] at 35 35
\pinlabel {$y_2$} [bl] at 30 230
\pinlabel {$x'$} [bl] at 365 125
\pinlabel {$x_2$} [bl] at 350 195
\pinlabel {$y'$} [bl] at 385 65
\pinlabel {$y''$} [bl] at 398 160
\endlabellist
  \centerline{\includegraphics[width=7cm]{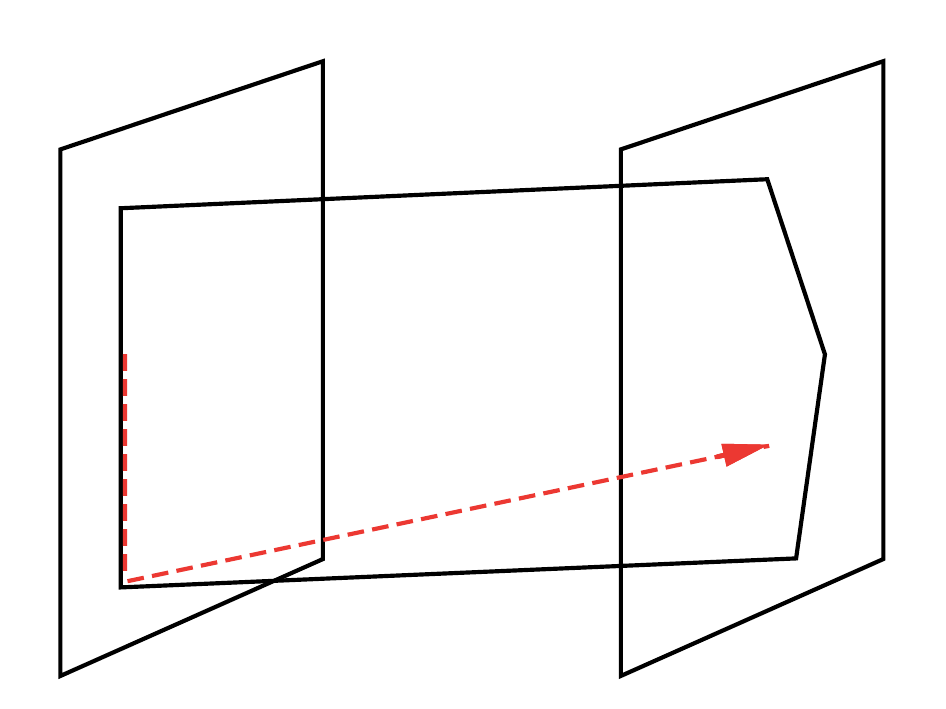}}
  \caption{A trajectory from $(x,0,0)$ at $\lambda$.}
  \label{fig:break_y_2}
\end{subfigure}
\begin{subfigure}{\textwidth}
\labellist 
\small\hair 2pt
\pinlabel {$x$} [bl] at 35 160
\pinlabel {$y_1$} [bl] at 35 35
\pinlabel {$y_2$} [bl] at 30 230
\pinlabel {$x_2$} [bl] at 350 210
\pinlabel {$y''$} [bl] at 398 160
\endlabellist
  \centerline{\includegraphics[width=7cm]{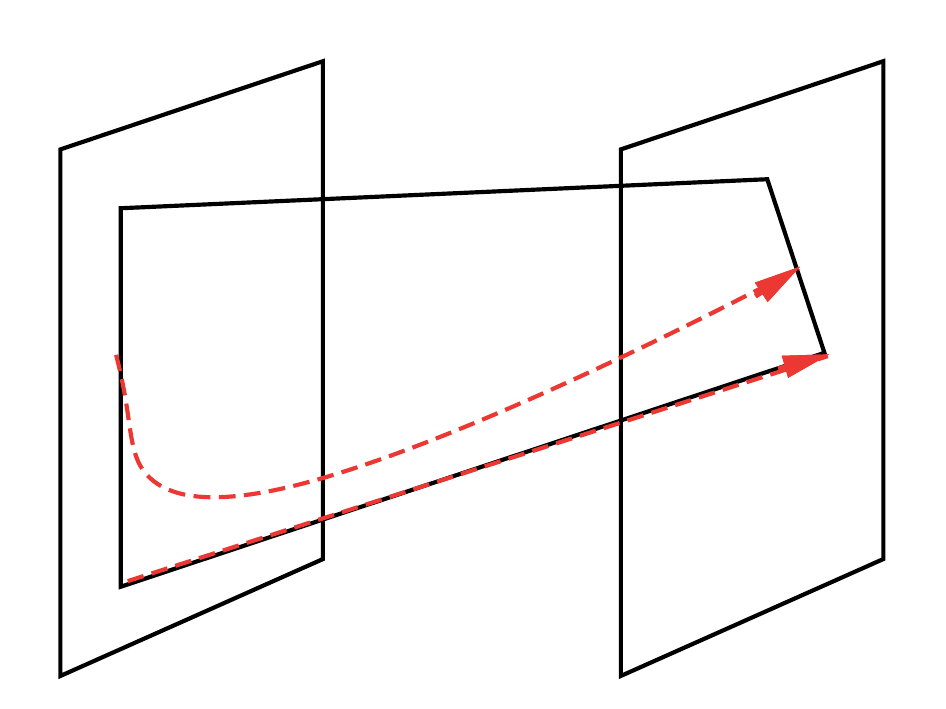}}
  \caption{A trajectory from $(x,0,0)$ at $\lambda+\varepsilon$.}
  \label{fig:break_y_3}
\end{subfigure}
\caption{A family of trajectories which breaks in $\lambda$, with $\lambda\in]0,\pi/2[,\alpha\in\{0,\pi\}$.}
\label{fig:break_y}
\end{figure} 

The case where the breaking happens in a point $(y,0,0)$, with $y\in{\rm Crit}_0(f_0)$, corresponds to the case where $\lambda\in]0,\pi/2[,\alpha\in\{0,\pi\}$. Consider a one-parameter family of trajectories from $(x,0,0)$ to $(x',1,1)$ which breaks in $(y,0,0)$ at $\lambda$. In particular, since there are only two such possible $y$'s, $y_1$ for $\alpha=0$ and $y_2$ for $\alpha=\pi$, then if $\alpha=0$, we get $$\tilde{\phi}_{\lambda-\varepsilon}(\sigma_x)=\sigma_{x'}\sigma_{x_2}…\sigma_{x_n}$$ (so $x'=x_1$), and for $\alpha=\pi$, we get $$\tilde{\phi}_{\lambda-\varepsilon}(\sigma_x)=\sigma_{x_1}…\sigma_{x_{n-1}}\sigma_{x'}$$ ($x'=x_n$) (see Figure \ref{fig:break_y_1}). Let us treat the case where $\alpha=0$, the other case being similar. The family of trajectories from $(x,0,0)$ to $(x',1,1)$ therefore breaks at $\lambda$ in $(y_1,0,0)$ (see Figure \ref{fig:break_y_2}). Since $(y_1,0,0)$ is of index 2, there is a unique trajectory at $\lambda-\varepsilon$ which goes from $(y_1,0,0)$ to an index 0 critical point $(y',1,1)$. At $\lambda$ however, we get a trajectory from $(y_1,0,0)$ to $(x',1,1)$. If $y''$ is the other extremity of $\sigma_{x'}$, we can glue the trajectory from $(y_1,0,0)$ to $(x',1,1)$ and the trajectory from $(x',1,1)$ to $(y'',1,1)$, which begins a one-parameter family (which depends on the angle) of trajectories from $(y_1,0,0)$ to $(y'',1,1)$. This family only begins after $\lambda$, because if such trajectories exist before $\lambda$ they have to cross the trajectory from $(x,0,0)$ to $(x',1,1)$. Therefore, at $\lambda+\varepsilon$, we can no longer have a trajectory from $(x,0,0)$ to $(x',1,1)$, because it would cross the trajectories from $(y_1,0,0)$ to $(y'',1,1)$. The trajectory from $(x,0,0)$ to $(x_2,1,1)$ nevertheless remains, and we get $$\tilde{\phi}_{\lambda+\varepsilon}(\sigma_x)=\sigma_{x_2}…\sigma_{x_n}$$ (see Figure \ref{fig:break_y_3}). (Similarly, in the case where $\alpha=\pi$, we get $\tilde{\phi}_{\lambda+\varepsilon}(\sigma_x)=\sigma_{x_1}…\sigma_{x_{n-1}}$). 

Crossing such a $\lambda$ therefore changes $\phi_{\lambda}$ at the Morse step level, but not at the Morse loop level: let $\tilde{x}\in{\rm Crit}_1(f_0)$ that also has $y_1\in\text{Crit}_0(f_0)$ in the boundary of its unstable manifold. $\sigma_{\tilde{x}}$ therefore precedes $\sigma_x$ in the concatenation which forms Morse loops for $f_0$. At $\lambda-\varepsilon$, there is no trajectory from $(\tilde{x},0,0)$ to $(x',1,1)$, because this would require crossing the trajectories from $(y_1,0,0)$ to $(y'',1,1)$. However, at $\lambda$, the trajectory from $(y_1,0,0)$ to $(x',1,1)$ glues with the trajectory from $(\tilde{x},0,0)$ to $(y_1,0,0)$, which begins a one-parameter family of trajectories from $(\tilde{x},0,0)$ to $(x',1,1)$. We therefore go from $$\tilde{\phi}_{\lambda-\varepsilon}(\sigma_x)=\sigma_{\tilde{x}_1}…\sigma_{\tilde{x}_n}$$ to $$\tilde{\phi}_{\lambda+\varepsilon}(\sigma_x)=\sigma_{\tilde{x}_1}…\sigma_{\tilde{x}_n}\sigma_{x'},$$ or, in the $\alpha=\pi$ case, from $$\tilde{\phi}_{\lambda-\varepsilon}(\sigma_x)=\sigma_{\tilde{x}_1}…\sigma_{\tilde{x}_n}$$ to $$\tilde{\phi}_{\lambda+\varepsilon}(\sigma_x)=\sigma_{x'}\sigma_{\tilde{x}_1}…\sigma_{\tilde{x}_n}.$$ Therefore, if this phenomenon happens inside a Morse path, there is no change. The only case where this induces a change of the path is when this phenomenon happens at extremities of the path, that is if it is the trajectory to $(\ast_{-},1,1)$ that breaks. Suppose we have chosen an orientation so that the Morse loop starts at $\alpha=\pi$. Then this means that the step $\sigma_{\lambda}$ through $x'$, which goes from from $\ast_{+}$ to $\ast_{-}$, is added to the front of the Morse loop at $\lambda+\varepsilon$. On the other end of the loop, which corresponds to $\alpha=0$, by periodicity, we get an added step $\sigma_{\lambda}^{-1}$ at the end of the loop. 

Therefore, if $\ell$ is an $\ast$-based loop, $\ell_-$ is its image by $\phi_{\lambda-\varepsilon}$, and $\ell_+$ is its image by $\phi_{\lambda+\varepsilon}$, then $\ell_+ = \sigma_\lambda \ell_- \sigma_\lambda^{-1}$. \end{proof}

We will now prove the following theorem:
\begin{theorem}\label{thm:diag}
Let $(f_0,g_0), (f_1,g_1), (f_2,g_2)$ be three Morse-Smale pairs over $M$, $(F_{01}, G_{01})$ an interpolation pair between $(f_0,g_0)$ and $(f_1,g_1)$,  $(F_{02},G_{02})$ an interpolation function between $(f_0,g_0)$ and $(f_2,g_2)$, $(F_{12},G_{12})$ an interpolation function between $(f_1,g_1)$ and $(f_2,g_2)$, which induce morphisms $\phi_{ij}: \pi_1^\text{Morse}(f_i,\ast_j)\rightarrow \pi_1^\text{Morse}(f_j,\phi_{ij}(\ast_i))$, $i<j$, between their Morse fundamental groups. Then there exists an isomorphism $\psi$ such that the following diagram is commutative: 

\[\begin{tikzcd}[column sep=4em, row sep=3em]
\pi_1^\text{Morse}(f_0,\ast_0) \arrow[r, "\phi_{02}"] \arrow[d, "\phi_{01}"]
& \pi_1^\text{Morse}(f_2,\phi_{02}(\ast_0)) \arrow[d, "\psi"', "\simeq"] \\
\pi_1^\text{Morse}(f_1,\phi_{01}(\ast_0))  \arrow[r, "\phi_{12}"]
& \pi_1^\text{Morse}(f_2,\phi_{12}(\phi_{01}(\ast_0))).
\end{tikzcd}\]
\end{theorem}

\begin{proof}
First note that the moduli spaces of trajectories at angle $\lambda$ do not change unless crossing a singular value of $\lambda$. In view of the previous proposition, and the observation that $\phi_0=\phi_{02}$ and $\phi_{\pi/2}=\phi_{12}\circ\phi_{01}$, we must just show that, for $\varepsilon>0$, $\phi_{0}=\phi_{\varepsilon}$, and $\phi_{\pi/2-\varepsilon}=\phi_{\pi/2}$. This is simply a consequence of the fact that arriving in $\pi/2$ (resp. 0) corresponds to breaking in an index 2 critical point of $F$ in $M\times\{0\}\times\{1\}$ (resp. $M\times\{1\}\times\{0\}$). 

The previous proposition therefore readily gives us $\psi$, which is the conjugation by the $f_2$-Morse-path from $\phi_{02}(\ast_0)$ to $\phi_{12}(\phi_{01}(\ast_0))$ in the boundary of the unstable manifold of $(\ast_0,0,0)$ (see Figure \ref{fig:wu_ast_1}). \end{proof}

\begin{figure}
\labellist
\small\hair 2pt
\pinlabel {$\ast_0$} [bl] at 7 11
\pinlabel {$\phi_{01}(\ast_0)$} [bl] at -11 85
\pinlabel {$\ast_0$} [bl] at 115 11
\pinlabel {$\phi_{02}(\ast_0)$} [bl] at 130 76
%\pinlabel {$\phi_{12}(\ast_0)$} [bl] at 87 95
\pinlabel {$\phi_{12}(\phi_{01}(\ast_0))$} [bl] at 85 99
\pinlabel {$y$} [bl] at 124 95
\pinlabel {$x_1$} [bl] at 100 60
\pinlabel {$x_2$} [bl] at 69 78
\endlabellist
  \centerline{\includegraphics[width=7cm]{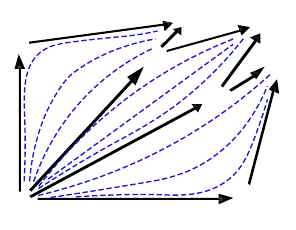}}
  \caption{The unstable manifold of $(\ast_0,0,0)$.}
  \label{fig:wu_ast_1}
\end{figure} 

\begin{cor}\label{cor:iso_morse} Let $(f_0,g_0), (f_1,g_1)$ be two Morse-Smale pairs over $M$. Let $(F_{01},G_{01})$ and $(F_{10},G_{01})$ be interpolation pairs which induce morphisms $$\phi_{01}: \pi_1^\text{Morse}(f_0,\ast_0)\rightarrow \pi_1^\text{Morse}(f_1,\phi_{01}(\ast_0))$$ and $$\phi_{10}:\pi_1^\text{Morse}(f_1,\ast_1)\rightarrow \pi_1^\text{Morse}(f_0,\phi_{10}(\ast_1))$$ between their Morse fundamental groups, for $\ast_0\in\text{Crit}_0(f_0), \ast_1\in\text{Crit}_0(f_1)$.  Then $\phi_{01}\circ\phi_{10}$ and $\phi_{10}\circ\phi_{01}$ are isomorphisms. In particular, $\phi_{01}$ and $\phi_{21}$ are isomorphisms. 
\end{cor}
\begin{proof} This follows directly by taking $f_2=f_0$, $\phi_{12}=\phi_{10}$ and $\phi_{02}=id$ in Theorem \ref{thm:diag}. \end{proof}

\subsection{The grafted case}\label{subsection:grafted_functoriality}

We can once again do this construction in the case where $f_0$, $f_1$, and $f_2$ are Morse functions over smooth closed manifolds $M_0$, $M_1$, and $M_2$. Suppose that there exist homotopies of differentiable maps $H^\tau_{01}:M_0\rightarrow M_1$ and $H^\tau_{12}:M_1\rightarrow M_2$, $\tau\in[0,1]$, such that $$(\text{Crit}(f_0)\times\text{Crit}(f_1))\cap\Gamma(H^\tau_{01})= (\text{Crit}(f_1)\times\text{Crit}(f_2))\cap\Gamma(H^\tau_{12})=\varnothing, \forall\tau\in[0,1].$$ We will also want to assume that $(\text{Crit}(f_0)\times\text{Crit}(f_2))\cap\Gamma(H^1_{12}\circ H^1_{01})=\varnothing$. We may then consider sets of grafted interpolation data $(f_0,f_1,H^\tau_{01})$ between $f_0$ and $f_1$, $(f_1, f_2,H^\tau_{12})$ between $f_1$ and $f_2$, and $(f_0,f_2,H^0_{12}\circ H^0_{01})$ between $f_0$ and $f_2$. Let $\tilde{K}_i: M_i\times[0,1]\times[0,1]\rightarrow\mathbb{R}$ be identically equal to $f_i$, and set $K_i =\tilde{K}_i  + h(s)+h(s')$, for $(s,s')\in [0,1]^2$. We assume that it is Morse-Smale relatively to $g_i$ and the standard Riemannian metric on $[0,1]\times[0,1]$. 

We will look at the moduli space $\mathcal{M}_{\rm gr^2}(x,y)$ of twice grafted trajectories between critical points $x$ of $K_0$ and critical points $y$ of $K_2$, that is triplets $(\gamma_0, \gamma_1, \gamma_2)$ of trajectories:
\begin{align*} \gamma_{0}: ]-\infty, 0]&\rightarrow M_0\times[0,1]\times[0,1],\ \gamma_{1}'(t) = -\nabla K_1(\gamma_{1}(t)), \\
\gamma_{1}:  [0,u]&\rightarrow M_1\times[0,1]\times[0,1]\ \gamma_{2}'(t) = -\nabla K_2(\gamma_{2}(t)),\\
\gamma_{2}: [u,+\infty[&\rightarrow M_2\times[0,1]\times[0,1],\ \gamma_{3}'(t) = -\nabla K_3(\gamma_{3}(t)),\end{align*}
 and such that 
\begin{gather*} \lim_{t\rightarrow-\infty}\gamma_{0}(t)\in{\rm Crit}(K_1),\\
\lim_{t\rightarrow+\infty}\gamma_{2}(t)\in{\rm Crit}(K_3),\\
\gamma_0(0)\in M_0\times[0,1]\times\{1/2\}, \gamma_1(0)=(H^{s_{0}}_{12}(m_{0}),s_{0},1/2),\end{gather*} where $\gamma_0(0)=(m_{0},s_{0},1/2)$, 
and where $u$ is such that \[\gamma_{1}(u) \in M_1\times\{s+s'=3/2\},\] and \[\gamma_2(u)=(H^{2s_u-1}_{23}(m_u),s_u,s'_u),\] where $\gamma_1(u)=(m_u,s_u,s'_u)$.
These flow lines are depicted in Figure \ref{fig:grafted_square}.

\begin{figure}
\labellist
\small\hair 2pt
\pinlabel {$F_{0}$} [bl] at 25 85
\pinlabel {$F_{0}$} [bl] at 290 85
\pinlabel {$F_{1}$} [bl] at 25 190
\pinlabel {$F_{2}$} [bl] at 215 265
\pinlabel {$F_0$} [bl] at 165 10
\pinlabel {$F_1$} [bl] at 100 265
\pinlabel {$F_2$} [bl] at 290 190
\pinlabel {$M_0$} [bl] at 165 100
\pinlabel {$M_1$} [bl] at 185 170
\pinlabel {$M_2$} [bl] at 230 210
\endlabellist
  \centerline{\includegraphics[width=8cm]{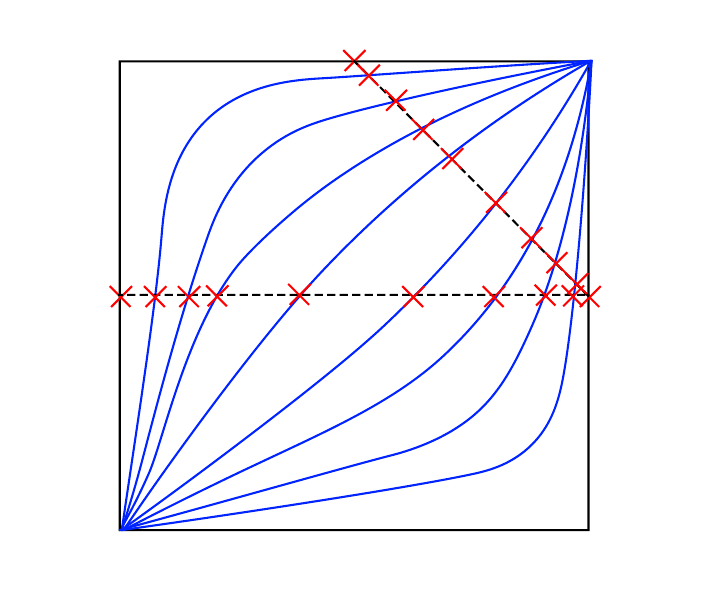}}
  \caption{Grafted trajectories between $M_0$, $M_1$, and $M_2$.}
  \label{fig:grafted_square}
\end{figure} 

Call $\Gamma(\nabla K_{1})$ the graph of the anti-gradient flow of $K_1$. It is diffeomorphic to $M_1\times[0,1]^2\times\mathbb{R}$, and therefore of dimension $\dim M_1 + 3$. Also set $\Gamma(H^s_{12}\arrowvert_{\{s'=1/2\}})=\{(x,s,1/2,H^s_{12}(x),s,1/2),x\in M_0, s\in[0,1]\}$ and $\Gamma(H^{2s-1}_{23}\arrowvert_{\{s+s'=3/2\}})=\{(x,s,s',H^{2s-1}_{23}(x),s,s'),x\in M_1,s+s'=3/2\}$. We have the following diffeomorphism:
\begin{align*} \mathcal{M}_{\rm gr^2}(x,y) \rightarrow &(W^u(x)\times \Gamma(\nabla K_{1})\times W^s(y)) \\ 
& \cap (\Gamma(H^s_{01}\arrowvert_{\{s'=1/2\}})\times\Gamma(H^{2s-1}_{12}\arrowvert_{\{s+s'=3/2\}})),\\
(\gamma_0,\gamma_1,\gamma_2) \mapsto &(\gamma_0(0),\gamma_1(0), \gamma_1(u),\gamma_2(u)).
\end{align*}
We may choose generic metrics on $M_i$ which satisfy ``grafted Morse-Smale'' conditions, that is that $\Gamma(H^s_{01}\arrowvert_{\{s'=1/2\}})$ intersects $W^u(x)\times W^s(z)$ transversally for any $x\in {\rm Crit}(K_0),z\in {\rm Crit}(K_1)$, and $\Gamma(H^{2s-1}_{12}\arrowvert_{\{s+s'=3/2\}})$ intersects $W^u(z)\times W^s(y)$ transversally for any $z\in {\rm Crit}(K_1),y\in {\rm Crit}(K_2)$. This in particular implies that $\Gamma(H^s_{01}\arrowvert_{\{s'=1/2\}})\times\Gamma(H^{2s-1}_{12}\arrowvert_{\{s+s'=3/2\}})$ intersects $W^u(x)\times \Gamma(\nabla K_{1})\times W^s(y)$ transversally for any $x\in {\rm Crit}(K_0),y\in {\rm Crit}(K_2)$. It follows that
\begin{align*}
\dim\mathcal{M}_{\rm gr^2}(x,y)=&\ {\rm ind}_{K_0}(x)+\dim (\Gamma(\nabla K_{1})) +\dim(M_2\times[0,1]^2) \\ &-{\rm ind}_{K_2}(y)+\dim \Gamma(H^s_{01}\arrowvert_{\{s'=1/2\}})+\dim  \Gamma(H^{2s-1}_{12}\arrowvert_{\{s+s'=3/2\}})\\ &-\dim(M_0\times[0,1]^2)-\dim(M_1\times[0,1]^2)\\ &-\dim(M_1\times[0,1]^2)-\dim(M_2\times[0,1]^2),\\
=&\ {\rm ind}_{K_0}(x) +\dim M_1+3 + \dim M_2 + 2 \\ &- {\rm ind}_{K_2}(y) + \dim M_0 + 1 + \dim M_1+1 \\ & - \dim M_0-2 -2(\dim M_1-2) - \dim M_2-2,\\
\dim\mathcal{M}_{\rm gr^2}(x,y)=&\ {\rm ind}_{K_0}(x)-{\rm ind}_{K_2}(y)-1.
\end{align*}

$\mathcal{M}_{\rm gr^2}(x,y)$ is compactified by adding trajectories in $M\times\{0,1\}\times\{0,1\}$ which correspond to the gluing on one end of trajectories of anti-gradient lines of $f_0+h$ (seen as an interpolation function between $f_0$ and $f_0$) and of grafted trajectories of anti-gradient lines of $f_0$ and $f_2$ with the map $H_{02}=H^1_{12}\circ H^1_{01}$, and on the other end, to gluing of grafted trajectories of anti-gradient flow lines of $f_0$ and $f_1$ and of grafted trajectories of anti-gradient flow lines of $f_1$ and $f_2$. 

As in the previous subsection, we will look at trajectories in $$\mathcal{M}_{\rm gr^2}(x)=\bigcup_{x'\in {\rm Crit}_1(K_2)}\mathcal{M}_{\rm gr^2}((x,0,0),(x',1,1)),$$ where $x\in{\rm Crit}_1(f_0)$ and $x'\in{\rm Crit}_1(f_2)$, which is a one-dimensional moduli space. It appears in the boundary of $$\widetilde{\mathcal{M}}_{\rm gr^2}(x)=\bigcup_{y\in {\rm Crit}_0(K_2)}\mathcal{M}_{\rm gr^2}((x,0,0),(y,1,1)),$$a two dimensional moduli space which can be parametrized by angles $\lambda$ and $\alpha$. Since $\lambda$ and $\alpha$ are defined at the start of the trajectories, before the graft, we may give them the same definitions as in Subsection \ref{subsection:functoriality}. See Figures \ref{fig:lambda} and \ref{fig:alpha}. Then trajectories in $\mathcal{M}_{\rm gr^2}(x)$ are uniquely determined by these angles. For generic values of $\lambda$, the restriction of $M\times[0,1]^2$ to the submanifold which is reached by trajectories which start at an angle $\lambda$ is Morse-Smale, and we may define a morphism $\phi_\lambda^{\rm gr}:\pi_1^{\text{Morse}}(f_0,\ast_0)\rightarrow\pi_1^{\text{Morse}}(f_2,\ast_3)$. The proof that this is well defined is analogous to that of Theorem \ref{thm:quotient_grafted}, with the only change being that we replace every moduli space involved in the proof with one of trajectories starting at angle $\lambda$. If we include broken trajectories, the following equalities hold: $$\phi^{\rm gr}_0 = \phi_{H_{02}}^{\rm gr}\circ id,$$and $$\phi^{\rm gr}_{\pi/2}=\phi_{H_{12}^{0}}^{\rm gr}\circ\phi_{H_{01}^{0}}^{\rm gr}.$$ 

We get an analog to Proposition \ref{prop:diag_lambda}.

\begin{prop}\label{prop:diag_lambda_gr}
For any non-generic value $\lambda\in]0,\pi/2[$, there exists a small enough $\varepsilon$ and an isomorphism $\psi_{\lambda}$ such that the following diagram is commutative:
\[\begin{tikzcd}[column sep=5em, row sep=3em]
\pi^{\text{Morse}}_1(f_0,\ast_0) \arrow[r,"\phi^{\rm gr}_{\lambda-\varepsilon}"] \arrow[rd, "\phi^{\rm gr}_{\lambda+\varepsilon}"']
& \pi^{\text{Morse}}_1(f_2,\phi^{\rm gr}_-(\ast_0)) \arrow[d,"\psi_\lambda"] \\
&  \pi^{\text{Morse}}_1(f_2,\phi^{\rm gr}_+(\ast_0)).
\end{tikzcd}\]
More precisely, $\psi_{\lambda}$ is the conjugation by the $f_2$-Morse path $\sigma_\lambda$ from $\phi^{\rm gr}_-(\ast_0)$ to $\phi^{\rm gr}_+(\ast_0)$ which appears in the boundary of the unstable manifold of $(\ast_0,0,0)$ between $\lambda-\varepsilon$ and $\lambda+\varepsilon$, ie $$\psi_\lambda(\ell_-)=\sigma_\lambda \ell_- \sigma_\lambda^{-1},$$ where $\ell_-$ is a $\phi^{\rm gr}_-(\ast_0)$ based Morse loop.
\end{prop}

\begin{proof} Every element $\mathcal{M}_{\rm gr^2}(x)$ is uniquely determined by $\alpha$ and $\lambda$, and we get the same picture as in Figure \ref{fig:angles}. In particular, we get the same type of singularities as in the non-grafted cases, and these singularities are not affected by the graft, as they have to do with breakings and gluings which only happen in the corners of the square, and birth-deaths which are not affected. We may then apply the same arguments as in the proof of Proposition \ref{prop:diag_lambda}, replacing all trajectories by grafted trajectories. The arguments for the birth-deaths and the case of $\lambda \in [0,1], \alpha\in{0,1}$ (see Figure \ref{fig:break_y}) adapt directly to the grafted case. The other arguments do not change when replacing the moduli spaces considered with those of grafted trajectories, and so the behavior is identical. \end{proof}

This proposition has several consequences. Firstly:

\begin{theorem} 
Let $(f_0,f_1,H_0)$ and $(f_0,f_1,H_1)$ be two sets of grafted interpolation data from $f_0:M_0\rightarrow\mathbb{R}$ to $f_1:M_1\rightarrow\mathbb{R}$ , and suppose $H_0$ and $H_1$ are homotopic. Let $$\phi_{H_0}^{\rm gr}:\pi_1^{\rm Morse}(f_0,\ast_0)\rightarrow\pi_1^{\rm Morse}(f_1,\phi_{H_0}(\ast_0))$$ and $$\phi_{H_1}^{\rm gr}:\pi_1^{\rm Morse}(f_1,\ast_0)\rightarrow\pi_1^{\rm Morse}(f_0,\phi_{H_1}(\ast_0))$$ be the associated grafted continuation maps between their Morse fundamental groups. Then there exists a Morse path $\ell$ from $\phi_{H_0}(\ast_0)$ to $\phi_{H_1}(\ast_0)$ such that $\phi_{H_1}^{\rm gr} = \ell \phi_{H_0}^{\rm gr}\ell^{-1}$. 
\end{theorem}
\begin{proof}
This is a direct consequence of Proposition \ref{prop:diag_lambda_gr}, where we take $M_1=M_2$, and $H_{12}^\tau=id$ for all $\tau\in[0,1]$. Then if $\dim M_1>1$, for any homotopy between between $H_0$ and $H_1$, one may find a homotopy $H^\tau$ such that at any time $\tau$, $\text{Crit}(f_0)\times\text{Crit}(f_1)=\varnothing$. If $\dim M_1=1$, we cannot in general achieve this condition, and there will be isolated moments $\tau$ at which we could have $(\text{Crit}(f_0)\times\text{Crit}(f_1))\cap\Gamma(H^\tau)\neq\varnothing$. We may perturb the metric on $M_1\times[0,1]^2$ (in a non-split way) near these points so that the moduli spaces are transversally cut out, without affecting the dynamics on the boundary of the square.  The arguments from the previous proofs still hold (in particular, most of the behavior of the grafts depends on the dynamics before the graft, which we have not changed).

Then $\phi^{\rm gr}_0$ and $\phi^{\rm gr}_{\pi/2}$ are obtained by taking the trajectories that break at $M\times\{1\}\times\{0\}$ and $M\times\{0\}\times\{1\}$, and so $\phi^{\rm gr}_0=\phi^{\rm gr}_\varepsilon$, and $\phi^{\rm gr}_{\pi/2}=\phi^{\rm gr}_{\pi/2-\varepsilon}$. The proposition tells us that $\phi_{0}^{\rm gr}=\ell \phi_{\pi/2}^{\rm gr}\ell^{-1}$, where $\ell$ is the Morse path from $\phi_{H_0}(\ast_0)$ to $\phi_{H_1}(\ast_0)$ which appears in the boundary of $\bigcup_{y\in{\rm Crit}_0(K_2)}\mathcal{M}_{\rm gr}(x,y)$. But we have $\phi_{0}^{\rm gr}=\phi_{H_{01}^{1}}^{\rm gr}\circ id$ and $\phi_{\pi/2}^{\rm gr} = id \circ \phi_{H_{01}^{0}}^{\rm gr}$.
\end{proof}

\begin{theorem}\label{thm:diag_gr}
Let $(f_0,g_0), (f_1,g_1)$, and $(f_2,g_2)$ be three Morse-Smale pairs over manifolds $M_0$, $M_1$, and $M_2$, and consider a grafted interpolation set $(f_0,f_1,H_{01})$ (resp. $(f_1,f_2,H_{12})$, $(f_0,f_2,H_{02})$) be a grafted interpolation pair between $f_0$ and $f_1$ (resp. between $f_1$ and $f_2$, between $f_0$ and $f_2$), which induces a morphism $\phi^{\rm gr}_{12}$ (resp. $\phi^{\rm gr}_{23}$, $\phi^{\rm gr}_{13}$) between the corresponding Morse fundamental groups. Then there exists an isomorphism $\psi$ such that the following diagram is commutative: 

\[\begin{tikzcd}[column sep=4em, row sep=3em]
\pi_1^\text{Morse}(f_0,\ast_0) \arrow[r, "\phi^{\rm gr}_{02}"] \arrow[d, "\phi^{\rm gr}_{01}"]
& \pi_1^\text{Morse}(f_2,\phi^{\rm gr}_{02}(\ast_0)) \arrow[d, "\psi"', "\simeq"] \\
\pi_1^\text{Morse}(f_1,\phi^{\rm gr}_{01}(\ast_0))  \arrow[r, "\phi^{\rm gr}_{12}"]
& \pi_1^\text{Morse}(f_2,\phi^{\rm gr}_{12}(\phi^{\rm gr}_{01}((\ast_0))).
\end{tikzcd}\]
\end{theorem}
\begin{proof} This theorem is an immediate consequence of Proposition \ref{prop:diag_lambda}, and of the equalities $\phi^{\rm gr}_0=\phi^{\rm gr}_\varepsilon$, $\phi^{\rm gr}_{\pi/2}=\phi^{\rm gr}_{\pi/2-\varepsilon}$, and $\phi^{\rm gr}_0 = \phi_{H_{02}^{1}}^{\rm gr}\circ id,$ and $\phi^{\rm gr}_{\pi/2}=\phi_{H_{12}^{1/2}}^{\rm gr}\circ\phi_{H_{01}^{0}}^{\rm gr}.$ In this case, the homotopies are constant, and we get $$\phi_{H_{12}}^{\rm gr}\circ\phi_{H_{01}}^{\rm gr}= \phi^{\rm gr}_{\pi/2} =\phi^{\rm gr}_0\circ\psi=\phi_{H_{02}}^{\rm gr}\circ \psi,$$ which, by taking $\phi^{\rm gr}_{01}=\phi^{\rm gr}_{H_{01}}$, $\phi^{\rm gr}_{12}=\phi^{\rm gr}_{H_{12}}$, and $\phi^{\rm gr}_{02}=\phi^{\rm gr}_{H_{02}}$, gives us the commutativity of the diagram. \end{proof}

\begin{cor}
Let $(f_0,g_0)$ be a Morse-Smale pair over $M_0$, $(f_1,g_1)$ be a Morse-Smale pair over $M_1$. We suppose that there exists a diffeomorphism $H:M_0\rightarrow M_1$. Let $(f_{0},f_{1}, H)$ and $(f_{1},f_{0}, H^{-1})$ be sets of grafted interpolation data which induce morphisms $\phi^{\rm gr}_{01}: \pi_1^\text{Morse}(f_0,\ast_0)\rightarrow \pi_1^\text{Morse}(f_1,\ast_1)$, $\phi^{\rm gr}_{10}:\pi_1^\text{Morse}(f_1,\ast_1)\rightarrow \pi_1^\text{Morse}(f_0,\ast_0)$ between their Morse fundamental groups. Then $\phi^{\rm gr}_{01}\circ\phi^{\rm gr}_{10}=id$ and $\phi^{\rm gr}_{10}\circ\phi^{\rm gr}_{01}=id$. In particular, $\phi^{\rm gr}_{01}$ and $\phi^{\rm gr}_{10}$ are isomorphisms. 
\end{cor}
\begin{proof} If $H$ is a diffeomorphism, then Theorem \ref{thm:diag_gr} tells us that the morphism $\phi^{\rm gr}_{01}$ induced by the set of grafted interpolation data $(f_0,f_1,H)$ is an isomorphism and its inverse is the morphism $\phi^{\rm gr}_{10}$ induced by the set of grafted interpolation data $(f_1,f_0, H^{-1})$.
\end{proof}

\section{A relative Morse fundamental group over $M\times\mathbb{R}$}\label{section:relative}

In this last section, we build an algebraic invariant based on the geometric setting that is used to build our continuation maps, i.e. Morse-Smale pairs on $M\times\mathbb{R}$ such that the restriction to certain ``slices'' are Morse-Smale pairs on $M$, and such that these latter Morse-Smale pairs provide all the critical points. We will show that properties of such an invariant can allow us to deduce information about the continuation maps between the Morse fundamental groups. Since $M\times\mathbb{R}$ is non-compact, this will take the form of a \textit{relative} Morse fundamental group over $M\times\mathbb{R}$. 

We will look at what we call \textit{interpolation-type} functions:
\begin{definition}\label{setting} An \textit{interpolation-type} function over a manifold $M$ is a Morse function $F:M\times\mathbb{R}\rightarrow \mathbb{R}$, such that $F$ can be written as $F=\tilde{F}+Ch$, for some $C>0$, where:
\begin{itemize}\item $h:\mathbb{R}\rightarrow\mathbb{R}$ is a  Morse function such that $\lim_{|s|\rightarrow+\infty}|h(s)|=+\infty$, with critical points $p_0,p_1,…,p_k\in\mathbb{R}$. \item $\tilde{F}:M\times\mathbb{R}\rightarrow\mathbb{R}$ is a function such that, for some disjoint intervals $I_1,…I_{k-1}$ that contain respectively $p_1, …, p_{k-1}$, $\tilde{F}\arrowvert_{M\times I_i}=f_i$, where $f_1,…,f_{k-1}:M\rightarrow\mathbb{R}$ are Morse functions on $M$. We also assume that, for some small enough $\varepsilon$, $\tilde{F}\arrowvert_{M\times]-\infty,p_0+\varepsilon]}\equiv f_0$ and $\tilde{F}\arrowvert_{M\times[p_{k}-\varepsilon,+\infty[}\equiv f_{k}$, where $f_0$ and $f_{k}$ are Morse functions over $M$.
\item $C$ is big enough so that $$|\frac{\partial \tilde{F}}{\partial s}(x,s)|<C|h'(s)|,  \forall (x,s)\in M\times(\mathbb{R}\setminus\{p_0,…,p_k\}).$$ 
\end{itemize}
\end{definition}

We choose a metric $G$ on $M\times\mathbb{R}$ for which $(F,G)$ is Morse-Smale, and such that its restriction to $M\times\{p_i\}$ is Morse-Smale relatively to $f_i$, and its restriction to $]-\infty,p_0+\varepsilon]\times M$ and $[p_{k}-\varepsilon,+\infty[\times M$ is identically equal to metrics $g_0$ and $g_k$ on $M$ which are Morse-Smale relatively to $f_0$ and $f_{k}$. We call such a pair an \textit{interpolation-type pair}.

Call $\mathcal{I}_M$ the subset of $\{0…,k\}$ such that $p_i$ is a local maximum for every $i\in\mathcal{I}_M$, and $\mathcal{I}_m$ the subset of $\{0,…,k\}$ such that $p_i$ is a local minimum for every $j\in\mathcal{I}_m$. Note that of course, two consecutive elements of $\{0,…,k\}$ cannot be in the same subset $\mathcal{I}_M$ or $\mathcal{I}_m$. Then critical points of $F$ are the collection of critical points of $f_0,…,f_k$, with the following index shift:
\[ {\rm Crit}_j(F)=\bigcup_{i\in\mathcal{I}_M} {\rm Crit}_{j-1}(f_i)\times\{p_i\} \cup \bigcup_{i\in\mathcal{I}_m} {\rm Crit}_{j}(f_i)\times\{p_i\}.\]

We will consider the Latour cells of these critical points. For critical points in $M\times\{p_1,…,p_{k-1}\}$, these behave like ordinary Latour cells. However, if $\lim_{s\rightarrow -\infty} h(s)=-\infty$, the Latour cells of critical points of $f_0$ have an open end at $-\infty$. Since the dynamics are split on $M\times]-\infty,p_0]$, we have a diffeomorphism
$$\overline{\mathcal{M}}((x,p_0),M\times]-\infty, p_0])\simeq \overline{\mathcal{M}}(x,M)\times]-\infty,p_0].$$ We may then compactify them by formally adding $\mathcal{M}(x,M)\times\{-\infty\}$. We may do the same for critical points in $M\times\{p_k\}$ if $\lim_{s\rightarrow +\infty} h(s)=-\infty$ by adding the Latour cells at $+\infty$. We call such Latour cells \textit{compactified Latour cells}.

Keeping this compactification as well as the index shift in mind, we give the following definition:
\begin{definition} A \textbf{Morse step} of a interpolation-type function $F$ is one of the following:
\begin{itemize} \item the (compactified) Latour cell of an index 1 critical point of $F$, with a choice of orientation;
\item if $\lim_{s\rightarrow -\infty} h(s)=-\infty$, the pair $(\overline{\mathcal{M}}(x,M),-\infty)$, where $\overline{\mathcal{M}}(x,M)$ is Latour cell of an index 1 critical point of $f_0$ (as a Morse function on $M$), with a choice of orientation;
\item if $\lim_{s\rightarrow +\infty} h(s)=-\infty$, the pair $(\overline{\mathcal{M}}(x,M),+\infty)$, where $\overline{\mathcal{M}}(x,M)$ is Latour cell of an index 1 critical point of $f_{k}$ (as a Morse function on $M$), with a choice of orientation.
\end{itemize}We say that the Morse step passes \textit{through} the index 1 critical point. \end{definition} 

There are two types of Morse steps through index 1 critical points of $F$. The first passes through a critical point in $\cup_{i\in\mathcal{I}_m}{\rm Crit}_1(f_i)\times\{p_i\}$. These steps are made of trajectories that move only in the $M$ direction, and correspond to Morse steps of the Morse function $f_i$ on $M$ at $p_i$. The second ones pass through critical points in $\cup_{i\in\mathcal{I}_M}{\rm Crit}_0(f_i)\times\{p_i\}$. The trajectories forming these steps do travel along the $\mathbb{R}$ coordinate (in fact, near $t\rightarrow-\infty$, they are constant in the $M$ coordinate). Suppose the step is through a point in ${\rm Crit}_0(f_i)\times\{p_i\}$, for $i\in\mathcal{I}_M$, and suppose $i\neq0,k$. Then the step ends at points in ${\rm Crit}_0(f_{i\pm1})\times\{p_{i\pm1}\}$. If $i=0$ (resp. $k$), then the restriction of the trajectory to $M$ is constant at a minimum $m_0$ or $m_{k}$ of $f_0$ or $f_{k}$, while the $\mathbb{R}$-coordinate diverges to $\pm\infty$: we say that the limit is $(m_0,-\infty)$  (resp. $(m_{k},+\infty)$) when $t\rightarrow+\infty$ (this is a formal notation). The compactification of the Latour cell adds a formal point at infinity of the form $(\gamma_{m_0},-\infty)$ or $(\gamma_{m_k},+\infty)$. Since we see steps through index 1 critical points of $f_0$ or $f_{k}$ as placed as infinity, we will also see those steps as ending in points of the form $(\gamma_{m_0},-\infty)$ or $(\gamma_{m_k},+\infty)$, and consider the points $m_0$ and $m_k$ to be points at $\pm\infty$.

\begin{remark} We may also see the trajectories which diverge towards points of the form $(m_0,-\infty)$ or $(m_{k},+\infty)$ as trajectories which intersect $M\times\{d_0\}$ or $M\times\{d_{k}\}$, for constants $d_0$ and $d_{k+1}$ that are respectively smaller, or bigger, than the smallest, or biggest, critical point of $h$. In this sense, we could also look at them as trajectories which enter a certain sublevel set $F^{-1}(c)$, where $c$ is smaller than any critical value of $F$. \end{remark} 

Let us choose $\ast \in {\rm Crit}_0(F)\cup\widetilde{{\rm Crit}}_0(f_0)\cup\widetilde{{\rm Crit}}_0(f_k)$, where $\widetilde{{\rm Crit}}_i(f_{0})$ (resp. $\widetilde{{\rm Crit}}_i(f_k)$) is equal to ${\rm Crit}_i(f_{0})\times\{-\infty\}$ (resp. ${\rm Crit}_i(f_k)\times\{+\infty\}$) if $\lim_{s\rightarrow -\infty} h(s)=-\infty$ (resp. $\lim_{s\rightarrow +\infty} h(s)=-\infty$), and is empty otherwise.

\begin{definition} We say that a Morse step \textit{begins} (resp. \textit{ends}) at a point $$ (y,s) \in \widetilde{{\rm Crit}}_0(f_0)\cup{\rm Crit}_0(F)\cup\widetilde{{\rm Crit}}_0(f_k)$$ if the beginning (resp. end) of the Latour cell (with regards to our chosen orientation on the step) is of the form $(\gamma,\gamma_{(y,s)})$ if $s\neq \pm\infty$, or of the form $(\gamma_y,s)$ if $s=\pm\infty$. Two Morse steps are said to be \textit{consecutive} if the end of the first one is the beginning of the second one. A Morse \textit{path} for $F$ is a word of the form $\sigma_1…\sigma_r$, where $\sigma_i$ is a Morse step for every $i\in\{1,…,r\}$, and $\sigma_i$ and $\sigma_{i+1}$ are consecutive steps. We say that it is \textit{based} at $\ast$ if $\sigma_1$ begins in $\ast$.  \end{definition}

We call $\tilde{\mathcal{L}}(F,\ast)$ the set of all Morse paths based at $\ast$ and which either end in $\ast$ (in this case we may call them Morse \textit{loops}), or in a point of $\widetilde{{\rm Crit}}_0(f_0)\cup\widetilde{{\rm Crit}}_0(f_k)$, adding to the set the empty word $1$ which designates the constant path at $\ast$. This set does not have a group structure, as in general, we may not concatenate paths, nor invert them. However, in certain cases, we may: if $\ell = \sigma_1…\sigma_k$ is a Morse loop, then $\ell^{-1}=\sigma_k^{-1}…\sigma_1^{-1}$ is as well. Furthermore, if $\ell'$ is another Morse loop, then $\ell\ell'\in\tilde{\mathcal{L}}(F,\ast)$. 

We define an equivalence relation on $\tilde{\mathcal{L}}(F,\ast)$ in the following manner: $ \ell \sim \ell' $ if
\begin{enumerate}[{(1)}] \item\label{relation:inverse} $\ell = \sigma_1…\sigma_{k-1}\sigma_k\sigma_k^{-1}\sigma_{k+2}…\sigma_r$ and $\ell'=\sigma_1…\sigma_{k-1}\sigma_{k+2}…\sigma_r$, where $\sigma_k^{-1}$ is the step $\sigma_k$ with the opposite orientation. 

\item\label{relation:boundary} $\ell = \sigma_1…\sigma_{i-1}\sigma_i\sigma_{i+1}…\sigma_r$ and $\ell'=\sigma_1…\sigma_{i-1}(\sigma'_1…\sigma'_p)\sigma_{i+1}…\sigma_r,$ where $\sigma_i^{-1}\sigma'_1…\sigma'_p$ form the sequence of index 1 Latour cells in the oriented boundary of the (compactified) Latour cell of some $x\in{\rm Crit}_{2}(F)\cup\widetilde{{\rm Crit}}_2(f_0)\cup\widetilde{{\rm Crit}}_2(f_{k})$.

Note that this implies that if there exists $x$ such that $\partial W^{u}(x)=\ell$, then $\ell$ is equivalent to the empty word. 

\item\label{relation:infty} $\ell=\sigma_1\sigma_2…\sigma_p$ and $\ell'=\sigma_1\sigma_2…\sigma_{p-1}$, where $\sigma_p$ is a Morse step through a critical point in $\widetilde{{\rm Crit}}_1(f_0)\cup\widetilde{{\rm Crit}}_1(f_{k})$. This case essentially tells us that \textit{where} we land at $\pm\infty$ does not matter, what matters is only whether or not we land at $\pm\infty$, as in the case where this applies, the word ends in a step through an index 0 critical point of $f_0$ or $f_{k}$ which goes to $\pm\infty$.

\end{enumerate}
We then set $\pi_1^{\text{Morse}}(F,\ast)=\faktor{\tilde{\mathcal{L}}(F,\ast)}{\sim}$. 

\begin{remark} These equivalences have several consequences on the inverted and concatenated paths described previously. Firstly, if $\ell$ is a loop in $\ast$, then $\ell\ell^{-1}\sim1$. Also, if $\ell\sim\ell'$, then $\ell^{-1}\sim\ell'^{-1}$. Finally, given two Morse loops $\ell_1$ and $\ell_2$ such that $\ell_1\sim\ell_2$, then if we have another path $\ell_3$, then $\ell_1\ell_3\sim\ell_2\ell_3$. \end{remark}
\begin{remark} A more visual way of defining relation \ref{relation:boundary} would be to say that two paths are equivalent if they form the boundary of the Latour cell of an index 2 critical point, or if together they form the boundary of the union of several such Latour cells (modulo the first relation). In the typical Morse fundamental group setting, since we have a group structure, this is achieved by  modding out by the normal subgroup generated by such boundaries. Here, since we do not have a group structure and so the notion of a generated subgroup no longer makes sense, we replace this notion with the one of equivalences on steps. Figure \ref{fig:gluing_unstable_mflds} illustrates how this is an equivalent approach: in this case, we have three unstable manifolds which glue together. The first one tells us that $\sigma_\infty\sim\sigma_3^{-1}\sigma_2^{-1}\sigma_1^{-1}$, the second one gives us $\sigma_4\sim\sigma_6^{-1}\sigma_5^{-1}\sigma_2^{-1}$, and the last $\sigma_\infty'\sim\sigma_7^{-1}\sigma_6^{-1}\sigma_8^{-1}$. One can check that this gives us two equivalence classes of paths based at $\ast$, $[\sigma_1\sigma_4^{-1}\sigma_7]$ and the empty word.
 \end{remark}

\begin{figure}
  \centerline{\includegraphics[width=11cm]{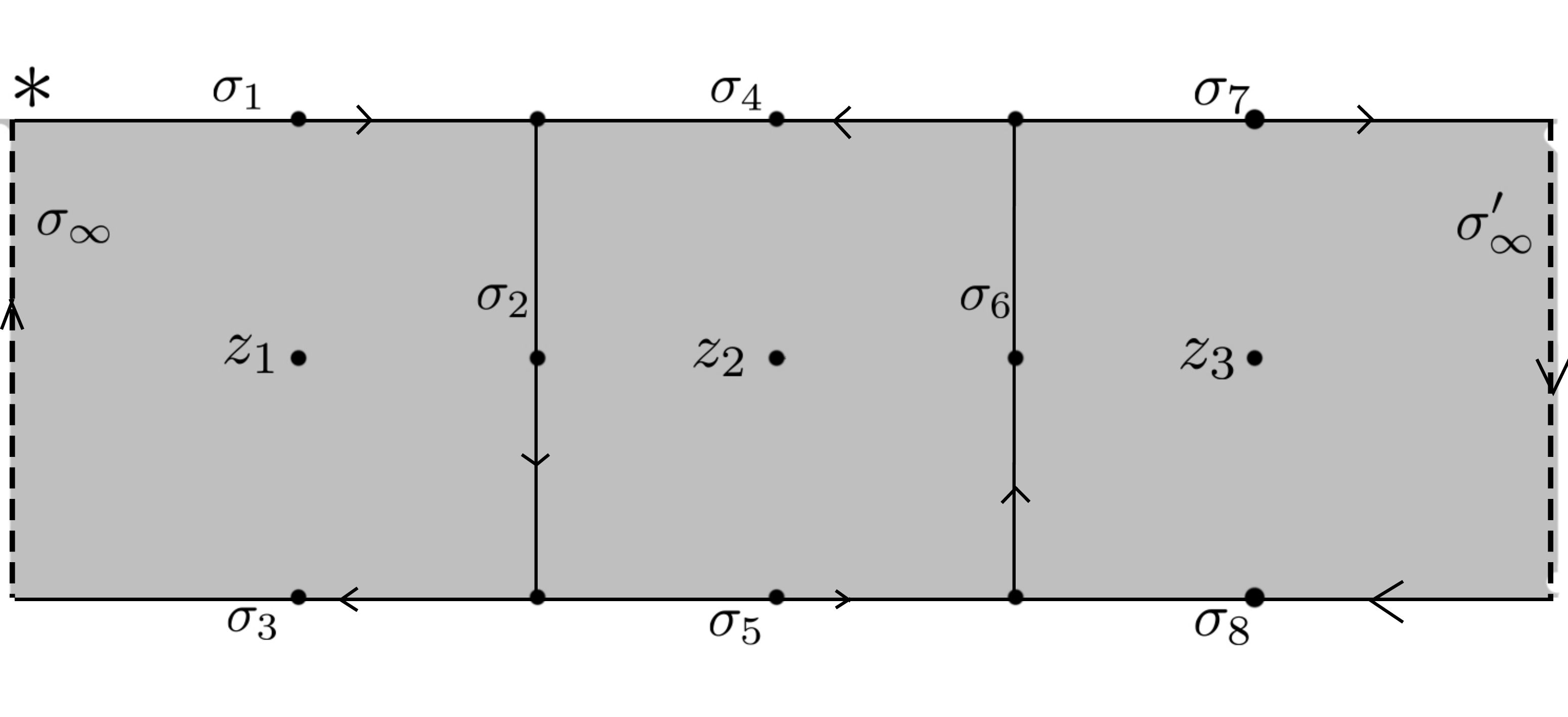}}
  \caption{Union of three disks, two of which have "ends" at $-\infty$. We have two equivalence classes of paths based at $\ast$, $[\sigma_1\sigma_4^{-1}\sigma_7]$ and $[1]$.}
  \label{fig:gluing_unstable_mflds}
\end{figure}

\begin{example}\label{ex:relative} Here are a few examples: 
\begin{enumerate} \item Let's look at what the Morse fundamental group looks like for the case described in Section \ref{section:continuation_map}, where $h(s)= s^3-\frac{3}{2}s^2$. Its limit at $-\infty$ is $-\infty$, and its limit at $+\infty$ is $+\infty$. We suppose $\tilde{F}\equiv f$, where $f:M\rightarrow \mathbb{R}$ is a Morse function. Then the anti-gradient trajectories of $F$ are determined by $f$ on $M$ coordinates, and by $h$ on the $\mathbb{R}$ coordinate. We suppose $f$ has a unique minimum $m_\ast$. We choose $(m_\ast,-\infty)$ as our base point. Our non-trivial Morse paths are concatenations of paths of the form $\beta_1…\beta_{r'}\sigma \alpha_1…\alpha_r\sigma^{-1}$, where $\alpha_i$ are Morse steps through index 1 critical points $(z_i,1)$ of $F$, $\sigma$ is the unique Morse step through $(m_\ast,1)$ which goes from $(m_\ast,-\infty)$ to $(m_\ast,1)$, and $\beta_i$ are Morse steps through $(z_i,-\infty)$ (we may take $r$ or $r'=0$).

For relations, we have relations which come from Morse relations of $f$ at $s=1$. We also have a relation for each $(z_i,0)$, which has an unstable manifold with the following oriented boundary: $\alpha_i\sigma^{-1}\beta_i\sigma$. This tells us that the path $\alpha_i\sigma^{-1}\beta_i\sigma\sim 1$, and $\sigma \alpha_i\sigma^{-1}\sim \beta_i^{-1} \sim 1$, where the last equivalence is per rule (\ref{relation:infty}). One can check that in fact, all paths are trivial, and so, $\pi_1^\text{Morse}(F,(m_\ast,-\infty))=\{1\}$. 

If we choose $(m_\ast,1)$ as our base point, the paths we have are different. They are concatenations of paths of the form $\alpha_1…\alpha_r\sigma^{-1}$ or $\alpha_1…\alpha_r\sigma^{-1}\beta_1…\beta_{r'}\sigma$. The relations of the previous paragraph give us: $\alpha_i\sim\sigma^{-1}\beta_i\sigma$, $\alpha_i\sigma^{-1}\sim\sigma^{-1} \beta_i^{-1}\sim\sigma^{-1}$. We also have the usual Morse relations of $f$ on the $x_i$s. All in all, we get $\pi_1^\text{Morse}(F, (m_\ast,1))=\pi_1^{\text{Morse}}(f, m_\ast)\cup \{[\sigma^{-1}]\}$. In particular, this is not a group, although it has a subset which is.

\item Suppose $h(s)=s^2$. We have a unique critical point at $s=0$, which happens to be a minimum, and its limits at $\pm \infty$ are $+\infty$ (in particular, it is decreasing before $s=0$, and increasing after). Therefore $F$ does not have any critical points outside of $\{s=0\}$, and no trajectory starting at $\{s=0\}$ can leave this zone. Morse steps based at a minimum $\ast=(m_\ast,0)$ of $F$ are therefore only $f$-steps, where $f$ is the Morse function in $\{s=0\}\times M$, and the relations are also only the Morse relations of this Morse function. Therefore, $\pi_1^{\text{Morse}}(F,\ast)=\pi_1^{\text{Morse}}(f,m_\ast)$.
\item Suppose $h(s)=-s^2$. Then it has a unique critical point at $s=0$, which is a maximum, and its limits at $\pm \infty$ are $-\infty$ (in particular, it is increasing before $s=0$, and decreasing after). As in the first example, we suppose $\tilde{F}\equiv f$, where $f:M\rightarrow \mathbb{R}$ is a Morse function which has a unique minimum $\ast$. Therefore, $F$'s only critical points are in $\{s=0\}$, so correspond to critical points of $f$, the Morse function on $M$ at $s=0$, but with an index shift of one. Since we have no minimum of $h$, we must take $(m_\ast,+\infty)$ or $(m_\ast,-\infty)$ as our base point. Let us choose $(m_\ast,-\infty)$. Therefore, the non trivial Morse paths are concatenations of paths of the form $\alpha_1…\alpha_r\sigma \beta_1…\beta_{r'}\sigma^{-1}$ or $\alpha_1…\alpha_r\sigma \beta_1…\beta_{r'}\sigma^{-1}$, where $\sigma$ is the unique step through $(m_\ast,0)$ and goes from $-\infty$ to $+\infty$, $\alpha_1…\alpha_r$ are $f$-steps at $-\infty$, and $\beta_1…\beta_r$ are $f$-steps at $+\infty$.

For every Morse step $\alpha$ at $-\infty$ through a critical point $(z,-\infty)$, there is a unique step $\beta$ at $+\infty$ through $(z,+\infty)$, as well as a unique relation thanks to $(z,0)$: the oriented boundary of the Latour cell of $(z,0)$ gives us the sequence $\alpha\sigma \beta\sigma^{-1}$. This in particular, tells us that this Morse path is equivalent to the empty word, that the Morse path $\sigma \beta \sigma^{-1}\sim \alpha^{-1} \sim 1$ (this last equivalence, we get from rule (\ref{relation:infty})), and $\alpha\sigma\sim\sigma \beta^{-1}\sim\sigma$. And so, the only non trivial Morse path is $\sigma$, and we get $\pi_1^\text{Morse}(F,(m_\ast,-\infty))=\{[\sigma],1\}$. \end{enumerate} \end{example}

Before stating our theorem, we recall the following definition of a relative fundamental group (see \cite{H} for more details):
\begin{definition} Let $(X,A)$ be a pair of spaces, with $A\subset X$, and choose a base point $\ast\in A$. Then $\pi_1(X,A,\ast)$ is the pointed set of homotopy classes of paths in $X$ which start in $\ast$ and end in $A$.\end{definition}
\begin{remark} Note that, despite its name, the relative fundamental group is not a group, but rather a pointed set. It is also in general not isomorphic to the quotient of $\pi_1(X,\ast)$ by $\pi_1(A,\ast)$.\end{remark}

Before stating our main theorem, we note that since $f_0$ and $f_k$ are Morse functions on $M$ which is compact, and $c$ is smaller than any critical value of $F$, then if $\lim_{s\rightarrow-\infty} h(s)=-\infty$ (resp. $\lim _{s\rightarrow+\infty} h(s)=+\infty$), we may choose a constant $d^{+}_c$ such that $M\times\{d^-_c\}\subset F^{-1}(]-\infty,c])$ (resp. $d^{+}_c$ such that $M\times\{d^-_c\}\subset F^{-1}(]-\infty,c])$).

\begin{theorem}Let $M$ be a closed smooth manifold, and $(F,G)$ an interpolation-type pair over $M$. We choose a base point $\ast \in {\rm Crit}_0(F)\cup\widetilde{{\rm Crit}}_0(f_0)\cup\widetilde{{\rm Crit}}_0(f_k)$. Let $c$ be a constant which is inferior to $F$'s smallest critical value, and $\ast'$ such that \[\ast'=\begin{cases} \ast &\text{ if } \ast\in{\rm Crit}_0(F),\\(m_\ast,d^-_c) &\text{ if } \ast =(m_\ast,-\infty)\in\widetilde{\text{Crit}}(f_0) ,\\ (m_\ast,d^+_c)&\text{ if } \ast =  (m_\ast,+\infty)\in\widetilde{\text{Crit}}(f_k) \end{cases}\]  Then, if $\pi_1^\text{Morse}(F,\ast)$ is the Morse fundamental group of $F$ with base point $\ast$, and $\pi_1(\mathbb{R}\times M, F^{-1}(]-\infty,c])\cup\{\ast'\},\ast')$ is the fundamental group of $\mathbb{R}\times M$ relatively to $F^{-1}(]-\infty,c])\cup\{\ast\}$ and with base point $\ast'$, there exists an isomorphism of pointed spaces \[\phi:\pi_1^\text{Morse}(F,\ast)\xrightarrow{\sim} \pi_1(\mathbb{R}\times M, F^{-1}(]-\infty,c])\cup\{\ast'\},\ast').\]   \end{theorem}

\begin{proof}
Firstly, notice we may see $M\times\{d_c^-;d_c^+\}$ as two copies of $M$ on which $f_0$ and $f_k$ are Morse functions, and if we restrict $(F,G)$ to these ``slices'' we get a Morse-Smale pair by construction. Therefore, $f_0$-steps and $f_k$-steps can be seen as Morse steps within $M\times\{d_c^-;d_c^+\}$.

We define a morphism \[ \tilde\phi: \tilde{\mathcal{L}}(F,\ast)\rightarrow\pi_1(\mathbb{R}\times M, F^{-1}(]-\infty,c])\cup\{\ast\},\ast)\] in the following manner:
\begin{itemize}
\item For Morse steps entirely contained in $M\times]p_0,p_{k+1}[$, we simply evaluate along the steps, which gives us topological paths traced by the steps.
\item For Morse steps through critical points of $f_0$ or $f_{k}$ and which go to $\pm\infty$, notice that the trajectories forming them which go to $\pm\infty$ must cross $M\times\{d^-_c;d^+_{c}\}$ in exactly one point, on the way to the limit at $t\rightarrow+\infty$ which is of the form $(m,\pm\infty)$. We take the topological path traced by the trajectory, but which stops at that point in $M\times\{d^-_c;d^+_{c}\}$. Since the dynamics are split, this trajectory is constant in the $M$ coordinate, and cutting it off means stopping it at a point in $\text{Crit}_0(f_0)\times\{d^-_c;d^+_{c}\}$. We then take the path which is traced by this cut-off trajectory.
\item For Morse steps "at $\pm\infty$", that is Morse steps through critical points of $f_0$ or $f_k$, we take the topological paths which trace the corresponding Morse steps in $M\times\{d^-_c;d^+_{c}\}$. 
\end{itemize}
We then glue these paths together, which gives us a well defined topological path, and take the homotopy class of such a path to be the image of the morphism. 

Let us show that this induces a morphism $\phi$ from $\pi_1^\text{Morse}(F,\ast)$ to $\pi_1(\mathbb{R}\times M, F^{-1}(]-\infty,c])\cup\{\ast\},\ast)$: take two Morse loops which are equivalent in $\pi_1^\text{Morse}(F,\ast)$. Then the steps that form them are equivalent, meaning that one is obtained from the other by a series of applications of rules (\ref{relation:inverse}),(\ref{relation:boundary}), or (\ref{relation:infty}). Let us check that each one induces a homotopy in $\pi_1(\mathbb{R}\times M, F^{-1}(]-\infty,c])\cup\{\ast\},\ast)$. For (\ref{relation:inverse}), this amounts to traveling along a portion of a path and then traveling back. One can therefore contract the path, and this defines a homotopy. For (\ref{relation:boundary}), this means we replace part of the path with another path which is on the other side of the union of several disks (the images of the Latour cell by the evaluation map, which give us disks with perhaps immersed and not embedded boundaries formed of the evaluation of Morse steps, which gives us homotopies between their evaluations). For Latour cells of critical points of $f_0$ or $f_k$, it is the same, but the disks are in $M\times\{d^-_c;d^+_{c}\}$. We may define a homotopy through the disks. Finally, for (\ref{relation:infty}), since for paths in $\pi_1(\mathbb{R}\times M, F^{-1}(]-\infty,c])\cup\{\ast\},\ast)$, it does not matter where the path ends as long as it ends in $F^{-1}(]-\infty,c])\cup\{\ast\}$, and one can contract along the final part of the path, which is a homotopy. Therefore, these loops are equivalent in $\pi_1(\mathbb{R}\times M, F^{-1}(]-\infty,c])\cup\{\ast\},\ast)$, and $\phi$ is well defined on the quotient.

Let us show that this morphism is surjective. The strategy is the following: we take a topological path which represents a class in $\pi_1(\mathbb{R}\times M, F^{-1}(]-\infty,c])\cup\{\ast\},\ast)$, and show that it is always homotopic to an element of the image of $\phi$. We define this homotopy by first taking any component of the path that is in  $M\times\mathbb{R}\setminus]d^-_c;d^+_{c}[$ and contracting them onto paths in $M\times\{d^-_c;d^+_{c}\}$, which we may so as these two are homotopy equivalent. We then apply the flow of the following vector field:
\[ X = \nabla \tilde{F} + \rho \nabla h,\] where $\rho$ is a cutoff function that is $1$ on a neighborhood of $M\times[p_0,p_{k}]$, and $0$ outside of $M\times]d^-_c;d^+_{c}[$. In particular, in $M\times[p_0,p_{k}]$, this matches the anti-gradient flow of $F$. In $M\times\{d^-_c;d^+_{c}\}$, this matches the flow of $f_0$ and $f_k$ due to the split nature of the dynamics (up to a change in speed). Moreover, the flow does not push any points outside of $M•\times[d^-_c;d^+_{c}]$. We may assume that, before applying the flow, the path intersects all stable manifolds transversally, and the portions in $M\times\{d^-_c;d^+_{c}\}$ intersect the stable manifolds of critical points of $f_0$ and $f_k$ transversally as well. Therefore, the dimension of the intersection between the path and such a stable manifold is $1-\text{ind}_F(x)$ for critical points of $F$, and $1-\text{ind}_{f_0}(x)$ (resp.$1-\text{ind}_{f_k}(x))$ for critical points of $f_0$ (resp. $f_k$). In particular, it is only non empty if these indexes are smaller than 1. The flow therefore pushes down the path until it is caught in tubular neighborhoods of unstable manifolds of index 1 critical points (or 0 critical points), and we may then project onto these unstable manifolds. The path we obtain is an element $\tilde{\mathcal{L}(F,\ast)}$ which maps by $\phi$ to a path homotopic to the original one. 

Let us now show that the morphism is injective. Take two Morse paths $\ell_1$ and $\ell_2$ which have the same image by $\phi$, ie that are sent onto homotopic paths, which by definition of $\phi$, are contained in $M\times[p_0,p_{k+1}]$. We may suppose that $\ell_1$ and $\ell_2$ end in the same point by applying rule (\ref{relation:infty}) to one of them, which does not change the homotopy type of the image. Since the paths are homotopic, and have the same ends, they must bound a (not necessarily embedded) disk. We may perturb the disk away from its boundary so that it is transverse to the stable manifolds of critical points of $F$, and by the same argument as before, it therefore only intersects the unstable manifolds of dimension 2 or less. Then we flow down the disk by the vector field of the previous paragraph, until it is caught in tubular neighborhoods of unstable manifolds of dimension 2 or 1, and then project. Unstable manifolds of index 2 critical points correspond to relations between $\ell_1$ and $\ell_2$ as in (\ref{relation:boundary}), and unstable manifolds of index 1 critical points correspond to relations as in rule (\ref{relation:inverse}). We may therefore deduce that $\ell_1$ and $\ell_2$ are in the same equivalence class of $\pi_1^\text{Morse}(F,\ast)$. \end{proof}

\begin{cor} For an interpolation-type pair $(F,G)$, if $h$ has the same limit at $+\infty$ and $-\infty$, then $\pi_1^\text{Morse}(F,\ast)$ takes only the following values:
\[ \pi_1^\text{Morse}(F,\ast)=\begin{cases} \pi_1(M,\rho_M(\ast)), & \text{if $\lim_{|s|\rightarrow+\infty} h(s)=+\infty$,}\\
\{[\sigma],1\}, &\text{if $\lim_{|s|\rightarrow+\infty} h(s)=-\infty$.}
\end{cases}\]
where $\rho_M(\ast)$ is the projection of $\ast$ onto $M$, and $\sigma$ is a step which has one end at $-\infty$ and another at $+\infty$. Moreover, if $h$ has limits with opposite signs at $\pm\infty$, then if $\ast\in\text{Crit}_0(F)$, then $\pi_1^\text{Morse}(F,\ast)=\pi_1(M,\rho(\ast))\cup\{\sigma'\}$, where $\sigma'$ is a step which goes begins at $\ast$ and ends at $-\infty$. If not, then $\pi_1^\text{Morse}(F,\ast)=\{1\}$.
\end{cor}
\begin{proof} Per the previous theorem, $\pi_1^\text{Morse}(F,\ast)\cong \pi_1(\mathbb{R}\times M, F^{-1}(]-\infty,c])\cup\{\ast'\},\ast')$ for some constant $c$ smaller than all critical values of $F$. The set on the right side depends only on the topology of $M$, the topology of $F^{-1}(]-\infty,c])\cup\{\ast'\}$, and if $\ast\in F^{-1}(]-\infty,c])$,  on the connected component it belongs to. However, the topology of $F^{-1}(]-\infty,c])\cup\{\ast'\}$ depends only on the nature of $h$ at $\pm\infty$ and on if $\ast\in F^{-1}(]-\infty,c])$ (as well as the topology of $M$, which is already counted for). Furthermore, if $\ast_0$ is in one of the connected components of $F^{-1}(]-\infty,c])$ and $\ast_1$ is in another, then the path that goes along the $\mathbb{R}$ coordinate of $M\times\mathbb{R}$ from one connected component to another allows us to change base point from $\ast_0$ to $\ast_1$ and defines an isomorphism between $\pi_1(\mathbb{R}\times M, F^{-1}(]-\infty,c]),\ast_0)$ and $\pi_1(\mathbb{R}\times M, F^{-1}(]-\infty,c]),\ast_1)$. Hence, the value of $\pi_1^\text{Morse}(F,\ast)$ depends only on the topology of $M$, the behavior at infinity of $h$, and whether or not the base point is chosen at infinity. The only cases are therefore the ones discussed in Example \ref{ex:relative}.\end{proof}

In the case where $h(s)= s^3-\frac{3}{2}s^2$, as in Section \ref{section:continuation_map}, and the base point is at $-\infty$, $\pi_1(\mathbb{R}\times M, F^{-1}(]-\infty,c]),\ast')=\{1\}$ because $\mathbb{R}\times M$ retracts onto $F^{-1}(]-\infty,c])$. Therefore, $\pi_1^{\text{Morse}}(F,\ast)=\{1\}$ for any choice of $f_0$ and $f_1$. In particular, this implies that any loop of the form $\sigma \ell\sigma'^{-1}$, where $\sigma$ and $\sigma'$ are steps through minimums of $f_0$ and $\ell$ is a Morse loop for $f_1$, is trivial, regardless of if the loop $\ell$ is trivial in $\pi^{\text{Morse}}_1(f_1,\ast_1)$ for some minimum $\ast_1$ of $f_1$. It must therefore be "filled" by the unstable manifold of an index 1 critical point of $f_0$. This provides an alternative proof of the surjectivity of the morphism $\phi_{01}$ in Corollary \ref{cor:iso_morse}.

\end{document}